\documentclass{gtpart}

\usepackage{graphicx}
\usepackage{amsfonts}
\usepackage{epsf}
\usepackage{amssymb}
\usepackage{amsmath}
\usepackage{amscd}
\usepackage{tikz}
\usepackage{pdfpages}
\usepackage{fancyhdr}
\usepackage{setspace}
\usepackage{hyperref}
\usepackage[all]{xy}
\usepackage{verbatim}
\usepackage{enumerate}
\usepackage{rotating}
\usepackage[text={6in,8.5in},centering,letterpaper,dvips]{geometry}
\usepackage{pinlabel}

\title{Small Seifert fibered surgery on hyperbolic pretzel knots}

\author{Jeffrey Meier} 
\givenname{Jeffrey}
\surname{Meier}
\address{Department of Mathematics\\University of Texas at Austin\\\newline
         Austin, TX 78712\\USA}
\email{jmeier@math.utexas.edu}
\urladdr{http://www.ma.utexas.edu/users/jmeier/}



\arxivreference{arXiv:submit/0582129 [math.GT] 29 Oct 2012}  
\arxivpassword{}   

\volumenumber{}
\issuenumber{}
\publicationyear{}
\papernumber{}
\startpage{}
\endpage{}
\doi{}
\MR{}
\Zbl{}
\received{}
\revised{}
\accepted{}
\published{}
\publishedonline{} 
\proposed{}
\seconded{}
\corresponding{}
\editor{}
\version{}

\theoremstyle{definition}
\newtheorem{theorem}{Theorem}[section]
\newtheorem{proposition}[theorem]{Proposition}
\newtheorem{lemma}[theorem]{Lemma}
\newtheorem{definition}[theorem]{Definition}
\newtheorem{question}[theorem]{Question}
\newtheorem{criterion}[theorem]{Criterion}
\newtheorem{corollary}[theorem]{Corollary}
\newtheorem{conjecture}[theorem]{Conjecture}

\newtheorem*{acknowledgements}{Acknowledgements}
\newtheorem*{organization}{Organization}

\makeatletter
\newtheorem*{rep@theorem}{\rep@title}
\newcommand{\newreptheorem}[2]{%
\newenvironment{rep#1}[1]{%
 \def\rep@title{#2 \ref{##1}}%
 \begin{rep@theorem}}%
 {\end{rep@theorem}}}
\makeatother

\newreptheorem{theorem}{Theorem}
\newreptheorem{lemma}{Lemma}
\newreptheorem{question}{Question}

\numberwithin{figure}{section}

\newcommand{\Kh}{\text{Kh}}

\begin{document}

\begin{abstract}

We complete the classification of hyperbolic pretzel knots admitting Seifert fibered surgeries.  This is the final step in understanding all exceptional surgeries on hyperbolic pretzel knots.  We also present results toward similar classifications for non-pretzel Montesinos knots of length three.

\end{abstract}

\maketitle


\setcounter{section}{0}


\section{Introduction}

The study of exceptional surgery on hyperbolic knots has been well developed over the last quarter century.  One particularly well studied problem is that of exceptional surgery on arborescent knots, which include Montesinos knots and pretzel knots.  Thanks to the positive solution to the Geometrization conjecture \cite{perelman2, perelman3, perelman1}, any exceptional surgery is either reducible, toroidal, or a small Seifert fibered space.  Exceptional surgeries on hyperbolic arborescent knots of length 4 or greater have been classified \cite{wu_large}, as have exceptional surgeries on hyperbolic 2-bridge knots \cite{britt-wu:2bridge}.   It has been shown that no hyperbolic arborescent knot admits a reducible surgery \cite{wu_arbor}, and toroidal surgeries on hyperbolic arborescent knots of length three are completely classified \cite{wu_toroidal}.

Therefore, it only remains to understand small Seifert fibered surgeries on Montesinos knots of length three.  Furthermore, finite surgeries on Montesinos knots only occur in two instances, along two slopes of each of the pretzel knots $P(-2,3,7)$ and $P(-2,3,9)$  \cite{futer-etal_finite,ichi-jong_finite}.  Thus,  one must only consider non-finite, atoroidal Seifert fibered surgeries on hyperbolic Montesinos knots of length three.

According to Wu \cite{wu_plbs,wu_imm}, the only hyperbolic Montesinos knots of length three that are pretzel knots and might admit Seifert fibered surgeries have the form $P(q_1,q_2,q_3)$ or $P(q_1,q_2,q_3,-1)$, where $(|q_1|, |q_2|, |q_3|) = (2, |q_2|, |q_3|)$, $(3,3,|q_3|)$, or  $(3, 4, 5)$, and in the length four case, then $q_i>0$ for $i=1,2,3$.  Recently, it was shown that hyperbolic pretzel knots of the form $P(p,q,q)$ with $p,q$ positive \cite{ichi-jong_pqq} or $P(-2,p,p)$ with $p$ positive \cite{ichi-jong-kabaya_2pp} do not admit Seifert fibered surgeries.

Further work by Wu \cite{wu_plbs,wu_large,wu_imm} tells us that if a non-pretzel Montesinos knot admits a small Seifert fibered surgery, then it has one of the following forms:  $K[1/3, -2/3, 2/5]$, $K[-1/2, 1/3, 2/(2a+1)]$ for $a\in\{3,4,5,6\}$, or $K[-1/2, 1/(2q+1), 2/5]$ for $q\geq 1$.

In this paper, we address the issue of which of the above listed Montesinos knots admit small Seifert fibered surgeries.  The main results are stated below.  Keep in mind that there is an orientation reversing homeomorphism $K(\alpha) = \overline{K}(-\alpha)$, where $\overline{K}$ is the mirror of $K$.  Thus, we often consider in our analysis, and present in our results, only one representative of $\{K, \overline K\}$.

For the following theorem, recall that the pretzel knot $P(p,q,r)$ with $|p|,|q|,|r|\geq2$ is hyperbolic unless it is either $P(-2,3,3)$ or $P(-2,3,5)$, in which case it is the torus knot $T(3,4)$ or $T(3,5)$, respectively (\cite{oertel}).  Below, when we consider the knots $P(-2,2p+1,2q+1)$, we will assume that $|p|<|q|$ when $p$ and $q$ have the same sign and that $p>0$ when their signs differ.

\begin{theorem}\label{thm:main1}
The hyperbolic pretzel knot $P(-2, 2p+1, 2q+1)$, with the conventions discussed above, admits a small Seifert fibered surgery if and only if $p=1$, in which case it admits precisely the following small Seifert fibered surgeries:
\begin{itemize}
\item $P(-2,3,2q+1)(4q+6) = S^2(1/2,-1/4,2/(2q-5))$
\item $P(-2,3,2q+1)(4q+7) = S^2(2/3,-2/5,1/(q-2))$
\end{itemize}
\end{theorem}

\begin{theorem}\label{thm:main2}
Hyperbolic pretzel knots of the form $P(3,3,m)$ or $P(3,3,2m,-1)$ admit no small Seifert fibered surgeries.  Pretzel knots of the form $P(3,-3,m)$, with $m>1$, admit small Seifert fibered surgeries precisely in the following cases:
\begin{itemize}
\item $P(3,-3, 2)(1) = S^2(1/3,1/4,-3/5)$
\item $P(3,-3, 3)(1) = S^2(1/2,-1/5,-2/7)$
\item $P(3,-3, 4)(1) = S^2(-1/2,1/5,2/7)$
\item $P(3,-3, 5)(1) = S^2(2/3,-1/4,-2/5)$
\item $P(3,-3, 6)(1) = S^2(1/2,-2/3,2/13)$
\end{itemize}
\end{theorem}

\begin{theorem}\label{thm:main3}
The pretzel knots $P(3,\pm4, \pm5)$ and  $P(3,4,5,-1)$ admit no small Seifert fibered surgeries.
\end{theorem}

\begin{theorem}\label{thm:main4}
Suppose that $K$ is a non-pretzel Montesinos knot and $K(\alpha)$ is a small Seifert fibered space.  Then either $K=K[-1/2,2/5, 1/(2q+1)]$ for some $q\geq 5$, or $K$ is on the following list and has the described surgeries.
\begin{itemize}
\item $K[1/3, -2/3,2/5](-5) = S^2(1/4,2/5,-3/5)$
\item $K[-1/2, 1/3, 2/7](-1) = S^2(1/3, 1/4, -4/7)$
\item $K[-1/2, 1/3, 2/7](0) = S^2(1/2, 3/10,-4/5)$
\item $K[-1/2, 1/3, 2/7](1) = S^2(1/2, 1/3, -16/19)$
\item $K[-1/2, 1/3, 2/9](2) = S^2(1/2, -1/3, -3/20)$
\item $K[-1/2, 1/3, 2/9](3) = S^2(1/2, -1/5, -3/11)$
\item $K[-1/2, 1/3, 2/9](4) = S^2(-1/4, 2/3, -3/8)$
\item $K[-1/2, 1/3, 2/11](-2) = S^2(-2/3, 2/5, 2/7)$
\item $K[-1/2, 1/3, 2/11](-1) = S^2(-1/2, -2/7, 2/9)$
\item $K[-1/2, 1/3, 2/5](3) = S^2(1/2, -1/3, -2/15)$
\item $K[-1/2, 1/3, 2/5](4) = S^2(1/2, -1/6, -2/7)$
\item $K[-1/2, 1/3, 2/5](5) = S^2(-1/3, -1/5, 3/5)$
\item $K[-1/2, 1/5, 2/5](7) = S^2(1/2, -1/5, -2/9)$
\item $K[-1/2, 1/5, 2/5](8) = S^2(-1/4, 3/4, -2/5)$
\item $K[-1/2, 1/7, 2/5](11) = S^2(-1/3, 3/4, -2/7)$
\end{itemize}
\end{theorem}

Each of the theorems stated above is proved below using a common procedure.  First, we exploit the symmetries of the Montesinos knots in question to describe the surgery space as a branched double cover of a link.  Next, we use rational tangle filling theory and exceptional surgery bounds to restrict our attention to a finite list of such links, i.e, we restrict the parameters for which the Montesinos knots in question can admit small Seifert fibered surgeries.  Finally, we use knot theory invariants to show that the branched double covers of links on this finite list cannot be Seifert fibered (excepting, of course, the cases that are).  This last step makes use of the Mathematica\textsuperscript{\textregistered} package KnotTheory` \cite{wolfram}.

It should be noted that, concurrent with the preparation of this paper, the author learned that similar results had been obtained by Wu, though using different techniques.  Wu also restricts the families to finite families of surgery spaces, but does so by studying exceptional surgery on \emph{tubed} Montesinos knots (see \cite{wu:wrapped}).  He then appeals to the computer program \emph{Snappex}, to determine the hyperbolic structure of the surgeries in question (see \cite{wu:SFSMontKnots}).  

\begin{organization}
Section \ref{section:preliminaries} presents general background material and outlines how knot invariants will be used to obstruct small Seifert fibered surgeries.  Sections \ref{section:2pq}, \ref{section:33q}, \ref{section:345}, and \ref{section:non-pretzel} present, respectively, the proofs of Theorems \ref{thm:main1}, \ref{thm:main2}, \ref{thm:main3}, and \ref{thm:main4}.
\end{organization}

\begin{acknowledgements}
The author would like to thank his advisor, Cameron Gordon, for many conversations full of insightful advice, patience, and encouragement.  The author is also grateful to Ying-Qing Wu for his friendly correspondence regarding the nature of this problem, and to the referee for his or her careful reading of the manuscript and helpful comments.  This work was supported by NSF grant number DMS-0636643.
\end{acknowledgements}

\subsection{A word on non-integral surgeries}

In a survey by Wu \cite{wu_survey}, it is shown how techniques and results from \cite{britt_exceptional,wu_sutured} can be combined with work of Delman \cite{delman_esslam} to study which  length three Montesinos knots have exteriors that admit persistent essential laminations.

\begin{theorem}\label{thm:non-integral}
Let $K$ be a hyperbolic Montesinos knot of length three.  Then the exterior of $K$ admits a persistent essential lamination, and, thus, cannot admit a non-integral small Seifert fibered surgery, unless $K=K[x,1/p,1/q]$ (or its mirror image), where $x\in\{-1/(2n), -1\pm1/(2n), -2+1/(2n)\}$, and $p,q,$ and $n$ are positive integers.
\end{theorem}

With this in mind, for many of the families of pretzel knots considered in this paper, it is only necessary to consider integral surgeries.  However, for some families, it is necessary to consider non-integral surgeries.  To be specific, of all the pretzel knots considered in this paper, only the following families could potentially admit non-integral small Seifert fibered surgeries:

\begin{itemize}
\item $P(-2,2p+1, 2q+1)$ with $1\leq p<q$
\item $P(3,3,-2m)$ with $m\geq 2$
\item $K[-1;1/3,1/3,1/2m]$ with $m\geq 1$
\item $P(3,-4,5)$ or $P(3,4,5,-1)$
\end{itemize}

Thus, whenever such a family is considered, we have shown that, in fact, there are no non-integral small Seifert fibered surgeries.  One of the biggest open problems in the study of exceptional Dehn surgery is the following conjecture (see \cite{gordon_survey}).

\begin{conjecture}\label{conj:gordon}
Any Seifert fibered surgery on a hyperbolic knot is integral.
\end{conjecture}

The results of this paper are the final steps of an affirmative answer to Conjecture \ref{conj:gordon} in the case of hyperbolic arborescent knots.

\begin{theorem}
Any Seifert fibered surgery on a hyperbolic arborescent knot is integral.
\end{theorem}

\subsection{Open questions}

Unfortunately, the techniques of this paper are insufficient to complete the classification of Seifert fibered surgery on Montesinos knots.  We are left with the following question, which is the final step in a complete classification of exceptional surgery on arborescent knots.

\begin{question}\label{question1}
Do the Montesinos knots $K[-1/2, 2/5, 1/(2q+1)]$ with $q\geq 5$ admit small Seifert fibered surgeries?
\end{question}

\section{Preliminaries}\label{section:preliminaries}

\subsection{Dehn surgery}

Let $K$ be a knot in $S^3$, and let $N(K)$ be a regular neighborhood of $K$.  Let $M_K = \overline{S^3\backslash N(K)}$ be the exterior of $K$.  The set of isotopy classes of simple closed curves on $\partial N(K) = \partial M_K$ is in bijection with $H_1(\partial M_K)$, the latter of which is naturally generated by two elements: $[\mu]$ and $[\lambda]$, where $[\mu]$ generates $H_1(M_K)\cong \Z$, $[\lambda]=0\in H_1(M_K)$, and $\mu$ and $\lambda$ intersect geometrically once on $\partial M_K$.  Orient $\mu$ and $\lambda$ so that $\mu\cdot\lambda=+1$.  The unoriented isotopy class of a simple closed curve $\gamma\subset \partial M_K$ is called a \emph{slope} and can be thought of as an element $m/l\in\Q\cup\{\infty\}$, where $[\gamma] =m[\mu]+l[\lambda]$ in $H_1(\partial M_K)$.  The curves $\mu$ and $\lambda$ are called the \emph{meridian} and the \emph{longitude}, respectively.

Given two slopes $\alpha$ and $\beta$ on $T^2$, let the \emph{distance} between $\alpha$ and $\beta$, $\Delta(\alpha, \beta)$ be their minimal geometric intersection number. If $\alpha = m/l$ and $\beta=m'/l'$, then we have $\Delta(\alpha,\beta) = |ml'-m'l|$.

Let $V$ be a solid torus, and let $\varphi:\partial V\to\partial M_K$ be a homeomorphism which takes the meridian of $V$ to a slope $\alpha$ on $\partial M_K$.  Then \emph{Dehn surgery on $K$ along $\alpha$}, or $\alpha$-\emph{Dehn surgery on $K$}, is the space $K(\alpha) = M_K\cup\varphi V$.  See Figure \ref{figure:Dehnsurgery}.  For a general overview of the theory of Dehn surgery, a subject that has been well-studied since its introduction by Dehn in 1910 \cite{dehn}, see \cite{gordon_survey}.

\begin{figure}\label{figure:Dehnsurgery}
\centering
\includegraphics[scale = .7]{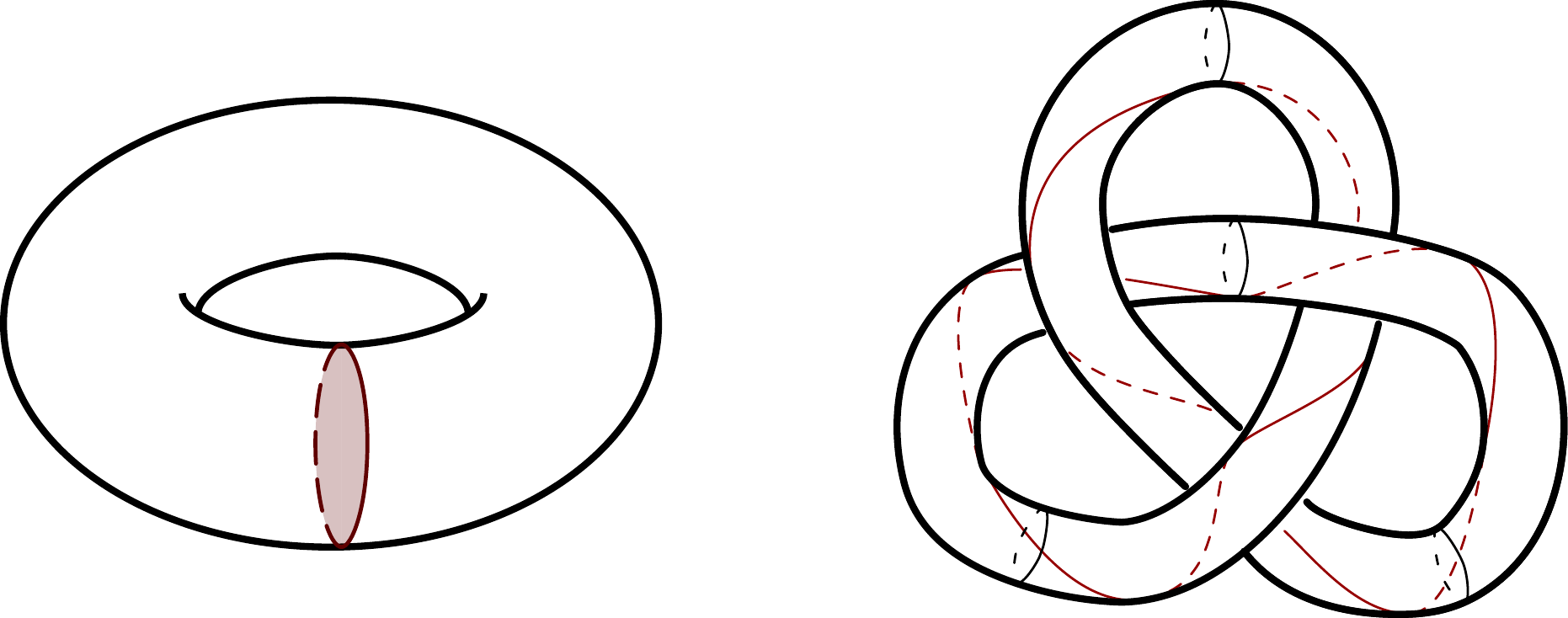}
\put(-310,0){$\star\times D^2$}
\put(-205,65){\LARGE$\stackrel{\varphi}{\longrightarrow}$}
\caption{On the right, we see the exterior, $M_K$, of the left-handed trefoil.  The surgery space $K(0)$ is formed by filling the boundary of $M_K$ with a solid torus such that the meridian maps to a 0-slope (a longitude of $K$) on $\partial M_K$.}
\end{figure}

Dehn surgery generalizes nicely to manifolds $M$ with a torus boundary component $T\subset\partial M$, where $M$ may not be the complement of knot in $S^3$.  Let $\alpha\subset T$ be a slope, then \emph{$\alpha$-Dehn filling of $M$ on $T$} is the space $M(\alpha) = M\cup_\varphi V$, where $\varphi:\partial V\to T$ sends the meridian of $V$ onto $\alpha$.  One difference in this scenario is that there may be no canonical way to distinguish a longitude on $T$, however, $\Delta(\alpha,\beta)$ is still well-defined for any pair of slopes, $\alpha$ and $\beta$.

\subsection{Cable spaces}

Let $V$ be a solid torus, and let $J$ be a $(p,q)$-curve inside $V$ (see Figure \ref{figure:CableAndCrossSection}).  The \emph{cable space}, $C(p,q)$, is the space formed by removing a regular neighborhood of $J$.  Let $T_1 =\partial V$ and $T_0 = \partial N(J)$.  There is a properly embedded annulus, $A$, connecting the two boundary components such that $A\cap T_1$ is a $p/q$-curve (in terms of the standard meridian and longitude on $V$) and  $\gamma =A\cap T_0$ is a $pq/1$-curve (see Figure \ref{figure:CableAndCrossSection}).  Let $\mu$ and $\lambda$ be some choice of meridian and longitude for $T_0$.  Then the slope $\gamma$ is called the \emph{cabling slope} for $C(p,q)$.

\begin{figure}\label{figure:CableAndCrossSection}
\centering
\includegraphics[scale = .9]{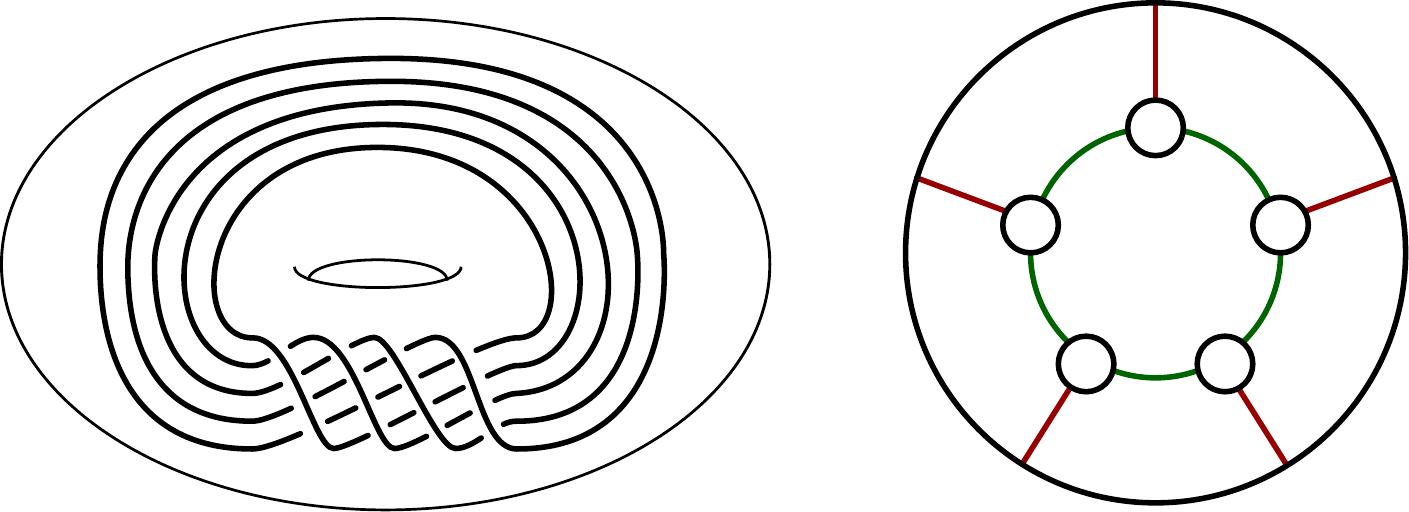}
\put(-117,86){\Large$A$}
\put(-57,78){\Large$A'$}
\caption{On the left, we see a $(4,5)$-curve $J$ inside a solid torus, $V$, and, on the right, we see a cross section of $V-N(J)$, along with two interesting annuli, $A$ and $A'$.}
\end{figure}

Let $A'$ be a properly embedded annulus such that $A'\cap T_0$ is two $pq/1$-curves, parallel to each other and to $\gamma$ (see Figure \ref{figure:CableAndCrossSection}).  Now, let $C(p,q)(\alpha)$ denote $\alpha$-Dehn filling on $T_0$.  Then, if $\alpha = \gamma$, this filling has the effect of capping off $A'$ to form a separating 2-sphere, $S$, and capping off one boundary component of $A$ to form a disk, $D$, which intersects $T_1$ in a $p/q$-curve.  The result is that $C(p,q)(\gamma) = (S^1\times D^2)\#L(q,p)$.

Let $t:C(p,q)\to C(p,q)$ represent Dehn twisting along $A$.  Then, $t^l(\mu) = \mu + l(pq\mu + \lambda)) = (lpq + 1)\mu + l\lambda$.  Since $C(p,q)(\mu) = S^1\times D^2$, it follows that $C(p,q)(t^l(\mu)) = S^1\times D^2$.  So, slopes of the form $(lpq+1)/l$ all correspond to surgery slopes on $T_0$ that yield solid tori.

This shows that cable spaces have infinitely many fillings returning solid tori, all at distance one from the cabling slope.

On the other hand, we have the following lemma, which follows from the Cyclic Surgery Theorem \cite{cgls} and work of Gabai \cite{gabai}.  See \cite{kang} for a proof and more general discussion.

\begin{lemma}\label{lemma:cable}
\begin{enumerate}[(a)]
\item Let $M\not=T^2\times I$ be an irreducible and $\partial$-irreducible 3-manifold with a torus boundary component, $T_0$.  Let $T_1$ be an incompressible torus in $M$, distinct from $T_0$.  If $\alpha$ and $\beta$ are slopes on $T_0\subset \partial M$ with $\Delta(\alpha, \beta)\geq 2$, such that $T_1$ is compressible in  $M(\alpha)$ and $M(\beta)$, then $M$ is a cable space with cabling slope $\gamma$ such that $\Delta(\alpha, \gamma) = \Delta(\beta,\gamma) = 1$.
\item Let $M\not=T^2\times I$ be an irreducible and $\partial$-irreducible 3-manifold with a torus boundary component, $T_0$.  Let $T_1$ be an incompressible torus in $M$, distinct from $T_0$.  If $\alpha$ and $\beta$ are slopes on $T_0\subset \partial M$ with $\Delta(\alpha, \beta)=1$, such that $T_1$ is compressible in  $M(\alpha)$ and $M(\beta)$, then either
\begin{enumerate}
\item $M$ is a cable space with cabling slope $\alpha$ or $\beta$, or
\item $M$ is the exterior of a braid in a solid torus, $M(\alpha)$ and $M(\beta)$ are solid tori, and $\Delta(\eta_\alpha,\eta_\beta)\geq 4$, where $\eta_\alpha$ and $\eta_\beta$ are the induced slopes of the meridian on $T_0$.
\end{enumerate}
\end{enumerate}
\end{lemma}

\subsection{Seifert fibered spaces}

A \emph{fibered solid torus} is formed by gluing the ends of $D^2\times I$ together with a twist $\rho$ through $\frac{2\pi p}{q}$, where $q\geq1$ and $p$ and $q$ are relatively prime.  There are two types of \emph{fibers}:  the \emph{central fiber}, i.e., the image of $(0,0)\times I$ after gluing, and the union of the arcs $x\times I, \rho(x)\times I, \cdots, \rho^{q-1}(x)\times I$, for $x\not=(0,0)$. 

A  \emph{Seifert fibered space} is a 3-manifold that can be decomposed as a disjoint union of circles (called \emph{fibers}), where each fiber has a regular neighborhood homeomorphic to a fibered solid torus, i.e., the fiber becomes the central fiber of the fibered solid torus.  Viewing the neighborhood this way, if $q=1$, we say the fiber is \emph{ordinary}.  If $q\geq2$, we say the fiber is \emph{exceptional with multiplicity} $q$.  In the latter case, the fibers surrounding the central fiber are called $(p,q)$-\emph{curves}.  

If $M$ is a Seifert fibered space, there is a natural projection $\pi:M\to\Sigma$ that identifies each fiber to a point.  The surface $\Sigma$ is called the base space.  We can record the exceptional fiber information in the form of cone points on $\Sigma$, so $M$ is a circle bundle over the resulting orbifold.  Another way to recover $M$ is to remove a disk neighborhood of each cone point on $\Sigma$ and cross the resulting surface with $S^1$.  The result is a manifold with torus boundary components.  If we choose meridian and longitude coordinates for each boundary component so that the projection of the meridians to the base surfaces is one-to-one onto the boundary of the removed disks and the longitude is $\star\times S^1$ in the circle product, then $M$ is the result of Dehn filling on the boundary components along the slopes $p'/q$, where $pp'\equiv 1\pmod q$.  If $M$ is a Seifert fibered space with base space $\Sigma$ and $n$ exceptional fibers with fibered solid torus neighborhoods consisting of $(p_i/q_i)$-curves for $i=1,2,\ldots, n$, we write $M=\Sigma(p'_1/q_1, \ldots, p'_n/q_n)$, or sometimes $M=\Sigma(q_1,\ldots, q_n)$.  In fact, the homeomorphism type of $M$ is determined by $\Sigma$ and the \emph{Seifert invariants}: $\{p'_1/q_1, \ldots, p'_n/q_n\}$, up to permutation, and up to the relation $\{p'_1/q_1, p'_2/q_2, \ldots, p'_n/q_n\} = \{p'_1/q_1\pm 1, p'_2/q_2\mp 1, \ldots, p'_n/q_n\}$.  In other words, $\sum_{i=1}^np'_i/q_i$ is an invariant of $M$.  Because of this, it is often useful to standardize the notation so that the Seifert invariants are all positive and less than one.  To do this, we subtract out the integer part of each fraction and collect them in a single term, $b$.  We write $M=\Sigma(b; p'_1/q_1, \ldots, p'_n/q_n)$, where $0<p'_i<q_i$ and $b\in \Z$.

A Seifert fibered space is called \emph{small} if the base space is a sphere and the number of exceptional fibers is at most three.

Next, we recall a fact about Dehn filling on Seifert fibered manifolds that will be useful  throughout this paper.  Let $M$ be a Seifert fibered manifold with a torus boundary component $T\subset \partial M$.  The fibering of $M$ induces a fibering of $T$, and the slope, $\gamma$, of the induced fibers on $T$ is called the \emph{Seifert slope} of $T$.  Now, the Seifert fibering of $M$ will extend to a Seifert fibering of $\alpha$-Dehn filling on $M$ provided that $\alpha\not=\gamma$.  In fact, we have the following.  See \cite{Heil} for a complete treatment of Dehn filling on Seifert fibered spaces with boundary.

\begin{lemma}\label{SFfillings}
If $M$ is a Seifert fibered manifold with base surface $\Sigma$ and $n$ exceptional fibers, $T\subset \partial M$ is a torus boundary component  (corresponding to a circle boundary component $C\subset\partial \Sigma$), and $\gamma$ is the Seifert slope $T$, then let $M(\alpha)$ denote $\alpha$-Dehn filling on $T$, let $d=\Delta(\alpha,\gamma)$, and let $\hat\Sigma = \Sigma\cup_CD^2$. Then,
\begin{enumerate}[(a)]
\item If $d\geq 2$, $M(\alpha)$ is a Seifert fibered space with base surface $\hat\Sigma$ and $n+1$ exceptional fibers (the original exceptional fibers, plus a new one of multiplicity $d$).
\item If $d=1$, $M(\alpha)$ is a Seifert fibered space with base surface $\hat\Sigma$ and (the original) $n$ exceptional fibers.
\item If $d=0$, $M(\alpha) = N\#L$, where $N$ is a Seifert fibered space with base surface $\hat\Sigma$ and (the original) $n$ exceptional fibers, and $L$ is a Lens space.
\end{enumerate}
\end{lemma}

As an example, consider $D^2(a,b)$ with Seifert slope $r/s$, and let $d=\Delta(m/l,r/s)= |ms-lr|$.  Then (as developed in \cite{gordon_survey}),

$$D^2(a,b)(m/l) = \begin{cases} 
S^2(a,b,d) & \text{ if } d\geq 2 \\
L(m, lb^2) & \text{ if } d=1 \\
L(a,b)\#L(b,a) & \text{ if } d=0
\end{cases}$$

\subsection{Montesinos knots}

A \emph{tangle} is a pair $(B,A)$, where $B\cong B^3$ and $A$ is a pair of properly embedded arcs in $B$.  A \emph{marked tangle} is a tangle along with an identification of its boundary $\partial(B,A) = (S, S\cap A)$, which is a 2-sphere with 4 distinguished points, with the pair $(S^2, \{NE, NW, SW, SE\})$.  The \emph{trivial tangle} is the tangle which is homeomorphic as a marked tangle to $(D^2,\{2\  points\})\times I$. Let $h$ and $r$ be the tangle operations where $h$ adds a positive horizontal half-twist (right-handed), and $r$ is reflection in the $(NW/SE)$-plane.

 Let $[c_1, c_2,\ldots, c_m]$ be a sequence of integers, and let $p/q = \frac{1}{c_1 + \frac{1}{c_2 + \frac{1}{\cdots + c_m}}}$.  The \emph{rational tangle}, $\mathcal R(p/q)$ is formed by applying the operation $(h^{c_m}r)(h^{c_{m-1}}r)\cdots(h^{c_1}r)$ to the trivial tangle, which we denote $\mathcal R(1/0)$.  Note that, as an unmarked tangle, $\mathcal R(p/q)$ is trivial, one can just untwist it.  On the other hand, Conway showed \cite{conway} that, as marked tangles, $\mathcal R(p/q)=\mathcal R(p'/q')$ if and only if $p/q=p'/q'$.
 
 A \emph{Montesinos link of length $n$} is a link formed by connecting $n$ rational tangles to each other in a standard fashion.  We denote such a knot by $K[p_1/q_1, \ldots, p_n/q_n]$ (see Figure \ref{figure:MontesinosLinks}).  In the special case where each $p_i = \pm1$, we have what is called a \emph{pretzel knot}.  In this case, each tangle is just a strand of vertical twists, since $1/q$ has the continued fraction expansion $[q]$.  It is easy to see that Montesinos links of length one or two are the same.  These links are called \emph{$2$-bridge links}, and will be denoted $K[p/q]$, where $p/q$ is the rational number describing the tangle twists.

\begin{figure}\label{figure:MontesinosLinks}
\centering
\includegraphics[scale = .9]{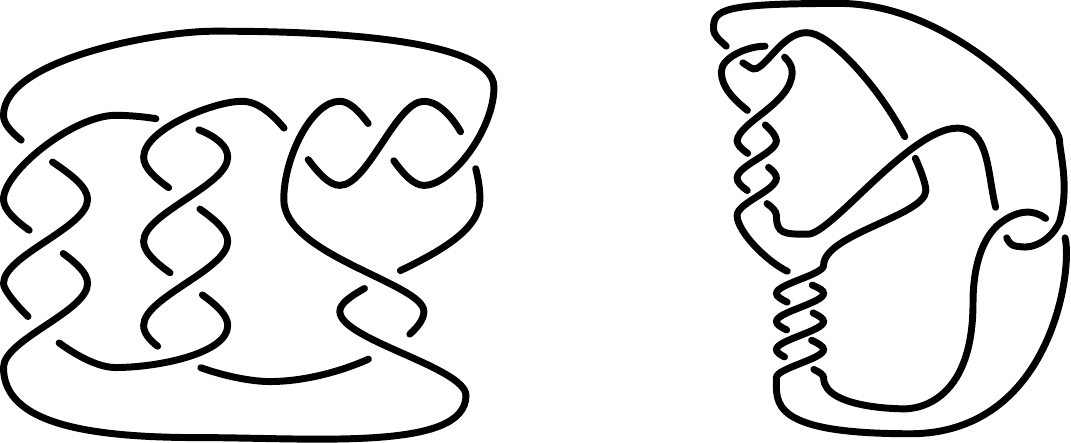}
\caption{Above we have the Montesinos knot $K[1/3, 1/4, -3/5]$ and the 2-bridge knot $K[43/95]$ (with continued fraction $[2,4,1,3,2]$).}
\label{figure:MontesinosLinks}
\end{figure}

Montesinos links of length three are determined up to the same relations as Seifert fibered spaces, but when $n>3$, the cyclic order of the strands also matters.  In either case, we can normalize the invariants and write $K[b;p_1/q_1, \ldots, p_n/q_n]$ where $0<p_i<q_i$ and $b\in\Z$.  In fact, we have the following proposition, which follows from Theorem \ref{thm:Montesinos} below.

\begin{proposition}
The double cover of $S^3$, branched along the Montesinos link $K[p_1/q_1, \ldots, p_n/q_n]$, is the Seifert fibered space $S^2(p_1/q_1, \ldots, p_n/q_n)$.
\end{proposition}

We remark that it is often helpful to allow $p_i/q_i$ to be zero, $\infty$, or 1 for some $i$, in either the notation for Montesinos links or Seifert fibered spaces.  For our purposes, this will only happen when the length $n$ is three or less, and the result should be clear from the context.  For example, $K[1/3, -1/2, 1/0]$ is the connected sum of a trefoil knot and a Hopf link, $K[1/3, 2/7, 0]=K[2/13]$, and $S^2(2,3,1)$ is a lens space.

\subsection{Seifert fibered surgery on knots with symmetries}\label{subsection:symmetries}

In this section, we recall some known results about Seifert fibered surgery on knots that admit a strong inversion, have period two, or both.  In what follows, let $K\subset S^3$ be a knot and let $\varphi:S^3\to S^3$ be a non-trivial orientation preserving involution such that $\varphi(K)=K$ and $C_\varphi = \text{Fix}(\varphi)\not=\emptyset$.  By the positive solution to the Smith conjecture, $C_\varphi$ is an unknotted circle in $S^3$ \cite{morgan-bass_smith}.

\begin{definition}  If $C_\varphi\cap K \not=\emptyset$, then $\varphi$ is called a \emph{strong inversion} of $K$ and $K$ is called \emph{strongly invertible}.  In this case, $C_\varphi\cap K =$ 2 points and $\varphi$ reverses the orientation of $K$.

If $C_\varphi\cap K= \emptyset$, then we say $\varphi$ is a \emph{cyclic symmetry of order 2} and that $K$ has \emph{period} $2$.
\end{definition}

In this paper, we will only be interested in strong inversions and cycles of period 2.  For a more general treatment of Dehn surgery on knots with symmetries, see \cite{motegi_symm}.

\begin{figure}
\centering
\includegraphics[scale = .6]{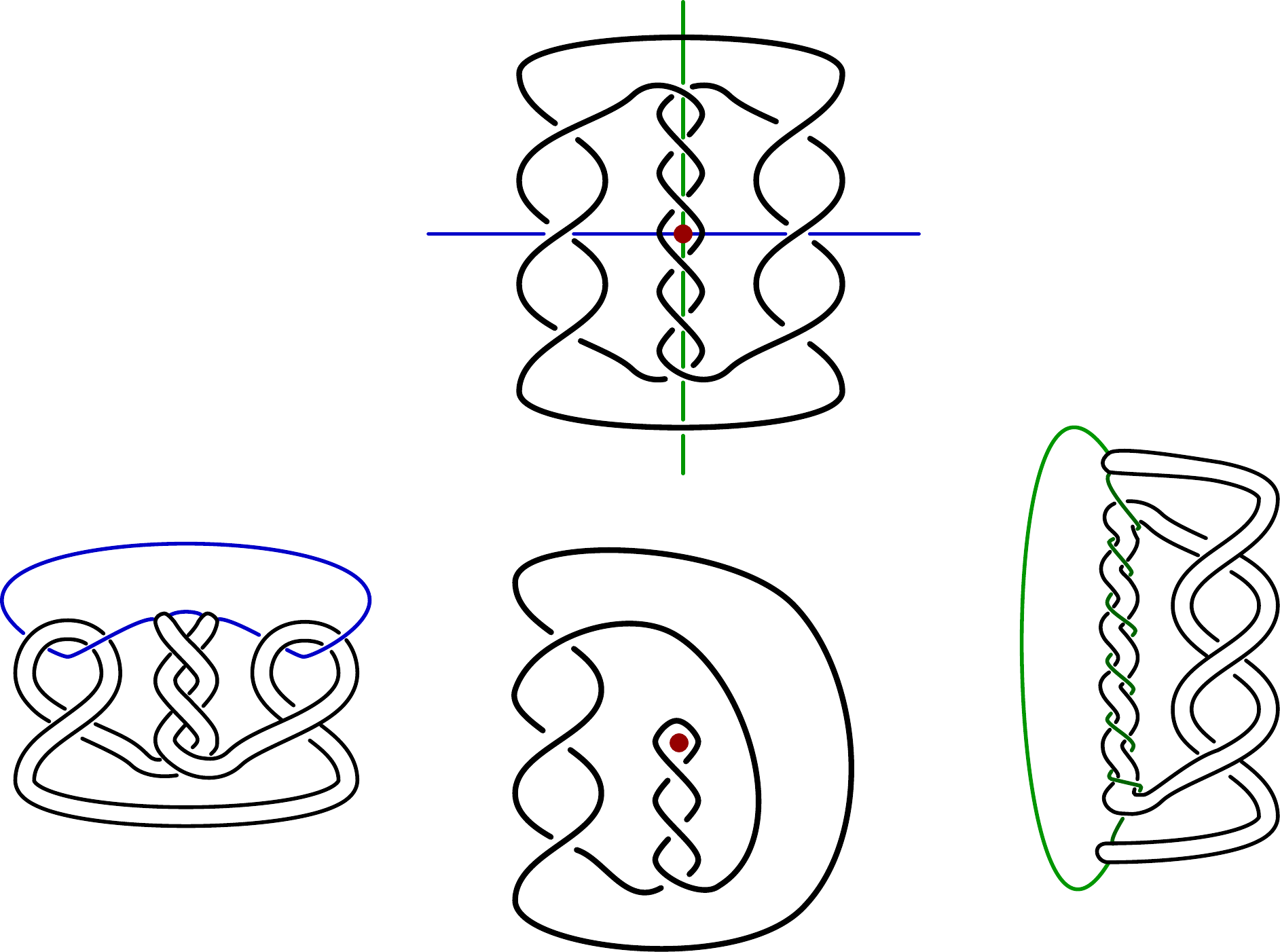}
\caption{The knot $P(3,3,-6)$, shown with its three symmetries and the resulting quotients.}
\label{figure:3Symmetries}
\end{figure}

First, let us consider strongly invertible knots.  Let $K\subset S^3$ be a knot with a strong inversion $\varphi$.  Then $\varphi$ restricts to an involution of the knot exterior, $M_K$, and the quotient of $M_K$ by the action of $\varphi$ is a tangle, $\mathcal T_K$. The well-known Montesinos trick gives a correspondence between Dehn filling on $M_K$ and rational tangle filling on $\mathcal T_K$.  For details, see \cite{gordon_survey}.  The following is originally due to Montesinos \cite{montesinos_surgery}.

\begin{theorem}\label{thm:Montesinos}
Let $\mathcal T$ be a marked tangle.  Then $\widetilde{\mathcal T}(r/s) \cong \widetilde{\mathcal T(-r/s)}$.
\end{theorem}

Let $L_{r/s} = \mathcal T_K(-{r/s})$, so $L_{r/s}$ is a knot or a two-component link in $S^3$ with $K({r/s})$ as the double cover of $S^3$, branched along $L_{r/s}$.  Suppose that $K({r/s})$ is a small Seifert fibered space.  Let $\bar\varphi:K({r/s})\to K({r/s})$ be the involution induced by extending $\varphi$ across the filling solid torus.  Then $K({r/s})/\bar\varphi = S^3$.

If $K$ is not a trefoil knot, then we can assume that $\bar\varphi$ is fiber-preserving \cite{motegi_symm}.  Let $\pi:K({r/s})\to S^2$ be the Seifert fibration of $K({r/s})$.  Let $C_{\bar\varphi} = \text{Fix}(\bar\varphi)$.  If each component of $C_{\bar\varphi}$ is a fiber in $K({r/s})$, then $K({r/s})\backslash C_{\bar\varphi}$ admits a Seifert fibered structure.  Since this structure is compatible with $\bar\varphi$, $S^3\backslash L_{r/s}$ admits a Seifert fibered structure.  In other words, $L_{r/s}$ is a \emph{Seifert link}.

Let $\hat\varphi:S^2\to S^2$ be the induced involution of the base orbifold.  If one component of $C_{\bar\varphi}$ is not a fiber in $K({r/s})$, then $\hat\varphi$ is reflection across the equatorial circle, $C_{\hat\varphi}$, of $S^2$ and all of the cone points lie on $C_{\hat\varphi}$ \cite{motegi_symm}.  In this case, $L_{r/s} = C_{\bar\varphi}/\bar\varphi$ is a length three Montesinos link \cite{miyazaki-motegi_Seifert}.  So, we have the following, as stated in \cite{ichi-jong_pqq}.

\begin{proposition}\label{prop:seifertormont}
Let $K$ be a strongly invertible hyperbolic knot, and let $r/s\in \Q$.  Let $L_{r/s}$ be the link obtained by applying the Montesinos trick to $K(r/s)$.  If $K(r/s)$ is a small Seifert fibered space with base orbifold $S^2$, then $L_{r/s}$ is either a Seifert link or a Montesinos link.
\end{proposition}

Seifert links are well understood \cite{burde-murasugi_links, eisenbud-neumann_links}.  In the present paper, we will only be concerned with Seifert knots and Seifert links with two components, in which case we have the following.

\begin{lemma}
Let $L\subset S^3$ be a Seifert link with at most two components.  Then $L$ is equivalent to one of the following:
\begin{enumerate}[(a)]
\item A torus knot
\item A two-component torus link
\item A two-component link consisting of a torus knot together with a core curve of the torus on which it lies.
\end{enumerate}
\end{lemma}

Note that, in particular, every component of a Seifert link is a torus knot or an unknot.

Now let $K\subset S^3$ be a knot with a cycle symmetry $\varphi$ of order 2.  Suppose that $K(r/s)$ is a Seifert fibered space with base surface $S^2$, and let $\bar\varphi$ be the extension of $\varphi|_{M_K}$ to $K(r/s)$.  Then, $K(r/s)$ has a $\bar\varphi$-invariant Seifert fibered structure \cite{miyazaki-motegi_Seifert}.  Let $C_{\bar\varphi} = \text{Fix}(\bar\varphi)$, and let $L_{r/s} = C_{\bar\varphi}/\bar\varphi$.

If $r$ is odd, then $L_{r/s}$ is a knot.  If $r$ is even, then $L_{r/s}$ is a link.  Let $K_\varphi= K/\varphi$.  $K_\varphi$ is called the \emph{factor knot of $K$} (with respect to $\varphi$), and let $C_\varphi = \text{Fix}(\varphi)$.  In the case where  $r$ is odd, we can view $L_{r/s}$ as the image of $C_\varphi/\varphi$ after $r/2s$ surgery on $K_\varphi$, so $L_{r/s}$ is a knot in $K_\varphi(r/2s)$.  If $r$ is even, then $L_{r/s}$ is the image of $C_\varphi/\varphi$ in $K_\varphi(r/2s)$ together with the core of the surgery torus, so $L_{r/s}$ is a link in $K_\varphi(r/2s)$.

Let $\pi:K(r/s)\to S^2$ be a Seifert fibering of $K(r/s)$, and let $\hat\varphi$ be the induced involution of $S^2$, with fixed point set $C_{\hat\varphi}$.  In \cite{miyazaki-motegi_Seifert}, it is shown that if $K$ is not a torus knot or a cable of a torus knot, then no component of $C_{\bar\varphi}$ is a fiber in $K(r/s)$ and $C_{\hat\varphi}$ is the equatorial circle in $S^2$.    This implies that $\hat\varphi$ is reflection across the equator.  Since $\bar\varphi$ is fiber preserving, $\hat\varphi$ must pair up cone points in the northern hemisphere with cone points in the southern hemisphere.  Let $k$ denote the number of cone points in the northern hemisphere.  For our purposes, $k=0$ or $k=1$.  From \cite{miyazaki-motegi_Seifert}, we have the following:

\begin{lemma}
\begin{enumerate}
\item  If $k=0$, then $K_\varphi(r/2s) = K(r/s)/\bar\varphi \cong S^3$.
\item  If $k=1$, then $K_\varphi(r/2s) = K(r/s)/\bar\varphi$ is a lens space..
\end{enumerate}
\end{lemma}

\begin{figure}
\centering
\includegraphics[scale = .5]{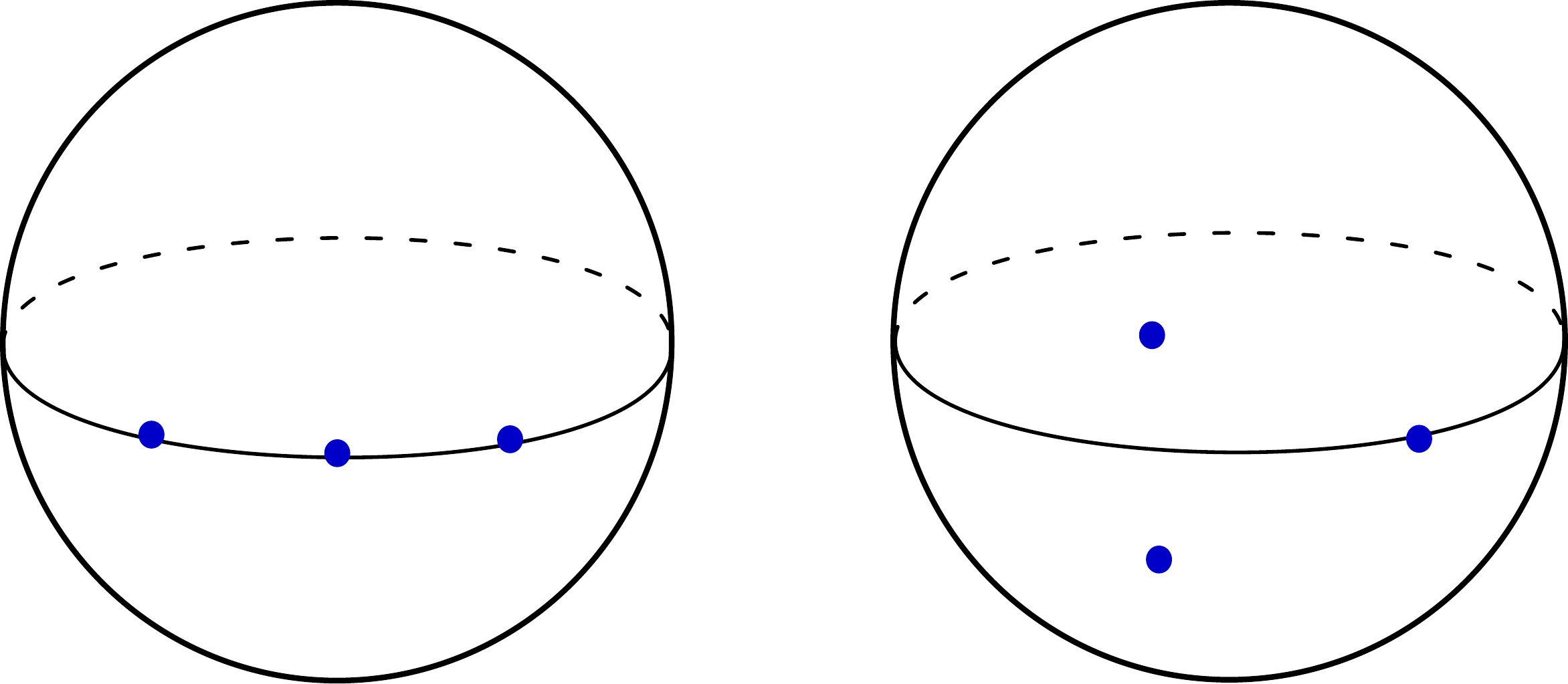}
\put(7,70){\large$C_{\hat\varphi}$}
\put(-185,70){\large$C_{\hat\varphi}$}
\caption{Possible configurations of cone points in the base sphere of a small Seifert fibered space}
\end{figure}

Note that $S^3$ and $S^2\times S^1$ are not lens spaces.  These facts can be helpful in obstructing Seifert fibered surgeries, based on the knot type of $K_\varphi$.  Throughout, $U$ will represent the unknot.

\begin{corollary}\label{corollary:FactorKnotType}
Let $K\subset S^3$ be a period 2 hyperbolic knot with factor knot $K_\varphi$.  Suppose $K(r/s)$ is a small Seifert fibered space.
\begin{enumerate}
\item  If $K_\varphi = T_{p,q}$ and $r$ is even, then $\Delta(pq, r/2s) = 1$, so $|r-2spq|=2$.
\item  If $K_\varphi = T_{p,q}$ and $r$ is odd, then $\Delta(pq, r/2s) = 1$, so $|r-2spq|=1$.
\item If $K_\varphi=U$ and $k=0$, then  $|r|\leq 2$.
\item If $K_\varphi= U$ and $k=1$, then $|r|\geq 3$.
\end{enumerate}
\end{corollary}

\begin{proof}
If $K(r/s)$ is a small Seifert fibered space, then $K_\varphi(r/2s)$ is a lens space if $k=1$ and $S^3$ if $k=0$.  Such surgeries on $U$ and $T_{p,q}$ are well understood (see \cite{gordon_survey}).
\end{proof}

\subsection{Some exceptional Dehn surgery results}

There are many important results in the study of exceptional Dehn surgery that give limitations on which slopes can be exceptional for a hyperbolic knot $K$.  Below, we present some of the results that will be used in this paper.  First, we state an important result of Lackenby and Meyerhoff \cite{lack-meyer_max} that tells us that exceptional fillings are always ``close'' to each other.

\begin{theorem}\label{thm:lack-meyer}
Suppose $M$ is a hyperbolic manifold with torus boundary component $T\subset\partial M$ and that $\alpha$ and $\beta$ are exceptional filling slopes on $T$.  Then $\Delta(\alpha,\beta)\leq 8$.
\end{theorem}

The distance bound of 8 above can be improved if one specifies the type of space for each of $M(\alpha)$ and $M(\beta)$.  Let $S$ and $T$ represent the sets of reducible and toroidal manifolds, respectively.  Let $L$ represent the set of lens spaces.  Let $\Delta(C_1, C_2)$ represent the largest possible value of $\Delta(\alpha, \beta)$ such that there exists a hyperbolic manifold $M$ with $M(\alpha)$ a manifold of type $C_1$ and $M(\beta)$ a manifold of type $C_2$ (we will always consider manifolds with one boundary component, though the theory is more general).  The following table presents the known values of $\Delta(C_1, C_2)$.

$$\begin{array}{|c||c|c|c|c|}
\hline
    &  \ S\   & \  T\   & \  S^3\   & \  L\    \\
\hline
\hline
S  &  1  &  3  &  ?  &  1  \\
\hline
T  &  &  8  &  2   &  ?    \\
\hline
S^3  &  &  &  0  &  1    \\  
\hline
L  &   &  &  &  1      \\
\hline
\end{array}$$

Notice that in the case of $(S, S^3)$, this is equivalent to the cabling conjecture, and in the case of $(T,L)$, the bound is known to be either 3 or 4 \cite{lee:lens}.  For a more thorough discussion of these bounds, the manifolds achieving them, and precise references, see \cite{gordon_survey} and \cite{gordon:small}.

Suspiciously absent from the above table are bounds on the distance between a (non-lens space) small Seifert fibered surgery and the other types of exceptional surgeries.  These seem to be the most difficult cases to analyze, and, in particular, it is not known whether or not the distance 8 bound of Lackenby and Meyerhoff can be improved in most of the cases (though, see Theorems \ref{BGZ:weak} and \ref{thm:BGZgenusone} below).

The following is a consequence of Corollary 7.6, Proposition 14.1, and Proposition 16.1 in \cite{BGZ:toroidalSeifert}.

\begin{theorem}\label{BGZ:weak}
For any hyperbolic manifold $M$, if $M(\alpha)=A\cup_{T^2}B$ is toroidal with one of $A$ or $B$ non-Seifert fibered, then for any slope $\beta$ such that $M(\beta)$ is a Seifert fibered space, $\Delta(\alpha,\beta)\leq 6$.
\end{theorem}

Since many of the pretzel knots studied below are genus one, it will be helpful for us to have the following result, which gives particularly strong bounds on small Seifert fibered surgery slopes \cite{BGZ:genusone} of such knots.

\begin{theorem}\label{thm:BGZgenusone}
Let $K$ be a hyperbolic knot of genus one such that $K(0)$ is a non-Seifert fibered toroidal manifold.  If $K(\alpha)$ is a small Seifert fibered space for some $\alpha\in\Q$, then $\Delta(\alpha,0)\leq 3$.
\end{theorem}

\subsection{Montesinos links, torus links, and invariants from knot theory}\label{subsec:knottheory}

In this section we give a very brief overview of some knot and link invariants and how they will be used to obstruct the quotient links encountered in this paper from being Seifert links or Montesinos links.  We will present a series of criteria that will applied in each of the following sections.

A link is called \emph{$k$-almost alternating} if it has a $k$-almost alternating diagram, but no $(k-1)$-almost alternating diagram, i.e., if it has a diagram $D$ such that changing $k$ crossings of $D$ gives a new diagram that is alternating, but no such diagram where the same result is achieved after $k-1$ crossing changes.

Recall that the Khovanov homology, $\text{Kh}(L)$, is a bi-graded abelian group associated to $L$, and that the \emph{width} of $\text{Kh}(L)$ is the number of diagonals that support a nontrivial element in $\text{Kh}(L)$.   Denote this width by $|\text{Kh}(L)|$. Then we have the following theorem.  (See, for example, \cite{asaeda-przytycki}.)

\begin{theorem}\label{thm:KhWidth}
Let $L$ be a non-split $k$-almost alternating link.  Then $|\text{Kh}(L)|\leq k+2$.
\end{theorem}

It has been shown by Abe and Kishimoto \cite{abe} that any Montesinos link is either alternating or 1-almost-alternating, so we have our first obstruction criterion.

\begin{criterion}\label{crit:width}
If $|\text{Kh}(L)|\geq 4$, then $L$ is not a length three Montesinos link.
\end{criterion}

When we encounter links $L$ that do not satisfy this criterion, then we will use the following program to show they are not a Montesinos link.  We will generate a list of all Montesinos links whose crossing numbers are compatible with that of $L$ (i.e., less than the number of crossings in a diagram of $L$).  (Note that the crossing number of a Montesinos link is well understood \cite{lick-thistle}.)  We will then check this list for elements that, if $L$ is a knot, have the same determinant, Alexander polynomial, Jones polynomial, Khovanov homology, and, if need be, Kauffman polynomial or HOMFLYPT polynomial, and that, if $L$ is a 2-component link, have the same determinant, Jones polynomial, Khovanov homology, and if need be, Kauffman polynomial or HOMFLYPT polynomial.  We will refer to this method as \emph{Method 1}.  This very large number of computations was performed using the KnotTheory` package for Mathematica\textsuperscript{\textregistered} \cite{wolfram}.

Examples of the Mathematica files used to implement Method 1 and to calculate knot invariants throughout this paper are available on the author's webpage, and further information will be provided upon request.

Next, we observe that if $K$ is a length three Montesinos link, then it is the union of 2-bridge knots and unknots.  If $K$ is the union of two unknots, then it has the form $K[p_1/q_1,p_2/q_2,p_3/q_3]$, where each $p_i$ is even.  If one component of $K$ is the 2-bridge knot $K[p/q]$, then $K$ has the form $K[p_1/q_1,p_2/q_2,x/q]$ with $q_1$ and $q_2$ even and with $x=p$ or $\bar p$, where $p\bar p \equiv 1\pmod q$.  If we consider the unknot a 2-bridge knot, then we have the following criterion.

\begin{criterion}\label{crit:2bridge}
If $K$ is a 2-component link such that one component is not a 2-bridge knot, then $K$ is not a Montesinos link.
\end{criterion}

Using Method 1 and Criteria \ref{crit:width} and \ref{crit:2bridge}, any knot or link we encounter that we claim is not a Montesinos knot or link is shown to not be a Montesinos knot or link.

Now we recall some facts about torus knots (see, for example, \cite{cromwell}).  Let $T(p,q)$ be the ($p,q)$-torus link for $p>q\geq 2$, where $T(p,q)$ is a knot if and only if $p$ and $q$ are coprime.  Then, $T(p,q)$ is a positive link, i.e., has a diagram with all positive crossings.  Furthermore, in the case of a torus knot, $2g(T(p,q)) = (p-1)(q-1)$, where $g(K)$ denotes the genus of the knot $K$, and $\det(T(p,q))=p$ if $q$ is even, and 1 if both $p$ and $q$ are odd, where $\det(L)$ denotes the determinant of the link $L$.  Let $s(K)$ denote the Rasmussen invariant of $K$, as defined in \cite{rasmussen}, where the following was shown.

\begin{proposition}
If $K$ is a positive knot, then $s(K)=2g(K)$. 
\end{proposition}

Recall that $2g(K)$ is bounded below by the breadth of the Alexander polynomial, which we denote $\text{br}(\Delta_K(t))$.  This gives us the following criterion.

\begin{criterion}\label{crit:positive}
If $s(K)<\text{br}(\Delta_K(t))$ or $2g(K)\not=s(K)$, then $K$ is not a torus knot.
\end{criterion}

We also have, by our discussion above:

\begin{criterion}\label{crit:det}
If $\det(K)>s(K)+1$, then $K$ is not a torus knot.
\end{criterion}

If we consider the unknot a torus knot, then each component of a two component Seifert link is a torus knot, so we have the following criterion.

\begin{criterion}\label{crit:torus}
If $L$ is a two-component link such that a component is not a torus knot, then $K$ is not a Seifert link.
\end{criterion}

In what follows, Criteria \ref{crit:positive}, \ref{crit:det}, and \ref{crit:torus} often suffice to prove that a link is not a Seifert link.  In the few cases where they fail, further argument is given to accomplish the feat.

\section{The case of $(2, |q_2|, |q_3|)$}\label{section:2pq}

Let $K_{p,q}$ be the hyperbolic pretzel knot  $P(-2, 2p+1, 2q+1)$ (see Figure \ref{fig:(2,p,q)PretzelTangleKleinBottle}).  Since Ichihara and Jong have shown that $K_{p,p}$ admits no small Seifert fibered surgery \cite{ichi-jong-kabaya_2pp}, and by interchanging $p$ and $q$ if necessary, we may assume that $|q|>|p|$ if $p$ and $q$ have the same sign and $p>0$ otherwise.  Let $\alpha_r = 4(p+q+1)-r$, for $r\in\Q$.

Note that $\alpha_r$ is chosen this way so that $\alpha_r$-surgery on $K_{p,q}$ will correspond with $r$-filling of $\mathcal T_{p,q}$.   This becomes clear if one carefully follows through the process of obtaining $\mathcal T_{p,q}$ from $K_{p,q}$ by applying the Montesinos trick and isotoping.

\begin{figure}
\centering
\includegraphics[scale = .50]{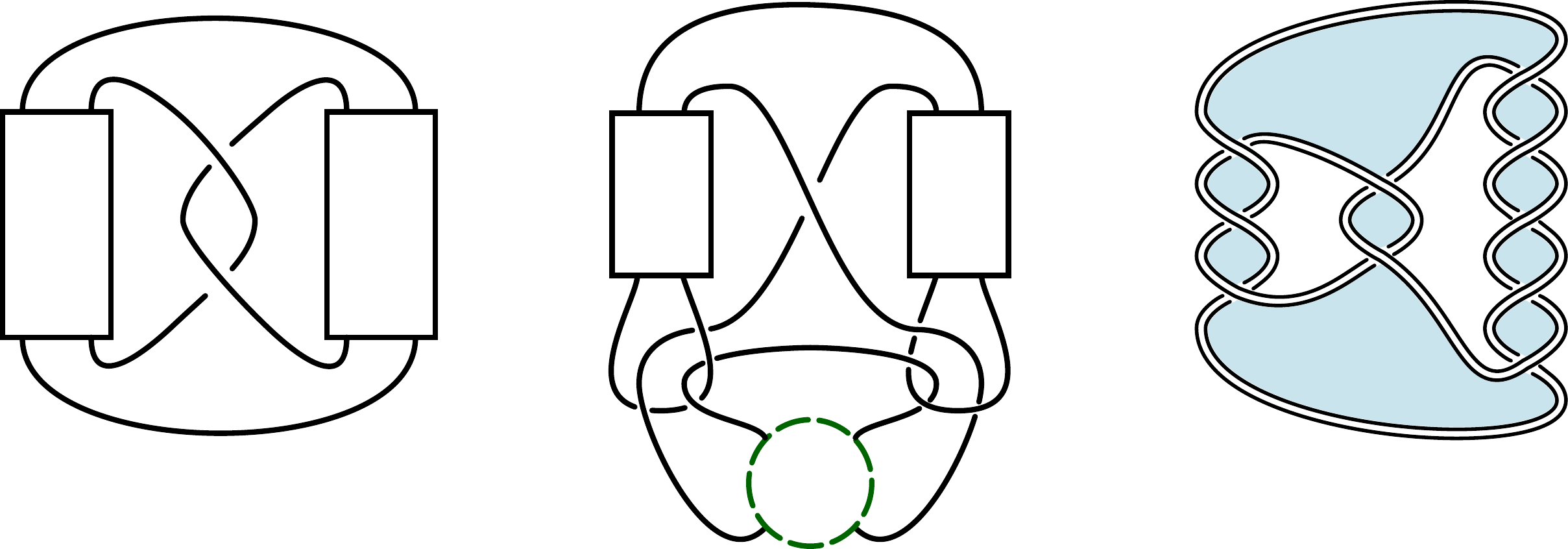}
\put(-138,78){$q$}
\put(-276,70){\footnotesize$2q$+$1$}
\put(-205,78){$p$}
\put(-348,70){\footnotesize$2p$+$1$}
\caption{The pretzel knot $P(-2,2p+1, 2q+1)$, the quotient tangle $\mathcal T_{p,q}$, and the pretzel knot $P(-2,5,-3)$, shown as the boundary of a punctured Klein bottle}
\label{fig:(2,p,q)PretzelTangleKleinBottle}
\end{figure}
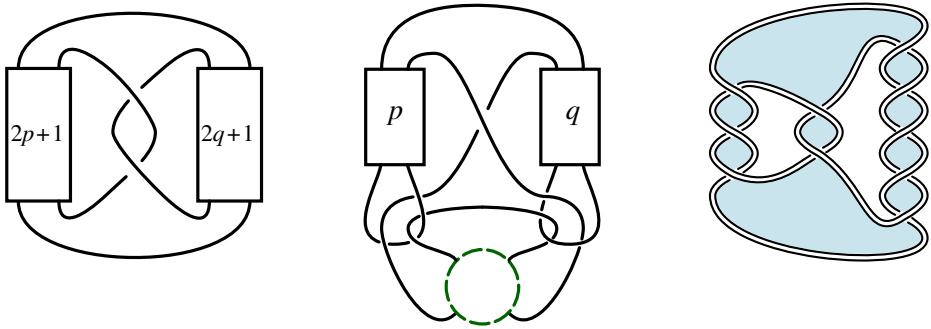

Our first result is the following.

\begin{reptheorem}{thm:main1}
The hyperbolic pretzel knot $P(-2, 2p+1, 2q+1)$, with the conventions discussed above, admits a small Seifert fibered surgery if and only if $p=1$, in which case it admits precisely the following small Seifert fibered surgeries:
\begin{itemize}
\item $P(-2,3,2q+1)(4q+6) = S^2(1/2,-1/4,2/(2q-5))$
\item $P(-2,3,2q+1)(4q+7) = S^2(2/3,-2/5,1/(q-2))$
\end{itemize}
\end{reptheorem}

We remark that the existence of these exceptional surgeries was previously known \cite{eudave-munoz}.

The key fact in our method of analyzing these knots is that they are strongly invertible.  Let $\mathcal T_{p,q}$ be the tangle obtained by performing the Montesinos trick (see Figure \ref{fig:(2,p,q)PretzelTangleKleinBottle}).  We now have the advantage of viewing the surgery space $K_{p,q}(\alpha_r)$ as the branched double cover of $S^3$ along $\mathcal T_{p,q}(r)$.  It is easy to verify the two classes of exceptional surgeries in Theorem \ref{thm:main1} by noticing that $\mathcal T_{1,q}(1)$ and $\mathcal T_{1,q}(2)$ are the Montesinos links $K[2/3, -2/5, 1/(q-2)]$ and $K[1/2, -1/4, 2/(2q-5)]$, respectively. [See Figure \ref{fig:(-2,3,2q+1)TangleAndExceptionalFillings}.]

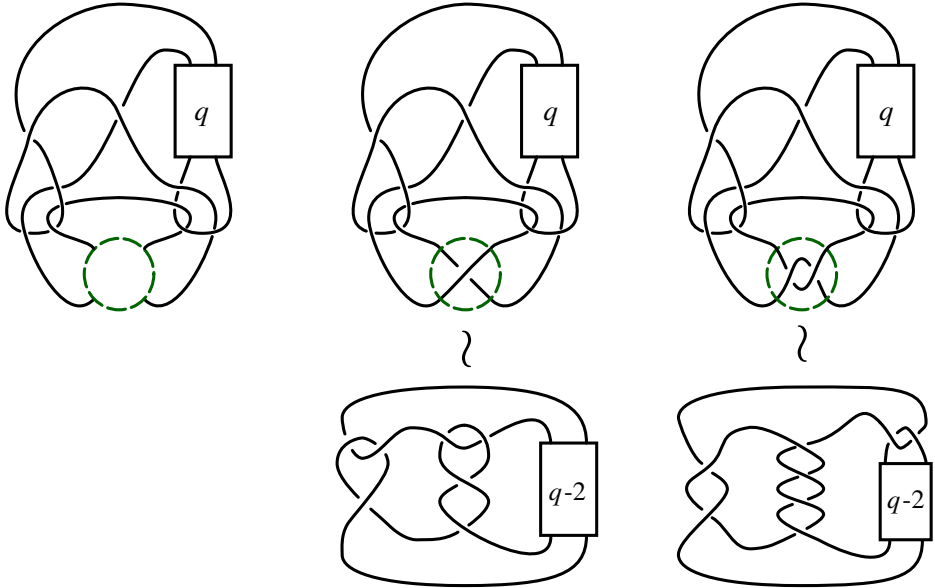
\begin{figure}
\centering
\includegraphics[scale = .48]{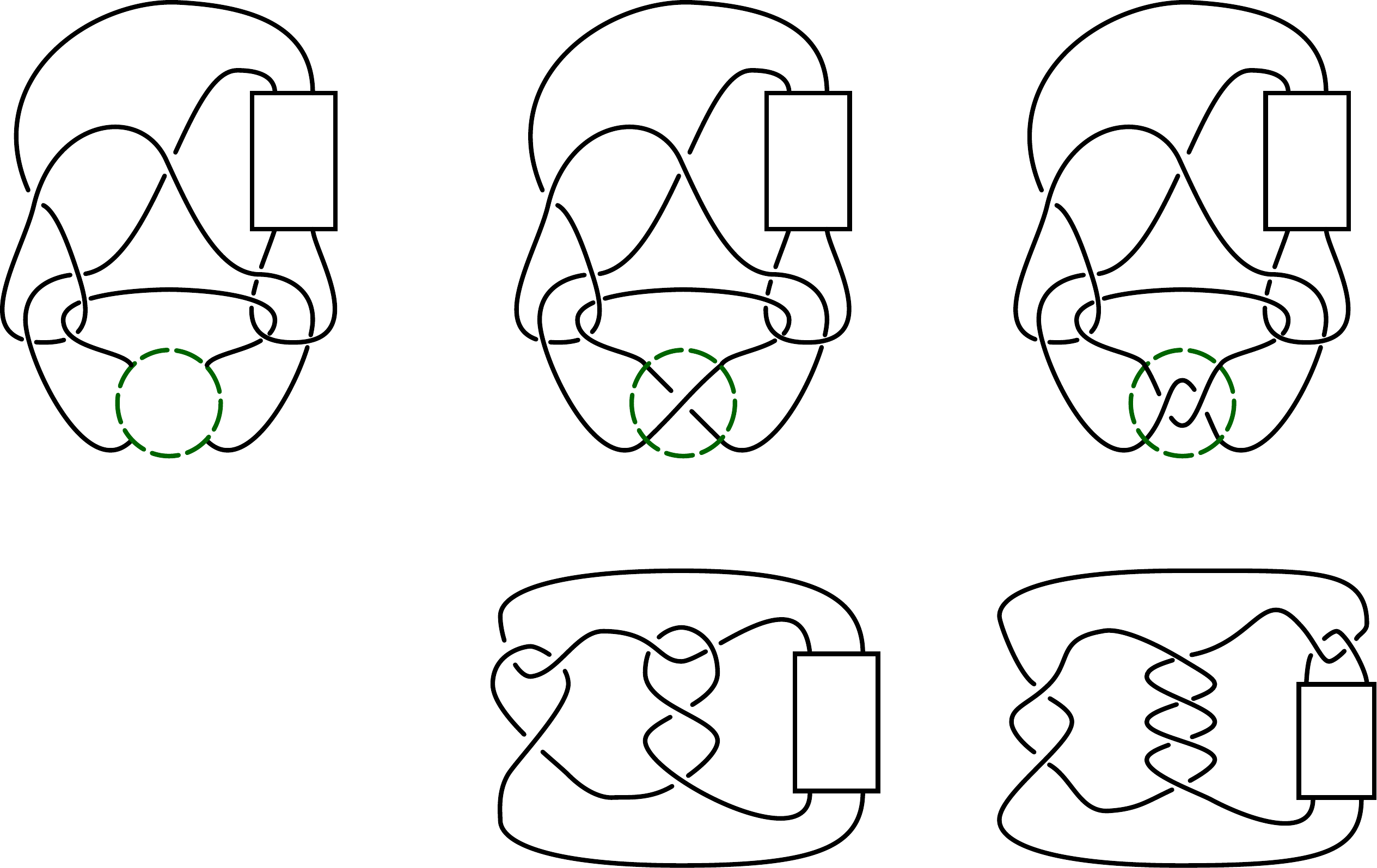}
\put(-21,177){$q$}
\put(-148,177){$q$}
\put(-18,30){\small$q$-2}
\put(-146,33){\small$q$-2}
\put(-280,177){$q$}
\put(-180,85){\huge$\wr$}
\put(-53,87){\huge$\wr$}
\caption{The tangle $\mathcal T_{1,q}$, along with fillings $\mathcal T_{1,q}(1)$ and $\mathcal T_{1,q}(2)$ and the respective Montesinos links that result after isotopy: $K[2/3,-2/5,1/(q-2)]$ and $K[1/2, -1/4, 2/(2q-5)]$}
\label{fig:(-2,3,2q+1)TangleAndExceptionalFillings}
\end{figure}

The proof that the $K_{p,q}$ admits no other small Seifert fibered surgeries is accomplished by the following two lemmas and the techniques of Subsection \ref{subsec:knottheory}.

\begin{lemma}\label{lemma:2pq1}
If $K_{p,q}(\alpha_r)$ is a small Seifert fibered space for $p\not=1$, then $|p|\leq 8$ and $|q|\leq 8$.
\end{lemma}

\begin{lemma}\label{lemma:2pq2}
If $K_{1,q}(\alpha_r)$ is a small Seifert fibered space for $r\not\in\{1,2\}$, then $|q|\leq 8$. 
\end{lemma}

Before we prove these lemmas, we should remark on the possible surgery slopes $\alpha_r$.  We notice that each knot $K_{p,q}$ bounds a punctured Klein bottle at slope $\alpha_0$ (see Figure \ref{fig:(2,p,q)PretzelTangleKleinBottle}).  It follows that $K_{p,q}(\alpha_0)$ is toroidal.  By Theorem \ref{thm:lack-meyer}, it follows that if $K_{p,q}(\alpha_r)$ is a small Seifert fibered space, then $\Delta(\alpha_r,\alpha_0)\leq 8$.

In many cases, it should be possible to reduce this distance bound to 5, but this is dependent on work in progress by Boyer, Gordon, and Zhang \cite{BGZ:toroidalSeifert}.  However, using Theorem \ref{BGZ:weak}, we can fairly easily show the following lemma.

\begin{lemma}\label{lemma:2pqslopes}
If $K_{p,q}(\alpha_r)$ is a small Seifert fibered space, and $|p|, |q|\geq 4$, then $\Delta(\alpha_r,\alpha_0)\leq 6$.
\end{lemma}

Of course, if $r$ is integral, this means that $|r|\leq 6$, and if we have $r/s\in\Q$, we have that $|4(p+q+1)s-r|\leq 6$.

\begin{proof}[Proof of Lemma \ref{lemma:2pqslopes}]
We begin by noticing that $K_{p,q}(\alpha_0) = D^2(2,2)\cup_{T^2}X_{p,q}$ (see Figure \ref{fig:Tpq(0)}). Under the hypotheses of the lemma, we will show that $X_{p,q}$ is not a Seifert fibered space, so Theorem \ref{BGZ:weak} gives us the desired bound.  If $p=1$, then $X_{1,q}=D^2(3,|q-1|)$, and Theorem \ref{BGZ:weak} does not apply.

\begin{figure}
\centering
\includegraphics[scale = .48]{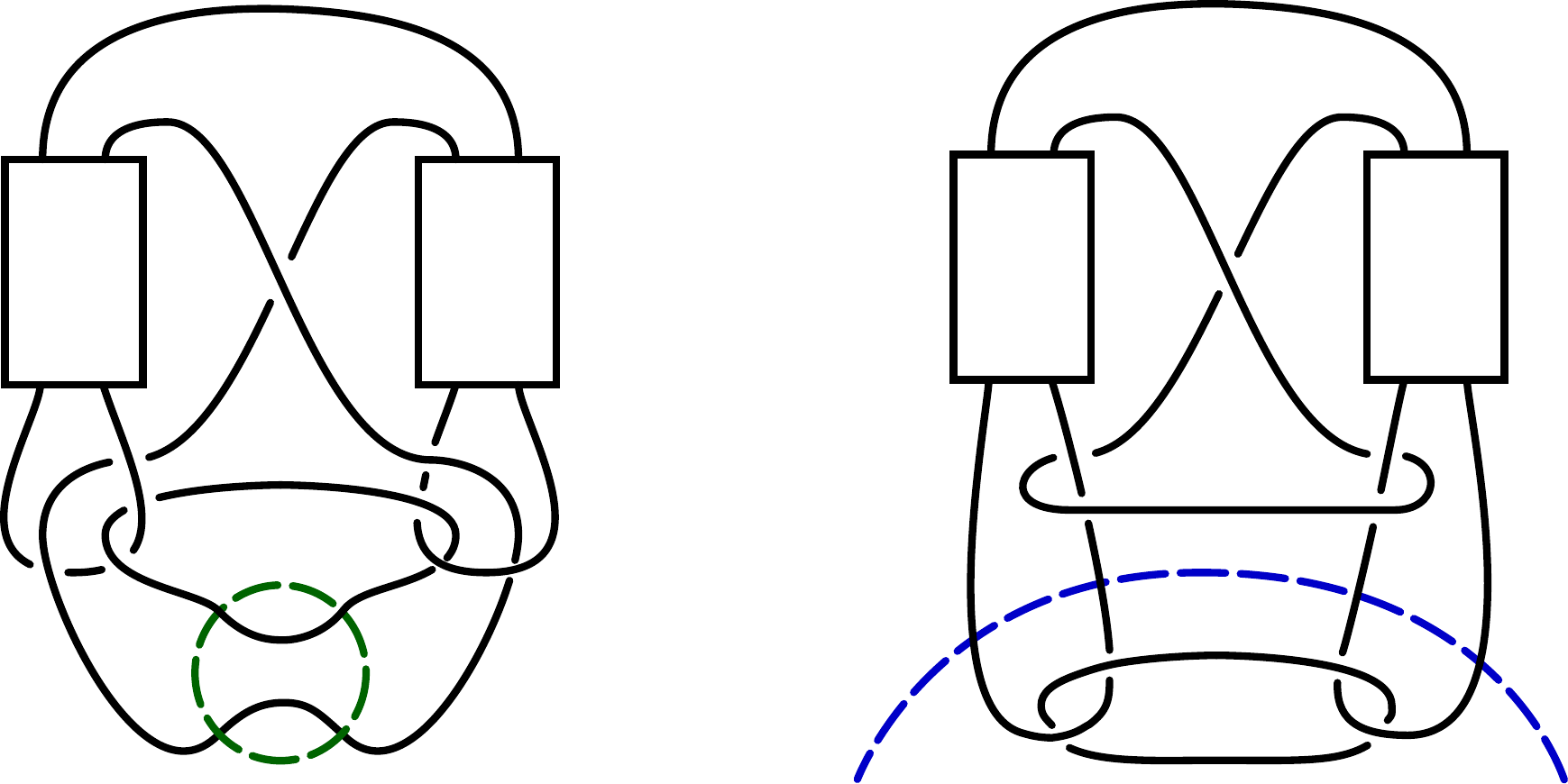}
\put(-89,77){$p$}
\put(-232,77){$p$}
\put(-29,77){\small$q$+1}
\put(-170,77){\small$q$}
\caption{The link $\mathcal T_{p,q}(0)$, whose branched double cover corresponds to $\alpha_0$-surgery on $K_{p,q}$}
\label{fig:Tpq(0)}
\end{figure}
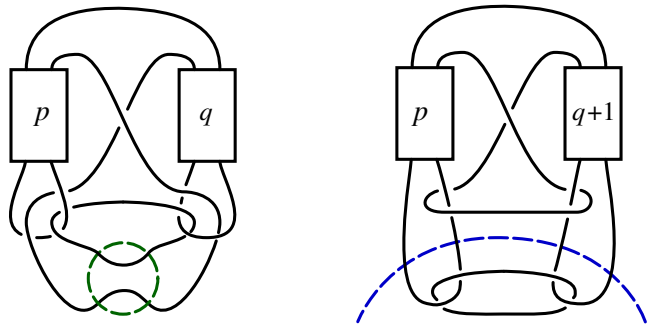

Consider the following fillings on $X_{p,q}$. (See Figure \ref{fig:TpqFillings}.)

$$\begin{array}{l} 
X_{p,q}(0) = L(p+q+1,1) \\
X_{p,q}(\infty) = L(4pq-2p-2q-3,2pq-q-2) \\
X_{p,q}(-1) = S^2(1/3,1/(p-1),q/(q-1))
\end{array}$$

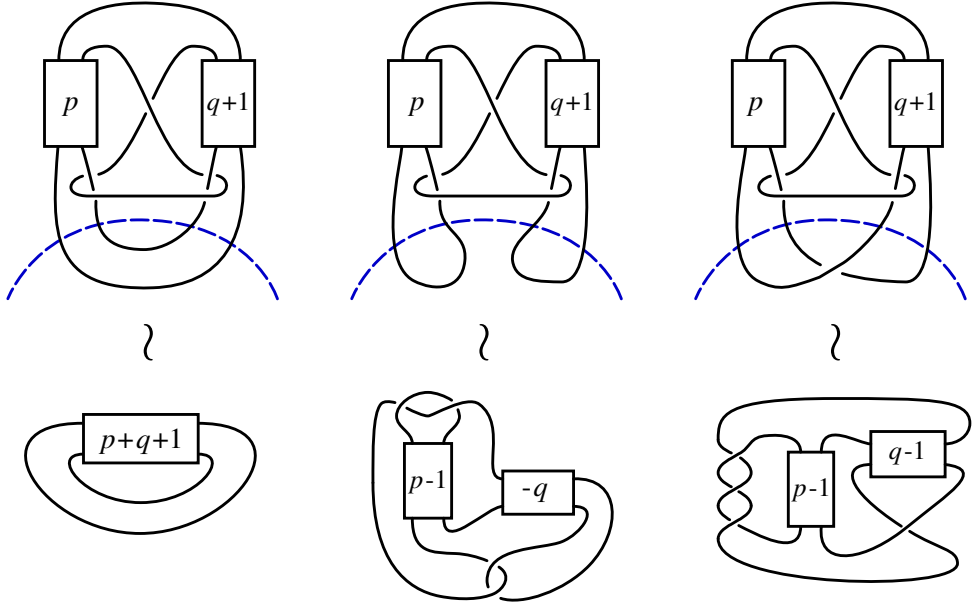
\begin{figure}
\centering
\includegraphics[scale = .45]{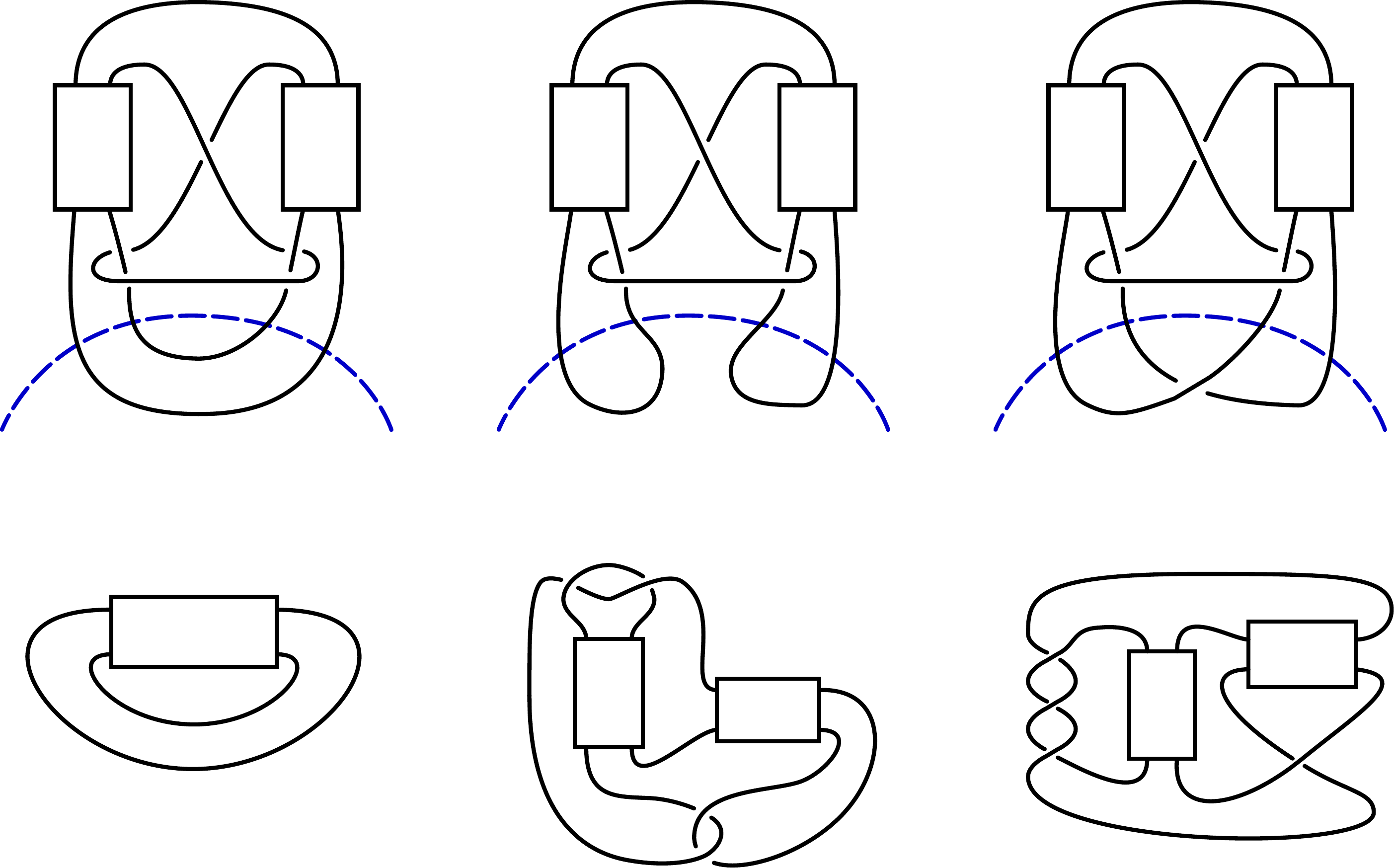}
\put(-344,186){$p$}
\put(-290,186){\small$q$+1}
\put(-213,186){$p$}
\put(-160,186){\small$q$+1}
\put(-85,186){$p$}
\put(-30,186){\small$q$+1}
\put(-330,58){$p$+$q$+1}
\put(-213,43){\small$p$-1}
\put(-170,40){-$q$}
\put(-32,54){\small$q$-1}
\put(-68,40){\small$p$-1}
\put(-315,93){\huge$\wr$}
\put(-188,93){\huge$\wr$}
\put(-55,93){\huge$\wr$}
\caption{The three fillings, 0, $\infty$, and 1, on $\mathcal T_{p,q}$ used to show that $X_{p,q}$ is not Seifert fibered.}
\label{fig:TpqFillings}
\end{figure}

If $X_{p,q}$ is a Seifert fibered space, then it has, for its base surface, either $D^2$ or $M^2$ (the M\"{o}bius band).  We will make use of Lemma \ref{SFfillings}.  If $X_{p,q}$ is Seifert fibered over the disk with more than two exceptional fibers it cannot have lens space fillings.  If $X_{p,q}$ has the form of $D^2(a)$, then it cannot have fillings with three exceptional fibers, so $X_{p,q}(-1)$ must be a lens space.  This implies that $p=2$ or $q=2$.  If $X_{p,q}$ has base surface $M^2$, then it can only have lens space fillings or fillings with at least three exceptional fibers, two of which have multiplicity two. Thus, we must have $p,q=2,3$.  So, assume $X_{p,q}=D^2(a,b)$.

In this case, $X_{p,q}$ has one reducible filling at slope $\gamma$ and the property that any lens space filling must be at distance one from $\gamma$.  By considering the three fillings given above, it follows that $\gamma = 0, \infty,$ or $\pm1$.  If $\gamma=-1$, then $X_{p,q}(-1)$ must be reducible, so $p=1$ or $q=1$, both of which are not allowed values.  If $\gamma=0$ or if $\gamma=\infty$, then $X_{p,q}(-1)$ must be a lens space, so $p=2$ or $q=2$.  Finally, if $\gamma=1$, then the filling $X_{p,q}(-1)$ is at distance two from the reducible filling, so it must have an exceptional fiber of multiplicity 2.  It follows that $p=3$ or $q=3$.
\end{proof}

We remark that the lemma could be strengthened to say that $X_{p,q}$ is non-Seifert fibered if and only if $p\not=1$ by showing that $X_{p,q}(1)$ is neither reducible, a lens space, or a small Seifert fibered space with finite fundamental group, as would need to be the case given the different Seifert fibered structures $X_{p,q}$ might have.  However, we will not need anything stronger than what we have proved.

\begin{proof}[Proof of Lemma \ref{lemma:2pq1}]

Suppose that $p\not=1$ and remove a ball around the $q$-twists of $\mathcal T_{p,q}(r)$ to form the tangle $\mathcal S_{p,r}$ (see Figure \ref{fig:SprWithFillings}).  Let $N_{p,r}$ denote the branched double cover of $\mathcal S_{p,r}$.  First, we will show that $N_{p,r}$ is hyperbolic.

We begin by showing some interesting fillings of $N_{p,r}$ (see Figure \ref{fig:SprWithFillings}).

$$\begin{array}{l} 
N_{p,r}(-1/q) = K_{p,q}(\alpha_r) \\ 
N_{p,r}(0) = D^2(1/2, -1/2)\cup_{T^2}D^2(-1/2,p/(2p-1)) \\
N_{p,r}(\infty) = T(2,2p+3)(\alpha_r) = S^2(-1/2, -1/(r+2), -2/(2p+3)) \\
N_{p,r}(-1) = K[-2p/(6p+1)](\alpha_r)
\end{array}$$

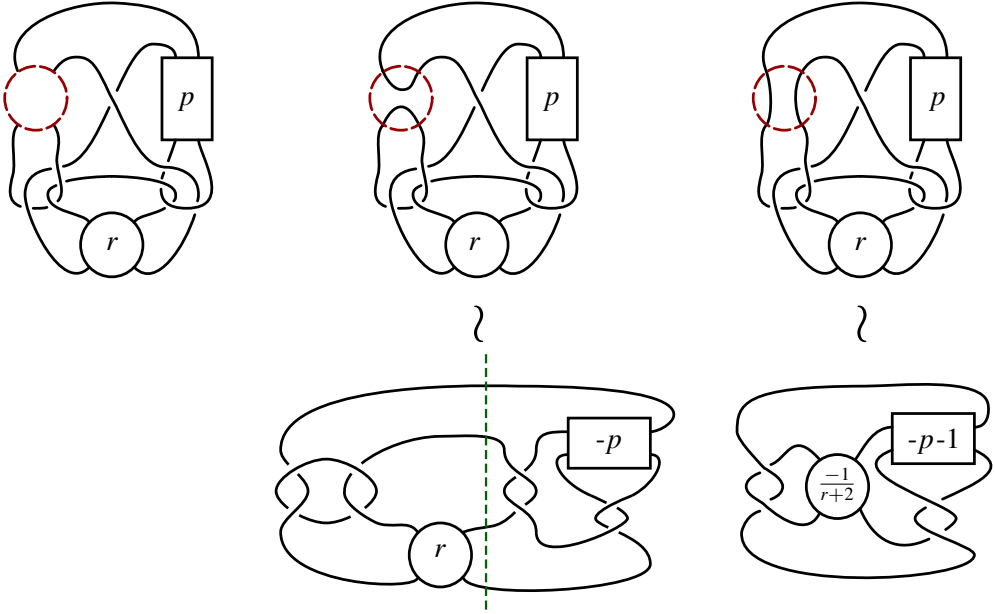
\begin{figure}
\centering
\includegraphics[scale = .43]{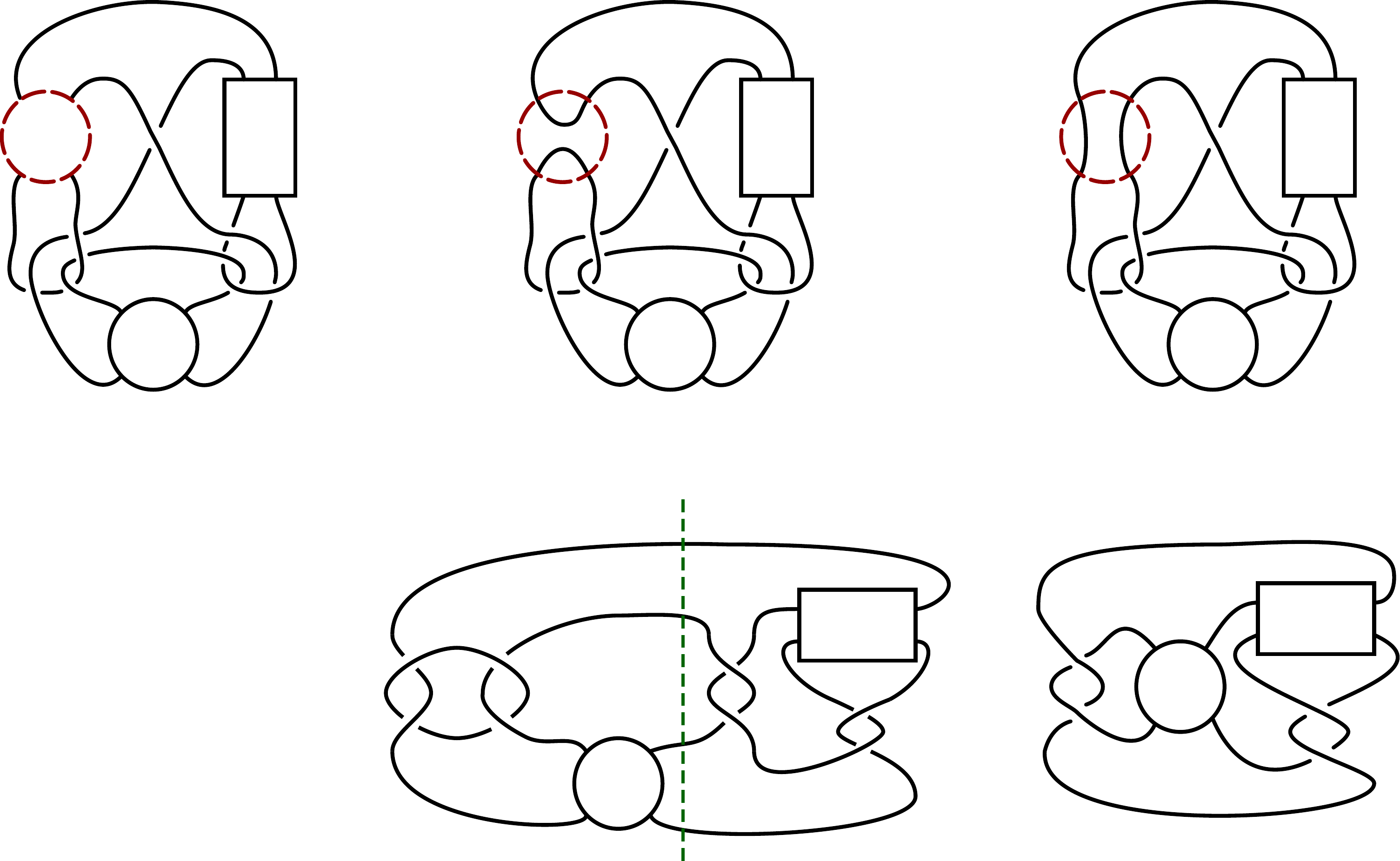}
\put(-308,191){$p$}
\put(-170,191){$p$}
\put(-24,191){$p$}
\put(-336,136){$r$}
\put(-199,136){$r$}
\put(-53,136){$r$}
\put(-32,61){-$p$-1}
\put(-150,61){-$p$}
\put(-212,20){$r$}
\put(-67,45){\small$\frac{-1}{r+2}$}
\put(-198,103){\huge$\wr$}
\put(-53,103){\huge$\wr$}
\caption{The tangle $\mathcal S_{p,r}$, along with two fillings, $\mathcal S_{p,r}(0)$ and $\mathcal S_{p,r}(\infty)$, and their equivalents after isotopy.}
\label{fig:SprWithFillings}
\end{figure}

We remark that $N_{p,r}(\infty)$ and $N_{p,r}(-1)$ correspond to ($\alpha_r$)-surgery on $K_{p,0}$ and $K_{p,-1}$, respectively.  The latter is a 2-bridge knot with no exceptional fillings (if $p\not=1$), according to the classification by Brittenham and Wu (\cite{britt-wu:2bridge}).  Thus, $N_{p,r}(-1)$ is hyperbolic if $p\not=1$.

Now, suppose that $N_{p,r}$ is not hyperbolic, so it must be reducible, $\partial$-reducible, Seifert fibered, or toroidal by geometrization.  However, $N_{p,r}$ cannot be reducible, since it has two distinct irreducible fillings at slopes 0 and $-1$ (this follows from the solution to the knot complement problem \cite{gordon-luecke}).  Similarly, it cannot be Seifert fibered, since it has a hyperbolic filling at slope $-1$.  It follows that $N_{p,r}$ cannot be $\partial$-reducible, since the only irreducible, $\partial$-reducible manifold with torus boundary is Seifert fibered, namely, the solid torus.

Finally, suppose that $N_{p,r}$ is toroidal.  If any essential torus were non-separating, then all fillings of $N_{p,r}$ would contain an essential non-separating torus, which is false here.  Suppose any essential torus is separating, and decompose $N_{p,r}$ along an outermost such torus, $F$, so that $N_{p,r} = A\cup_FB$ with $A$ atoroidal and $\partial N_{p,r}\subset B$.  If we assume, for a contradiction, that $N_{p,r}(-1/q)$ is small Seifert fibered for some $q$ with $|q|>8$, then we have that $F$ compresses in $B(-1)$, $B(\infty)$, and $B(-1/q)$.  It follows, from Lemma \ref{lemma:cable} that $B$ is a cable space with cabling slope $\gamma=0$.  But $N_{p,r}(0)$ is neither reducible nor a lens space, so we reach a contradiction.  It follows that $N_{p,r}$ is not toroidal, and must be hyperbolic.

Since $N_{p,r}$ is hyperbolic, and $N_{p,r}(\infty)$ is exceptional, it follows that for any exceptional filling $N_{p,r}(-1/q)$, $\Delta(\infty, -1/q)\leq 8$, by Theorem \ref{thm:lack-meyer}.  It follows that $|q|\leq 8$, as desired.  A similar argument shows that $|p|\leq 8$, as well.

\end{proof}

\begin{proof}[Proof of Lemma \ref{lemma:2pq2}]

We will proceed as in the lemma above by analyzing the tangle $\mathcal S_r = \mathcal S_{1,r}$ formed by removing a ball containing the $q$-twist region of knot $\mathcal T_{1,q}(r)$.  We will show that the branched double cover $N_r$ of $\mathcal S_r$ is hyperbolic.

Consider the following fillings on $N_r$ (see Figure \ref{fig:S1rFillings}).  Assume for a contradiction that $|q|\geq 9$ and $K_{p,q}(\alpha_r) = N_r(-1/q)$ is a small Seifert fibered space.

$$\begin{array}{l} 
N_{r}(-1/q) = K_{p,q}(\alpha_r) \\ 
N_{r}(0) = S^2(1/2, -1/2, 1/(2-r))\\
N_{r}(\infty) = S^2(1/2,-2/5,-1/(r+2)) \\
N_{r}(-1) = S^2(1/3,-1/4,-1/r)\\
N_r(-1/2) = S^2(-1/3,2/5,-1/(r-1))\\
N_r(1) = K[-2/7](\alpha_r)
\end{array}$$

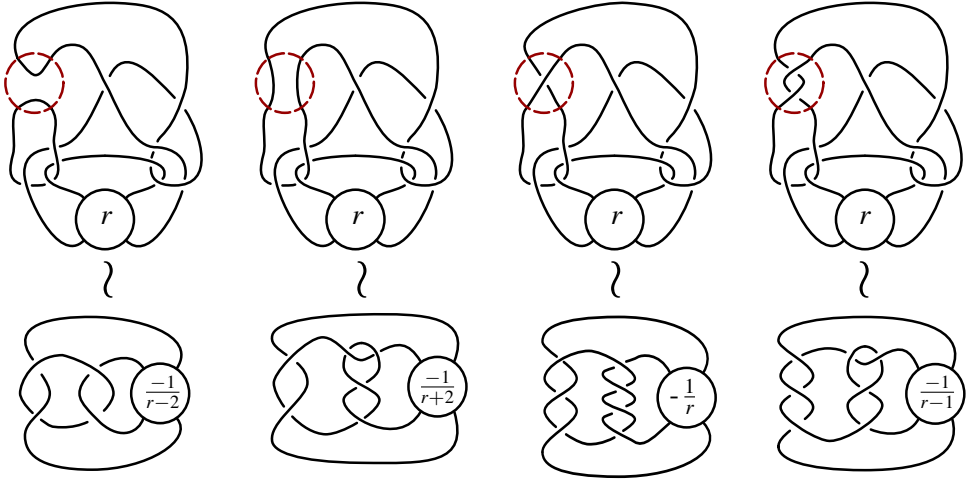
\begin{figure}
\centering
\includegraphics[scale = .40]{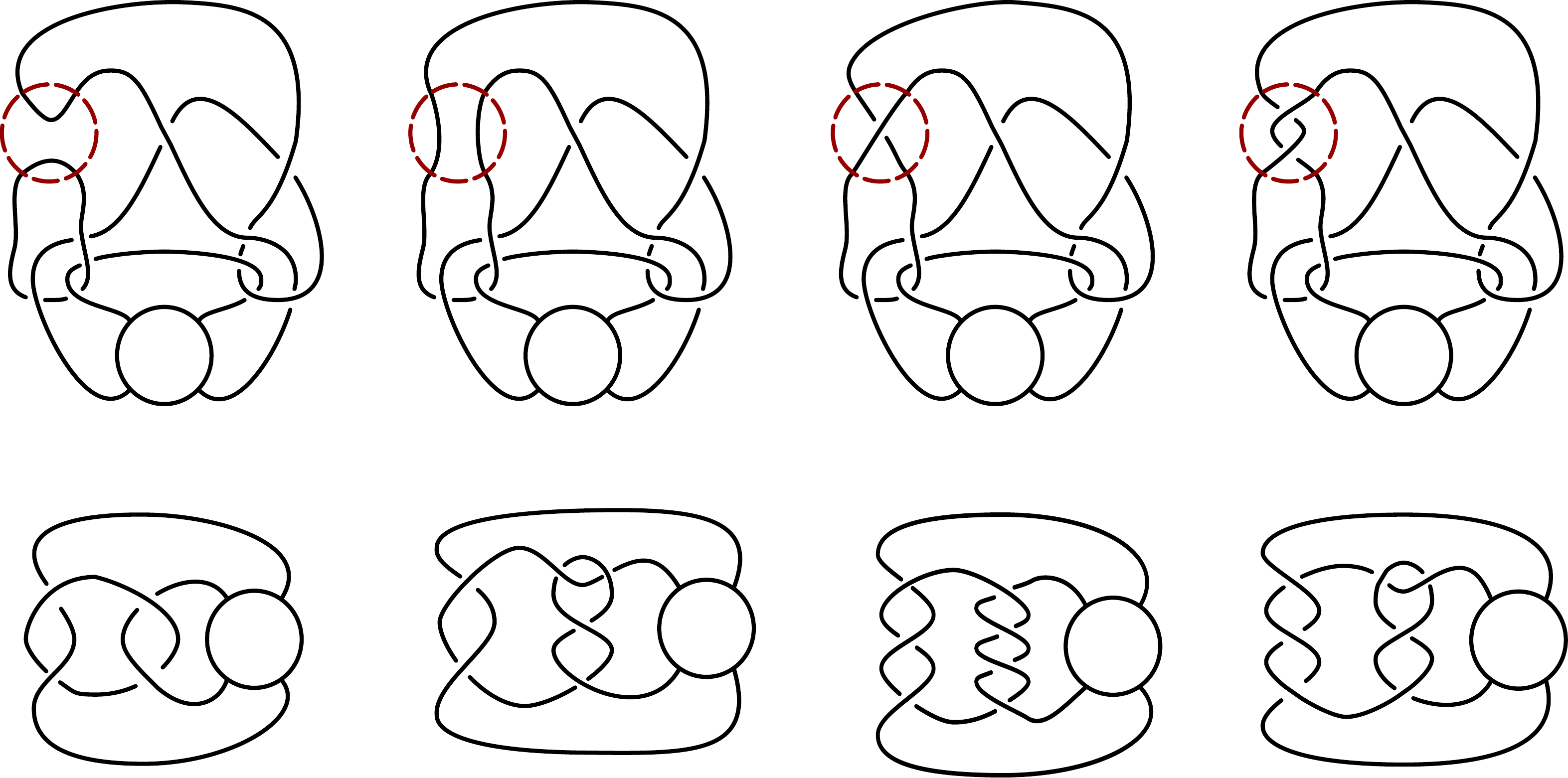}
\put(-328,96){$r$}
\put(-233,96){$r$}
\put(-41,96){$r$}
\put(-135,96){$r$}
\put(-312,30){\small$\frac{-1}{r-2}$}
\put(-208,32){\small$\frac{-1}{r+2}$}
\put(-112,28){-$\frac{1}{r}$}
\put(-19,30){\small$\frac{-1}{r-1}$}
\put(-328,70){\huge$\wr$}
\put(-232,70){\huge$\wr$}
\put(-136,70){\huge$\wr$}
\put(-42,70){\huge$\wr$}
\caption{Four fillings, 0, $\infty$, 1, and 1/2, of the tangle $\mathcal S_{r}$ that help to prove that $N_r$ is hyperbolic if $r\not=0,1,2$.}
\label{fig:S1rFillings}
\end{figure}

It is clear from this that $N_r$ is irreducible (again, by \cite{gordon-luecke}), since it has distinct irreducible fillings, for any value of $r$.  Suppose that $N_r$ is Seifert fibered.  Since, for all values of $r$, $N_r$ has fillings that are Seifert fibered with base surface $S^2$, but do not contain a pair of exceptional fibers of multiplicity 2, the base surface of $N_r$ is orientable, i.e., $D^2$.  A Seifert fibered space with connected boundary with a small Seifert fibered filling must have 2 or 3 exceptional fibers.  Furthermore, since no slope is distance one from 0, $\infty$, and $-1$, $N_r$ has 2 exceptional fibers, i.e., $N_r = D^2(a,b)$.

By the classification of exceptional surgeries on 2-bridge knots \cite{britt-wu:2bridge}, $N_r(1)$ is exceptional if and only if $\alpha_r\in\{0,1,2,3,4\}$.  In this case, $p=1$ and $q=-1$, so $\alpha_r=4-r$ and this is the equivalent to $r\in\{4,3,2,1,0\}$.  However, we already know that if $r=0,1,2$, then $N_r(-1/q)$ is exceptional, so we only need to consider $r=3,4$.  

If $r=3$, then, by considering $N_3(\infty)$, we see that $a=5$, and by considering $N_3(-1)$, we see that $a$ cannot be 5.  If $r=4$, then by considering $N_4(-1/2)$, we see that $a=3$, and by considering $N_4(\infty)$, we see this is impossible.  It follows that $N_r$ cannot be Seifert fibered if $r\not\in\{0,1,2\}$.

It follows that $N_r$ is non-Seifert fibered and, thereby, $\partial$-irreducible.  If $N_r$ were toroidal, since it has atoroidal fillings at distance two, it must be a cable space, by Lemma \ref{lemma:cable}.  However, the only cabling slope $\gamma$ that satisfies $\Delta(\gamma, -1/2)=1$ and $\Delta(\gamma, \infty)=1$ is $\gamma=0$, in which case we must have $N_r(0)$ be reducible or a lens space.  So, we must have $r=3$ and $N_3(0)$ is a lens space.  Suppose $N_3 = A\cup_{T^2}B$ with $A$ atoroidal and $\partial N_3\subset B$.  

Then, because $B(\alpha)=S^1\times D^2$ for $\alpha\in\{-1/q,\infty,-1,-1/2,1\}$ each of these fillings induces a filling $\eta_\alpha$ on $A$.  Since $B(0)=(S^1\times D^2)\#L$ for some lens space $L$, and since $N_3(0)$ is a lens space, it follows that $A(\eta_0)=S^3$, so $A$ is a knot complement.  By construction, $A$ is atoroidal.  If $A$ were Seifert fibered, then, by considering $A(\eta_\infty)=N_3(\infty)$ and $A(\eta_{-1})=N_3(-1)$ just as before, we reach a contradiction.

It follows that $N_r$ must be hyperbolic (for $r\not\in\{0,1,2\}$).  An application of Theorem \ref{thm:lack-meyer} gives us that $\Delta(-1/q, \infty)\leq 8$ if $N_r(-1/q)$ is a small Seifert fibered space, which proves the lemma.

\end{proof}

\subsection{Completing the proof of Theorem \ref{thm:main1}}

The work above leaves us with a finite list of knots and links $L_{p,q,r}=\mathcal T_{p,q}(r)$ whose branched double covers might be Seifert fibered.  We must consider non-integral $r$ if and only if $p$ and $q$ are both positive.  In the event of a non-integral slope $r/s$, we may assume $|s|\leq 8$ by Theorem \ref{thm:lack-meyer}, since $1/0$ is an exceptional filling.  By Proposition \ref{prop:seifertormont}, we must show that each of these links is not a Montesinos link or a Seifert link.

Method 1 (see Subsection \ref{subsec:knottheory}) can be used to show that none of the $L_{p,q,r}$ are Montesinos knots or links, though it should be noted that the Kauffman polynomial must be employed in a handful of cases, including distinguishing $L_{1,4,-1}$ from $K[1/3, 2/5, -2/5]$ and some non-integral cases, and that the HOMFLYPT polynomial must be employed to distinguish $L_{1,6,-2}$ from $K[-1/4,1/6,2/7]$.  In other words, these pairs are not distinguished by their Alexander polynomial and Khovanov homology alone.

Now, consider when $L_{p,q,r}$ is a link.  Then it is the union of the unknot and the 2-bridge knot $K[n/m]$, where $n=2pq-p-2$ and $m=4pq-2p-2q-3$.  $K[n/m]$ is a torus knot only if $p=1$ and $q=3$ or 4.  In the latter case, $L_{1,4,r}$ is the union of a trefoil and an unknot.  If this link is to be a Seifert link, the unknotted component must lie as the core of the torus upon which the trefoil sits.  However, the link just described can be distinguished from $L_{1,4,r}$ for all values of $r$ using the Jones polynomial.  Concerning $L_{1,3,r}$, exceptional surgeries on the knot $P(-2,3,7)$ are previously well-understood \cite{eudave-munoz}.

It only remains to show that $L_{p,q,r}$ is never a torus knot.  If $p\not=1$, this is accomplished by applying Criterion \ref{crit:positive}.  For $p=1$, Criterion \ref{crit:det} suffices.


\section{The case of $(3,3,|q_3|)$}\label{section:33q}

We now turn our attention to pretzel knots $P(q_1, q_2, q_3)$ such that $|q_1|=|q_2|=3$.  The case of $P(3,3,q_3)$, where $q_3>0$ was handled by Ichihara and Jong in \cite{ichi-jong_pqq}.  We break up the remaining cases as follows.
\begin{enumerate}
\item $P(3,\pm3, -2m)$ with $m\geq 1$
\item $P(3,3, 2m+1)$ with $ m\leq -2$
\item $P(3,-3, 2m+1)$ with $m\leq -3$ or $2\leq m$
\item $P(3,3,2m,-1)$ with $2\leq m$
\end{enumerate}

In Cases (2) and (3), it is only necessary to consider integral surgery slopes by Theorem \ref{thm:non-integral}.  Our main result is:

\begin{reptheorem}{thm:main2}
Hyperbolic pretzel knots of the form $P(3,3,m)$ or $P(3,3,2m,-1)$ admit no small Seifert fibered surgeries.  Pretzel knots of the form $P(3,-3,m)$, with $m>1$, admit small Seifert fibered surgeries precisely in the following cases:
\begin{itemize}
\item $P(3,-3, 2)(1) = S^2(1/3,1/4,-3/5)$
\item $P(3,-3, 3)(1) = S^2(1/2,-1/5,-2/7)$
\item $P(3,-3, 4)(1) = S^2(-1/2,1/5,2/7)$
\item $P(3,-3, 5)(1) = S^2(2/3,-1/4,-2/5)$
\item $P(3,-3, 6)(1) = S^2(1/2,-2/3,2/13)$
\end{itemize}
\end{reptheorem}

\subsection{Case (1)}

Let $K^\pm_m=P(3,\pm3,-2m)$.  To avoid the redundancy of mirrors, we can restrict to $m>0$.  Recall that these knots can only admit non-integral small Seifert fibered surgeries in the case of $K^+_m$.  Our first result is half of Theorem \ref{thm:main2} (up to mirroring).

\begin{figure}
\centering
\includegraphics[scale = .6]{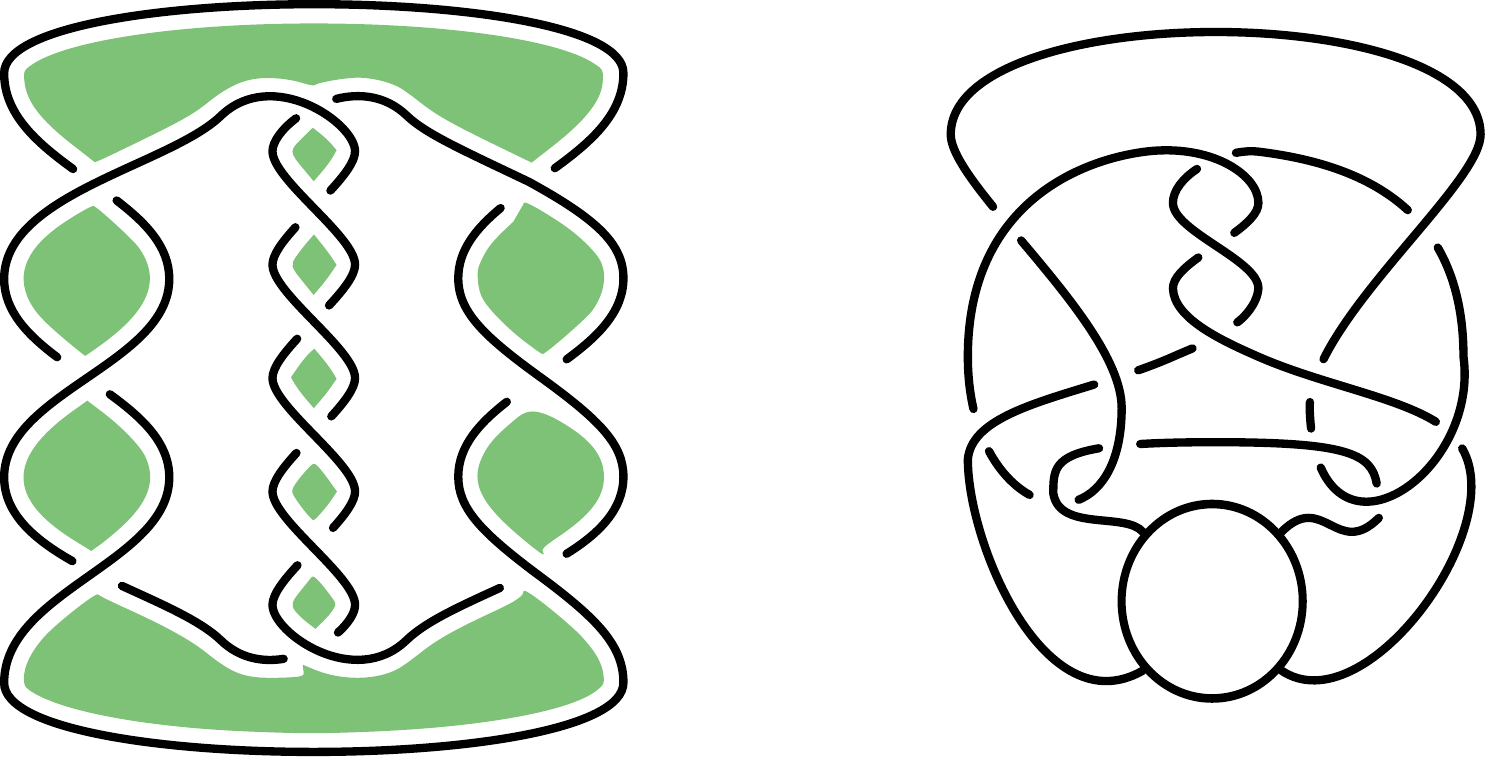}
\put(-210,-15){(a)}
\put(-55,-15){(b)}
\caption{(a) The knot $K_m^-=P(3,-3,-2m)$ (shown here with $m=3$) bound punctured Klein bottles.  (b) The tangle $\mathcal T^+_3$.}
\label{figure:P(3,-3,m)KleinBottleAndTangle}
\end{figure}

\begin{proposition}\label{prop:(3,3,2m)}
Let $K^\pm_m=P(3,\pm3,-2m)$ with $m>0$ be hyperbolic.  Then, $K^\pm_m$ admits no small Seifert fibered surgeries, except in the following three instances.
\begin{itemize}
\item $P(3,-3,-2)(-1) = S^2(2/3,-1/4,-2/5)$
\item $P(3,-3,-4)(-1) = S^2(1/2,-1/5,-2/7)$
\item $P(3,-3,-6)(-1) = S^2(1/2,-1/3,-2/13)$
\end{itemize}
\end{proposition}

As in Section \ref{section:2pq}, we will proceed in this case by first limiting the possible surgery slopes, then limiting the size of $m$, then using techniques from Subsection \ref{subsec:knottheory} to check that small values of $m$ and slopes satisfying the relevant bound do not produce small Seifert fibered spaces (except for the three noted cases).  Let $K^\pm_m=P(3,\pm3,-2m)$ with $m>0$.  We begin by observing that $K^\pm_m$ bounds a punctured Klein bottle.  Let $\alpha_r^+=12-r$ and $\alpha_r^-=-r$ (again, these are chosen so that $K_m^\pm(\alpha_r)$ corresponds to $r$-filling on the corresponding tangle).  Then this Klein bottle has boundary slope $\alpha^\pm_0$ (see Figure \ref{figure:P(3,-3,m)KleinBottleAndTangle}(a)).  Since surgery along this slope produces a toroidal manifold (as in the previous section), any exceptional surgery slope for $K^\pm_m$ must be close to $\alpha^\pm_0$.  In particular, $\Delta(\alpha_r^\pm,\alpha_0^\pm)\leq 8$.

Next, we remark that $K^\pm_m$ is strongly invertible.  Let $\mathcal T^\pm_m$ be the resulting quotient tangle, and let $L^\pm_{m,r}=\mathcal T^\pm_m(r)$.


\begin{lemma}
Suppose $K^\pm_m(\alpha_r)$ is a small Seifert fibered space.  Then, $m\leq 8$.
\end{lemma}

\begin{proof}

Let $L^\pm_{m,r}=\mathcal T^\pm_m(r)$ and form the tangle $\mathcal S^\pm_r$ by removing a 3-ball containing the $m$-twist box of $L^\pm_{m,r}$ (see Figure \ref{figure:P(3,-3,m)MrAndLr(m)}).  Let $M^\pm_r = \widetilde{S^\pm_r}$.  Of course, $\mathcal S^\pm_r(-1/m) = L^\pm_{m,r}$, and $K^\pm_m(\alpha_r) = M^\pm_r(1/m)$.  Assume, for a contradiction, that $M^\pm_r(1/m)$ is a small Seifert fibered space for some $m\geq 9$.

\begin{figure}
\centering
\includegraphics[scale = .6]{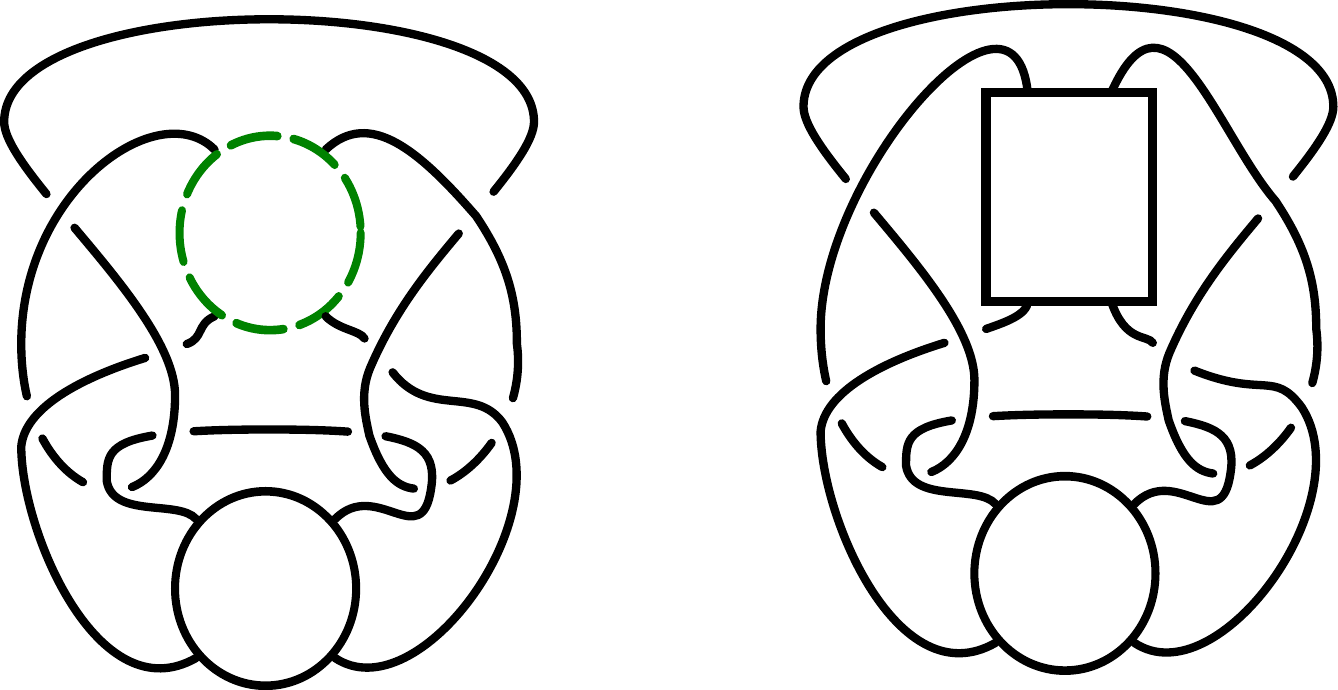}
\put(-55,84){\large-$m$}
\put(-190,14){\large$r$}
\put(-52,16){\large$r$}
\caption{The tangle $\mathcal S_r^-$ and the link $L_{m,r}^-$}
\label{figure:P(3,-3,m)MrAndLr(m)}
\end{figure}

Consider the following fillings of $M^\pm_r$, which we can easily visualize and verify by looking at the corresponding rational tangle fillings of $\mathcal S^\pm_r$ (see Figure \ref{figure:P(3,-3,2m)MrFillings}).
$$\begin{array}{rcl} 
M_r^\pm(1/m) & =& \text{small Seifert fibered space (by assumption)}\\
M_r^-(0)    & = &(S^1\times S^2)\#L(r,1) \\
M_r^\pm(\infty)  & =& D^2(2,3)\cup_{T^2}D^2(2,3) 
\end{array}$$

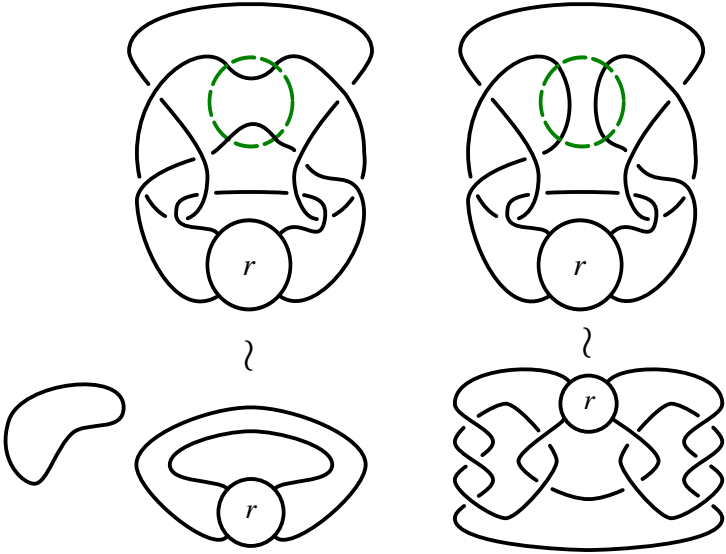
\begin{figure}
\centering
\includegraphics[scale = .6]{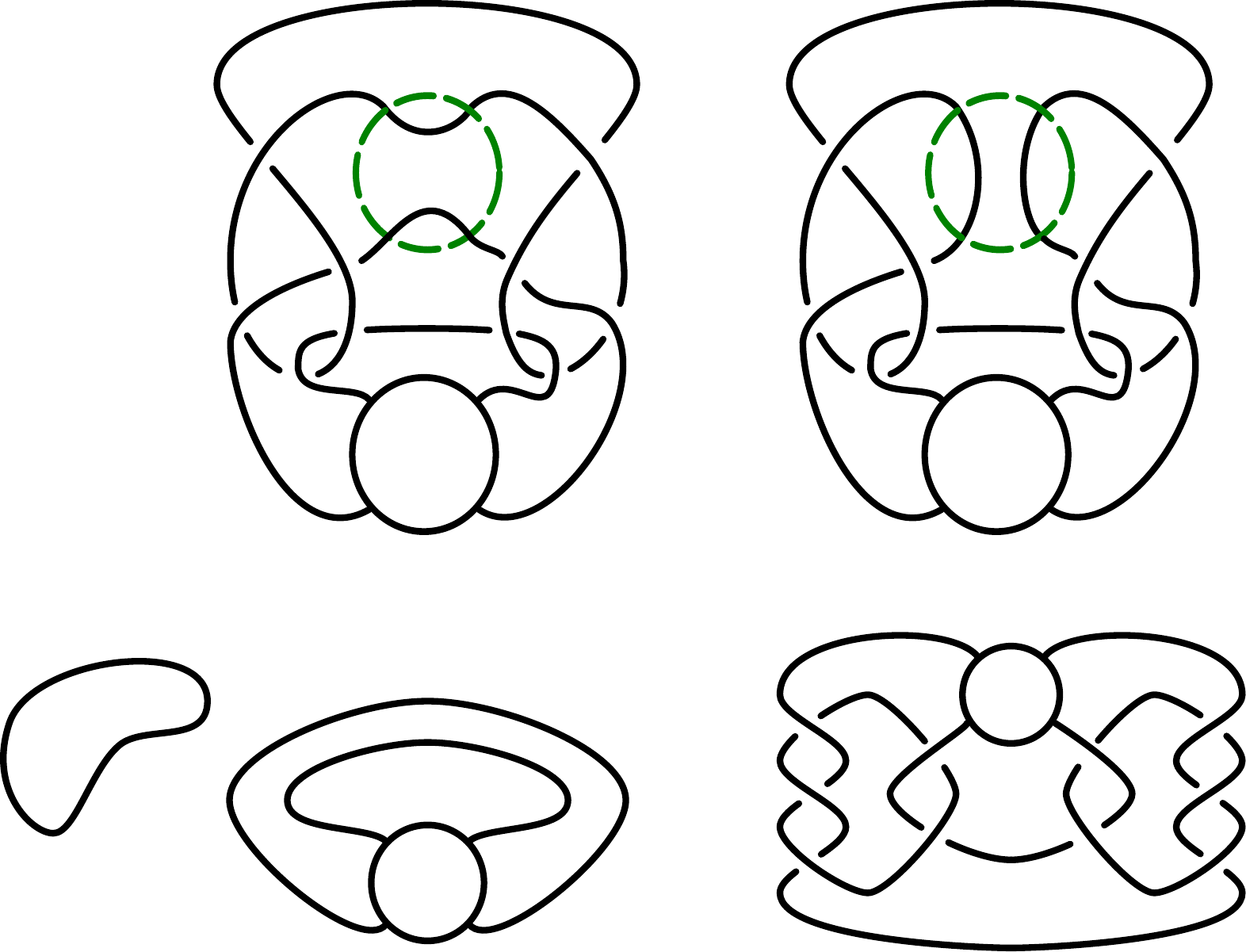}
\put(-183,105){\large$r$}
\put(-58,105){\large$r$}
\put(-183,70){\LARGE$\wr$}
\put(-55,75){\LARGE$\wr$}
\put(-182,13){$r$}
\put(-54,54){$r$}
\caption{Two interesting fillings of $\mathcal S_r^-$}
\label{figure:P(3,-3,2m)MrFillings}
\end{figure}

As was argued in Section \ref{section:2pq}, $M^\pm_r$ is irreducible (it has distinct irreducible fillings), non-Seifert fibered (it has a non-Seifert fibered, non-reducible filling), and $\partial$-irreducible (it is irreducible and not $S^1\times D^2$).  Assume that $M^-_r$ is toroidal, so $M^-_r = A\cup_F B$ with $A$ atoroidal and $\partial M^-_r\subset B$.

Suppose that $F$ compresses in $M^-_r(\infty)$.  Then, since $\Delta(1/m,\infty) = |m|\geq 2$, $B$ is a cable space with cabling slope $\gamma$ satisfying $\Delta(\gamma, \infty) = \Delta(\gamma, 1/m) = 1$.  It follows that $\gamma=a\in\Z$, and $|ma-1|=1$.  Since $|m|\geq 9$, $a$ must be zero, so $\gamma=0$.  It follows that $B(0) = (S^1\times D^2)\#L$, where $L$ is a lens space, and $B(\infty) = B(1/m) = S^1\times D^2$.  Let $\eta_0,\eta_\infty,$ and $\eta_{1/m}$, be the slopes of the induced slopes of the meridian of $B$ after the above fillings are performed, so $M^-_r(\alpha) = A(\eta_\alpha)$ for $\alpha =\infty$ or $1/m$, and $M^-_r(0) = A(\eta_0)\#L$.  Since $L$ has finite fundamental group, $L = L(r,1)$ and $A(\eta_0) = S^1\times S^2$.

Now, since $\Delta(\eta_\infty, \eta_{0})\geq 4$ (by Lemma \ref{lemma:cable}), $A$ cannot be hyperbolic, because $\Delta(S^2,T^2)=3$.  Since $A$ was assumed to be atoroidal, it follows that $A$ must be Seifert fibered.  But $A(\eta_\infty)$ is irreducible and not Seifert fibered, so this cannot be.  This contradiction means that $F$ cannot compress in $M^-_r(\infty)$.

So, assume $F$ remains incompressible in $M^-_r(\infty)$.  If $|r|=1$, then $M_1^-(0) = S^1\times S^2$.  But since $F$ is the unique incompressible torus in a non-Seifert fibered graph manifold, $A=D^2(2,3)$.  In particular, $M^-_1(0) = A(\eta_0) = S^1\times S^2$ is not a possible filling of a trefoil complement.

If $|r|>1$, then since $F$ compresses in fillings at slopes $1/m$ and 0, which are at distance one, by Lemma \ref{lemma:cable}, $B$ is either a cable space or the exterior of a braid in a solid torus.  Since $M^-_r(0) = (S^1\times S^2)\#L(r,1)$, either $B(0)$ or $A(\eta_0)$ is $S^1\times S^2$.  However, this is not possible for such spaces $B$, nor is it possible for $A = D^2(2,3)$.

Thus, $M^-_r$ is not toroidal, and must be hyperbolic.  So, by Theorem \ref{BGZ:weak}, $\Delta(1/m, \infty) = |m|\leq 8$, a contradiction that yields the desired result.  The reasoning is very similar to show that  $M_r^+$ must be hyperbolic as well, noting that $M_r^+(0) = D^2(2,2)\cup_{T^2}D^2(2,r)$ is toroidal for all $r$, and $M_r^+(1) = S^2(1/3, -1/4, -1/r)$.

\end{proof}

\subsection{Completing the proof of Proposition \ref{prop:(3,3,2m)}}

By our work above, we can conclude that if $P(3,\pm 3, -2m)(\alpha^\pm_r)$ with $m>0$ admits a small Seifert fibered surgery, then $|r|,|m|\leq 8$.  Let $L^\pm_{m,r}=\mathcal T^\pm_{m}(r)$.  We assume that $r$ is integral for $L^-_{m,r}$.

First, consider the links $L_{m,r}^\pm$, which are the union of an unknotted component with a component $J_m^\pm = K[2/3, \pm 2/3, -1/2m]$ (to see this, consider $L^\pm_{m,0}$, or compare with Figure \ref{figure:P(3,3,2n,-1)Lf}).  Because $J_m^\pm$ is not a torus knot or a 2-bridge knot for any $m$,  by Criteria \ref{crit:2bridge} and \ref{crit:torus}, we can conclude that $L_{m,r}^\pm$ is never a 2-component Seifert link or Montesinos link.

When $L_{m,r}^\pm$ is a knot, we see that $2g(L_{m,r}^\pm)\not=s(L_{m,r}^\pm)$, so, by Criterion \ref{crit:positive}, $L_{m,r}^\pm$ is never a torus knot.  To see that $L_{m,r}^\pm$ is never a Montesinos knot, we implement Method 1 (see Subsection \ref{subsec:knottheory}), accounting for $r$ non-integral when necessary.


\subsection{Case (2)}\label{subsection-Case(2)}

Next, we will consider hyperbolic pretzel knots of the form $K_m=P(3,3,2m+1)$.  Here, $K_m$ is hyperbolic if $m\not=-1$ or 0, and if $m$ is positive, then Ichihara-Jong have shown that $K_m$ admits no Seifert fibered surgeries \cite{ichi-jong_pqq}.  The case when $m=-2$ will be covered in Subsection \ref{subsec:3}, so assume $m\leq -3$.  In this section, we prove the following.

\begin{proposition}\label{prop:(3,3,odd)}
A pretzel knot of the form $P(3,3,2m+1)$ with $m\leq -3$ admits no small Seifert fibered surgeries.
\end{proposition}

As before, we will first restrict the possible values of $m$ for which $K_m$ might admit a Seifert fibered surgery, then rule the remaining cases out by computer.  First we note that $K_m$ has genus one, so by Theorem \ref{thm:BGZgenusone}, $|r|\leq 3$ (see Figure \ref{figure:P(3,3,odd)SurfaceAndPeriod}).  Since the three pretzel parameters are all odd, $K_m$ cannot admit non-integral Seifert fibered surgeries by Theorem \ref{thm:non-integral}.  Let $\alpha_r =- r$, so that $K_m(\alpha_r)=K_m(-r)$ will correspond with $\mathcal T_m(r)$ (the rationally-filled quotient tangle), as before.

\begin{figure}
\centering
\includegraphics[scale = .55]{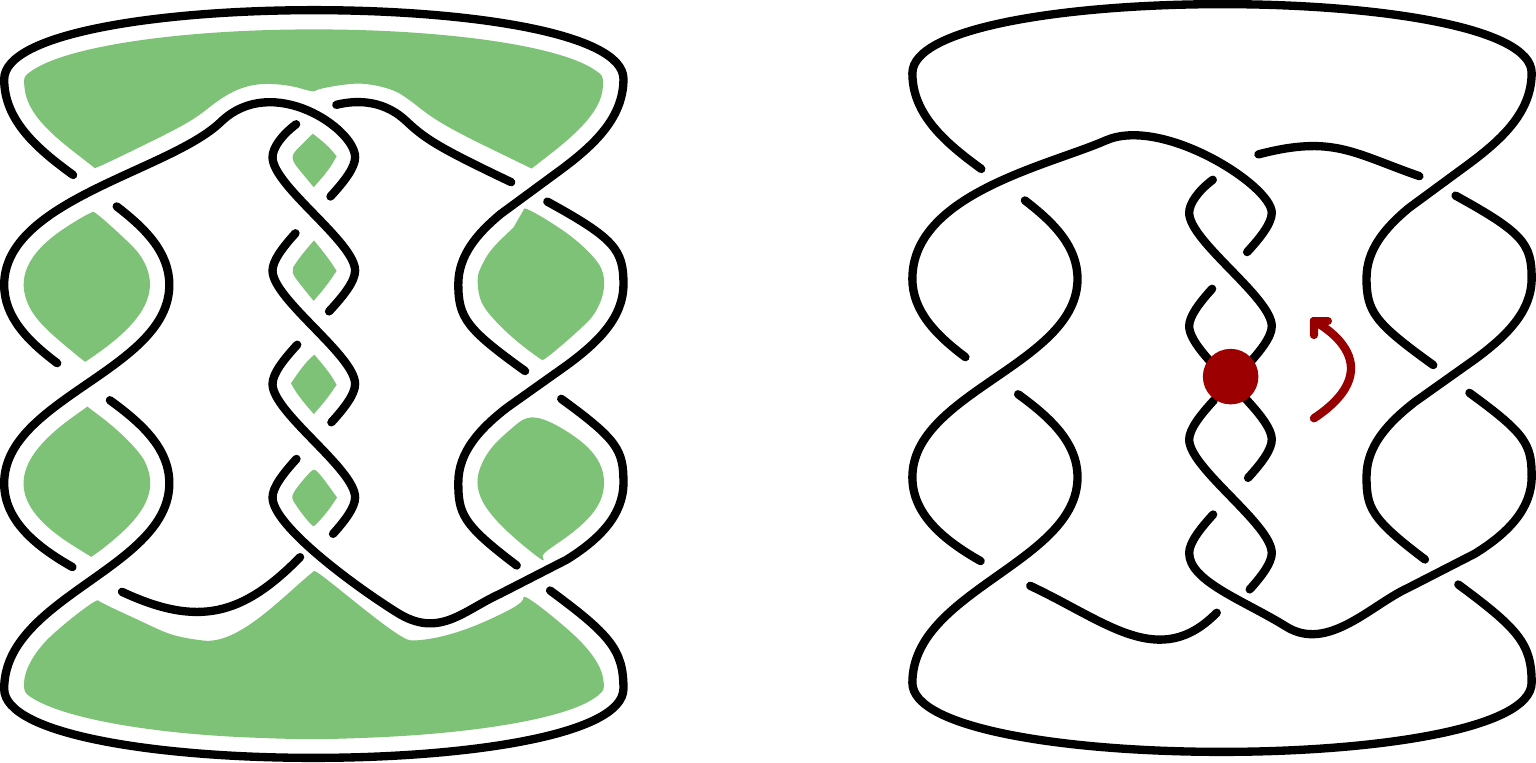}
\caption{The left figure above shows the pretzel knot $P(3,3,2m+1)$ as the boundary of a punctured torus, while the right one exhibits the strong inversion of $P(3,3,2m+1)$ given by rotation.  Here, $m=-3$.}
\label{figure:P(3,3,odd)SurfaceAndPeriod}
\end{figure}

\begin{lemma}
If $m\leq -10$, then $K_m(r)$ is not a small Seifert fibered space.
\end{lemma}

\begin{proof}
In a slight variation of the preceding cases, these knots possess a strong inversion that is a half-rotation of the plane.  Let $\mathcal T_m(r)$ be the resulting quotient link, as before (see Figure \ref{figure:P(3,3,odd)RotationalPeriod}).  Again, we form the tangle $\mathcal S_r$ by removing a ball containing the $(m+1)$-twist region (see Figure \ref{figure:P(3,3,odd)Fillings}).  If we denote by $N_r$ the branched double cover of $S^3$ along $\mathcal S_r$, and assume for a contradiction that $K_m(r)$ is a small Seifert fibered space for some $m\leq -10$ and some $r$, then we have the following fillings of $N_r$ (see Figure \ref{figure:P(3,3,odd)Fillings}).  Note that $N_r(\infty) = P(3,3,-1)(r)$, so it is simply $r$-surgery on the left-handed trefoil.

\begin{figure}
\centering
\includegraphics[scale = .5]{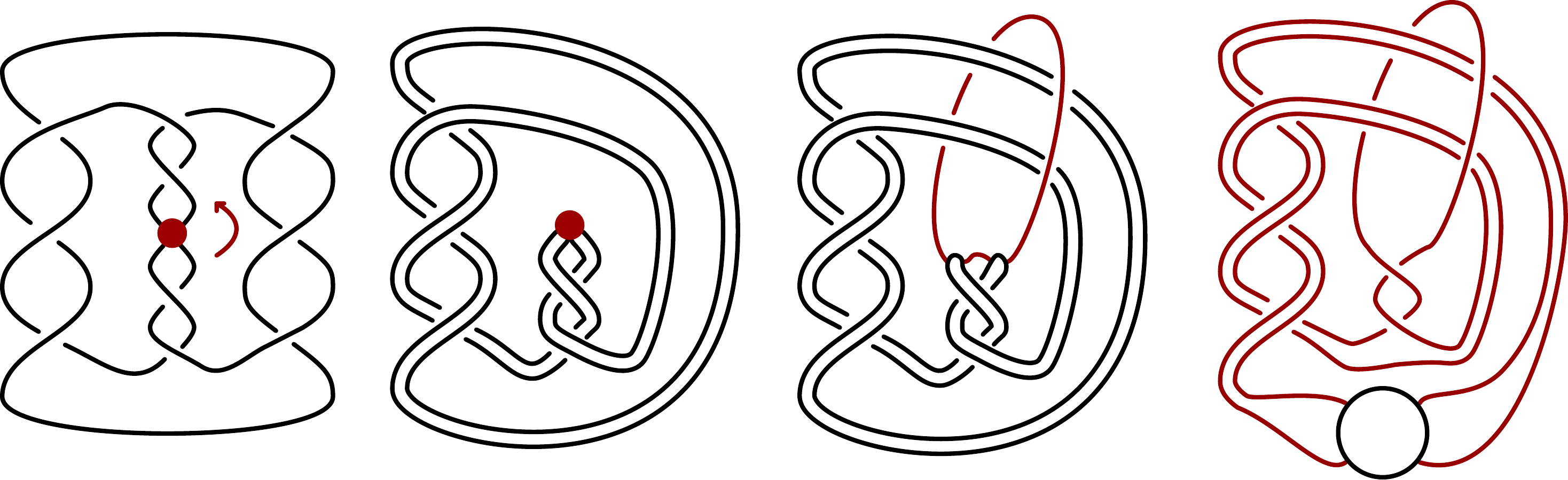}
\put(-59.5,9){\small$r$+6}
\caption{The figures above illustrate how to obtain the quotient tangle $\mathcal T_{m}$ by applying the Montesinos trick to the strong inversion of $P(3,3,2m+1)$ given by rotating the knot $\pi$ radians through its center. Here, $m=-3$.}
\label{figure:P(3,3,odd)RotationalPeriod}
\end{figure}

$$\begin{array}{rcl} 
N_r(-1/(m+1)) & =& \text{small Seifert fibered space (by assumption)}\\
N_r(0) & = & \text{non-Seifert fibered toroidal space}\\
N_r(\infty)  & =& S^2(-1/2,1/3, -1/(r+6))
\end{array}$$

\begin{figure}
\centering
\includegraphics[scale = .5]{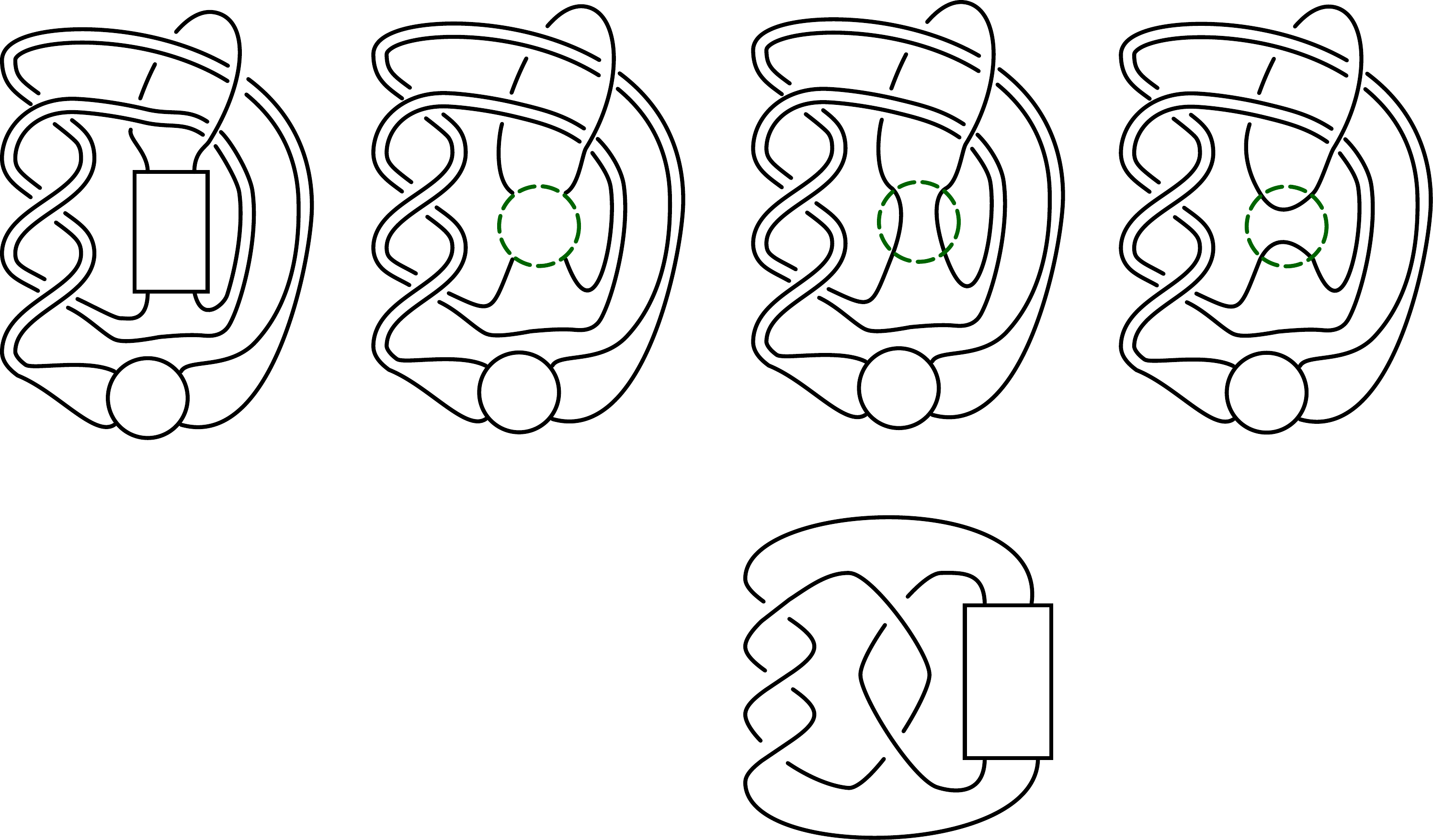}
\put(-390,182){\small$m$+1}
\put(-60,132){\small$r$+6}
\put(-169.5,132.5){\small$r$+6}
\put(-285,132){\small$r$+6}
\put(-396,131.5){\small$r$+6}
\put(-57,110){\small (d)}
\put(-169.5,-13){\small(c)}
\put(-282,110){\small (b)}
\put(-392,110){\small (a)}
\put(-137,45){\small -6-$r$}
\put(-165,105){\huge$\wr$}
\caption{The first figures above, from left to right, are: (a) the link $\mathcal T_m(-\alpha_r)$, (b) the tangle $\mathcal S_r$, (c) the filling $\mathcal S_r(\infty)$, which is isotopic to $K[-1/2, 1/3, -1/(6+r)]$, and (d) the fillings $\mathcal S_r(0)$, whose link complement contains an essential torus.}
\label{figure:P(3,3,odd)Fillings}
\end{figure}
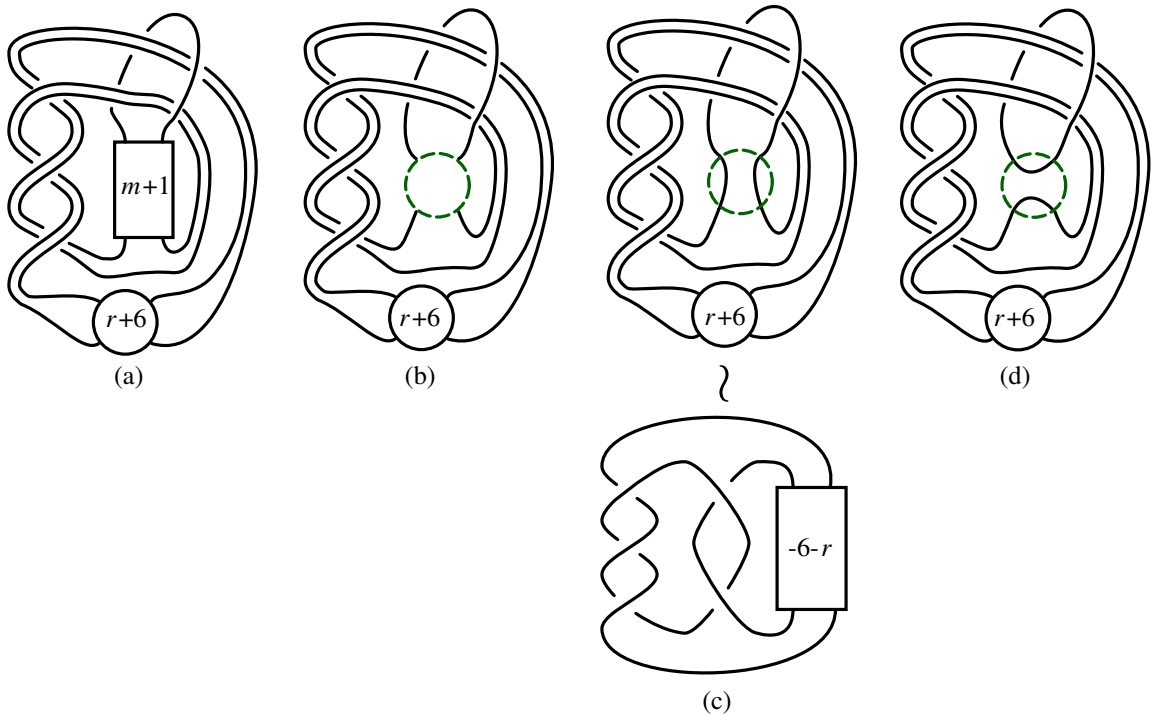

Since $N_r$ has distinct irreducible fillings and a non-Seifert fibered irreducible filling, $N_r$ is irreducible, non-Seifert fibered, and $\partial$-irreducible.  If $N_r$ is toroidal, then since it has atoroidal fillings at distance $\Delta(-1/(m+1), \infty) = |m+1|>2$, it has as a subspace a cable space with cabling slope $\gamma = 0$.  This means that $N_r(0)$ is either reducible or a lens space, which is a contradiction.  It follows that $N_r$ is hyperbolic and that $\Delta(-1/(m+1), \infty) = |m+1|\leq 8$, a contradiction that completes the proof.

\end{proof}

\subsection{Completing the proof of Proposition \ref{prop:(3,3,odd)}}

Let $L_{m,r} = \mathcal T_m(r)$ denote the quotient link described above, and note that it is only necessary to consider $r\in\Z$ here. If $L_{m,r}$ is a link, then we see that it is the union of a trefoil with the knot $J_m=K[1/2,1/3, (m-1)/(2m-3)]$.  Since $m\leq -3$, by assumption, $J_m$ is never a torus knot or a 2-bridge knot.  It follows that $L_{m,r}$ is never a Seifert link or a Montesinos link, by Criteria \ref{crit:torus} and \ref{crit:2bridge}, respectively.

When $L_{m,r}$ is a knot, we see that $|\text{Kh}(L_{m,r})|=6$ and $s(L_{m,r})<\text{br}(\Delta_{K_{m,r}}(t))$, so $L_{m,r}$ cannot be a Montesinos knot or a torus knot by Criteria \ref{crit:width} and \ref{crit:positive}, respectively.

\subsection{Case (3)}\label{subsec:3}

We now consider hyperbolic pretzel knots $K_m=P(3,-3,2m+1)$.  We will allow $m$ to be positive or negative, which will allow us to restrict our analysis to positive surgery slopes (which must be integral if they are to be exceptional by Theorem \ref{thm:non-integral}).   These are genus one knots, so $K_m(r)$ can only be exceptional if $|r|\leq 3$ by Theorem \ref{thm:BGZgenusone}.  Thus, we assume $r=1,2,$ or 3.

\begin{figure}
\centering
\includegraphics[scale = .5]{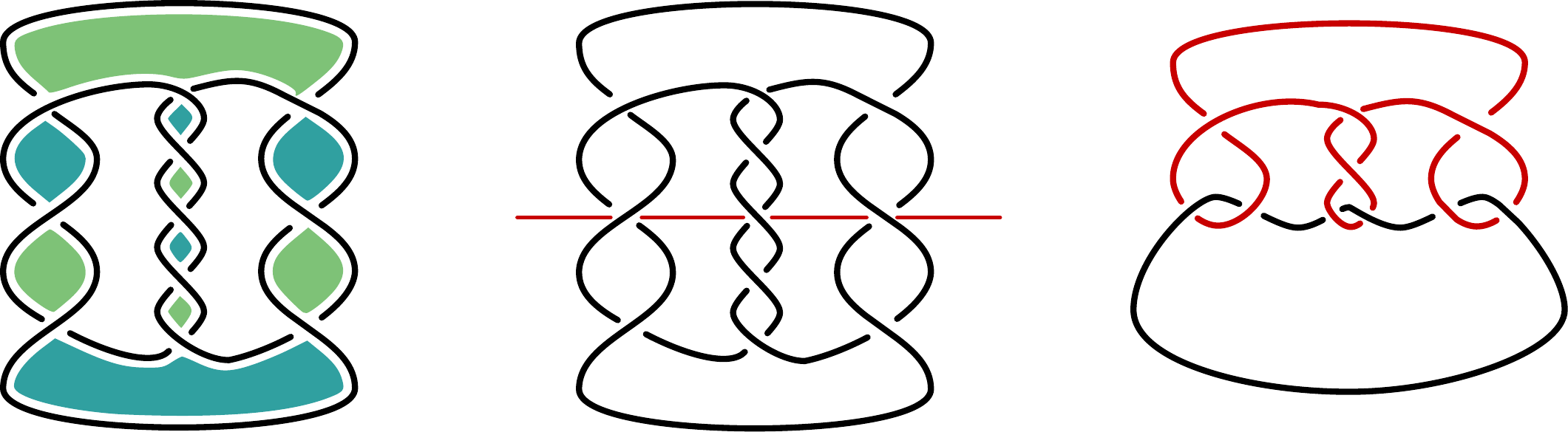}
\caption{(a) The knot $P(3,-3,2m+1)$ (here, $m=-3$), shown as the boundary of a punctured torus, (b) along with the axis of the cycle $f$ of period 2 of the knot, and (c) the quotient knot $K_f$ (here, the unknot), along with the image of Fix($f$) in the quotient.}
\label{figure:P(3,-3,2n+1)SurfaceAndInversion}
\end{figure}

These knots are the first that we have encountered with no strong inversion (excepting $P(3,3,-3))$, so we cannot make use of the Montesinos trick.  However, $K_m$ does have cyclic period 2, so we will study the space $K_m(r)$ by studying the link $(L_m)_f$, which is the image of $\text{Fix}(f)/\langle f\rangle$ in $(K_m)_f(r/2)$ (recall this set-up from Section \ref{subsection:symmetries}).  See Figure \ref{figure:P(3,-3,2n+1)SurfaceAndInversion}.

\begin{proposition}\label{prop:(3,-3,odd)}
A hyperbolic pretzel knot of the form $P(3,-3, 2m+1)$ admits a small Seifert fibered surgery precisely in the following instances.
\begin{itemize}
\item $P(3,-3,3)(1) = S^2(1/2,-1/5,-2/7)$
\item $P(3,-3,5)(1) = S^2(-1/3, -1/4, 3/5)$
\end{itemize}
\end{proposition}

Note that the first exceptional surgery was discovered by Song, and the second by Mattman, Miyazaki, and Motegi, see \cite{MMM}.  Again, our first task is to restrict the possible values of $m$ for which $K_m$ might admit a Seifert fibered surgery, then rule out the remaining cases using knot invariants.  We will handle the three cases $r=1, 2,$ and 3 separately below.

\begin{lemma}
The space $P(3,-3,2m+1)(1)$ is not a small Seifert fibered space for $|m|\geq 9$.
\end{lemma}

\begin{proof} 

Assume that $P(3,-3,2m+1)(1)$ is a small Seifert fibered space with $|m|\geq 9$.  As we have seen $K_m(1)$ is the branched double cover of $S^3$ along the knot $(L_m)_f$.  We form the tangle $\mathcal S$ by removing the $m$-twist box of $(L_m)_f$ (see Figure \ref{figure:P(3,-3,2m+1)KfLfRr=1}).  Let $Z = \widetilde{\mathcal S}$.  Then, we have the following fillings:

\begin{figure}
\centering
\includegraphics[scale = .5]{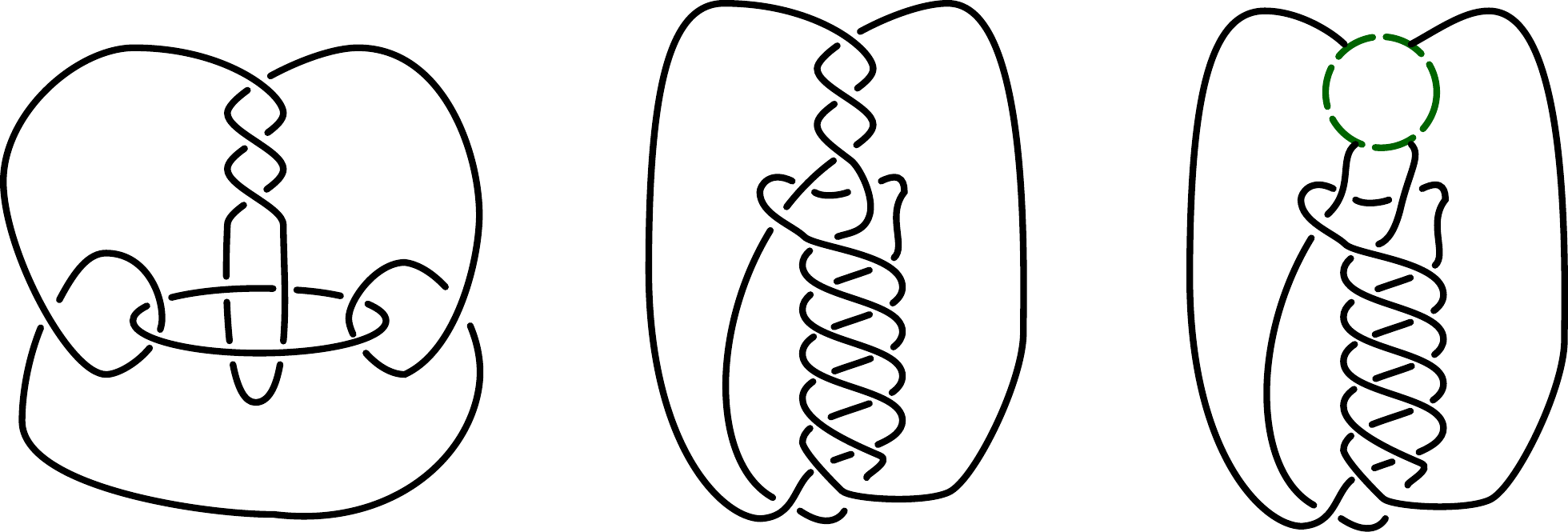}
\put(-40,-15){(c)}
\put(-140,-15){(b)}
\put(-245,-15){(a)}
\caption{(a) The link $(K_m)_f\cup\text{Fix}(f$) in the quotient, (b) the knot $(L_m)_f$ resulting from $(1/2)$-surgery on the unknotted component, and (c) the tangle $\mathcal S$ formed by removing the $m$-twist area of the knot.  Here, $m=-3$.}
\label{figure:P(3,-3,2m+1)KfLfRr=1}
\end{figure}

$$\begin{array}{l} 
Z(-1/m) = \text{small Seifert fibered (by assumption)}\\
Z(0) = S^2\times S^1 \\
Z(-1/3)  = D^2(2,3)\cup_{T^2}D^2(2,3) \\
Z(-1/2) = S^2(3,5,7) \\
Z(-1) = S^2(2, 5, 7) \\
Z(\infty) = S^2(2,3,11)
\end{array}$$

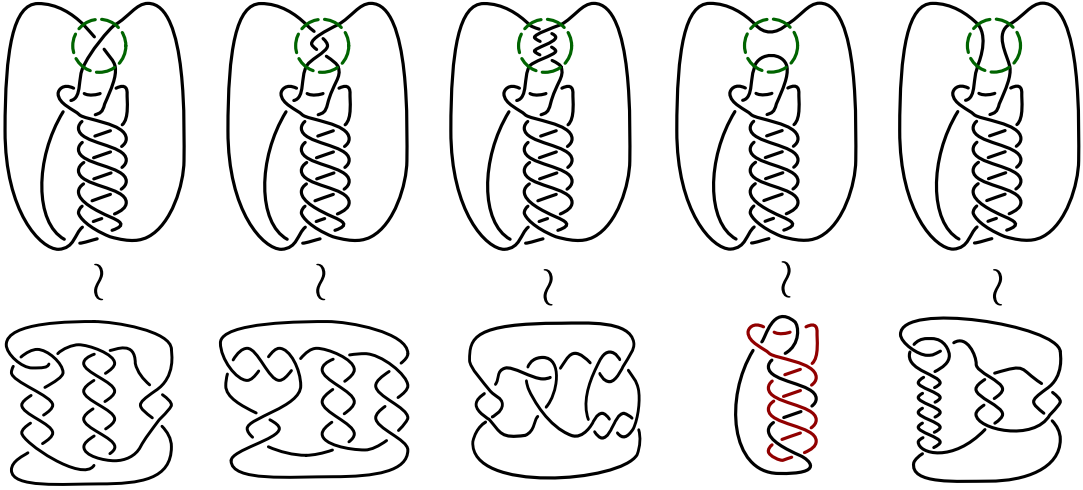
\begin{figure}
\centering
\includegraphics[scale = .5]{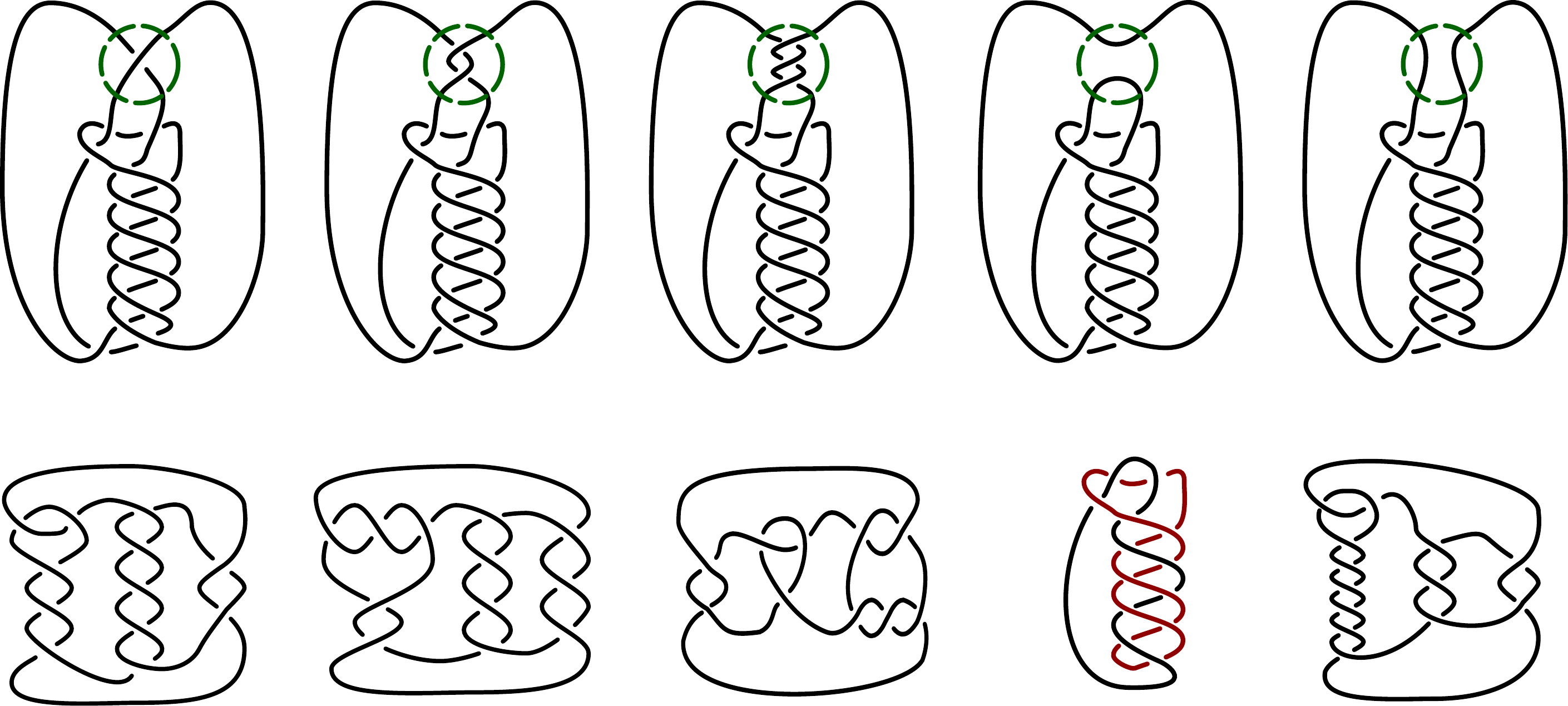}
\put(-115,73){\huge$\wr$}
\put(-205,70){\huge$\wr$}
\put(-291,72){\huge$\wr$}
\put(-375,72){\huge$\wr$}
\put(-35,70){\huge$\wr$}
\caption{Above, we see five interesting fillings of $\mathcal S$: $\mathcal S(1),\mathcal S(1/2),\mathcal S(1/3),\mathcal S(0)$, and $\mathcal S(\infty)$.}
\label{figure:P(3,-3,2m+1)(1)Fillings}
\end{figure}

It is worth noting that the last four fillings on the list correspond to exceptional fillings of hyperbolic pretzel knots. $P(3,-3,7)(1) = Z(-1/3)$, $P(3,-3,5)(1) = Z(-1/2)$, $P(3,-3,3)(1) = Z(-1)$, and $P(3,-3,1)(1) = Z(\infty)$.  The lattermost is surgery on the rational knot $\mathcal K[-2/9]$.  See Figure \ref{figure:P(3,-3,2m+1)(1)Fillings}.

Because $Z$ has distinct irreducible fillings as well as an irreducible non-Seifert fibered filling ($Z(-1/3)$ is a non-Seifert fibered graph manifold), it is impossible for $Z$ to be Seifert fibered, reducible, or $\partial$-reducible.  Assume that $Z$ is toroidal, so $Z=A\cup_FB$ with $A$ atoroidal and $\partial Z\subset B$.  Then, since $F$ compresses in $Z(-1/2)$, $Z(-1)$, and $Z(\infty)$, and $\Delta(1/2, \infty) \geq 2$, $B$ must be a cable space.  The cabling slope $\gamma$ is restricted to be distance one from $\infty$ and $-1/2$, so $\gamma=0$ or $-1$.  Since $Z(-1)$ is neither a lens space nor reducible, we cannot have $\gamma=-1$.  If $\gamma=0$, then $B(0) = (S^1\times D^2)\#(S^1\times S^2)$, which is not a possible result of filling on a cable space.

It follows that $Z$ is hyperbolic, so $\Delta(-1/m),\infty)\leq 8$, so $|m|\leq 8$, which gives the desired contradiction.

\end{proof}

\begin{proposition}
$P(3,-3,2m+1)(2)$ is never a small Seifert fibered space for $m\not=0,-1$.
\end{proposition}

\begin{proof}
We can precede as above, by analyzing the result of $1$-surgery on the quotient knot $(K_m)_f$.  This gives a two-component link in $S^3$, $(L_m)_f$, such that the double cover of $S^3$ branched along $(L_m)_f$ is the surgery space $P(3,-3,2m+1)(2)$.  Thus, to show this surgery space is not a small Seifert fibered space, it suffices to show that $(L_m)_f$ is neither a Montesinos link of length three with two components nor a Seifert link with two components.

\begin{figure}
\centering
\includegraphics[scale = .55]{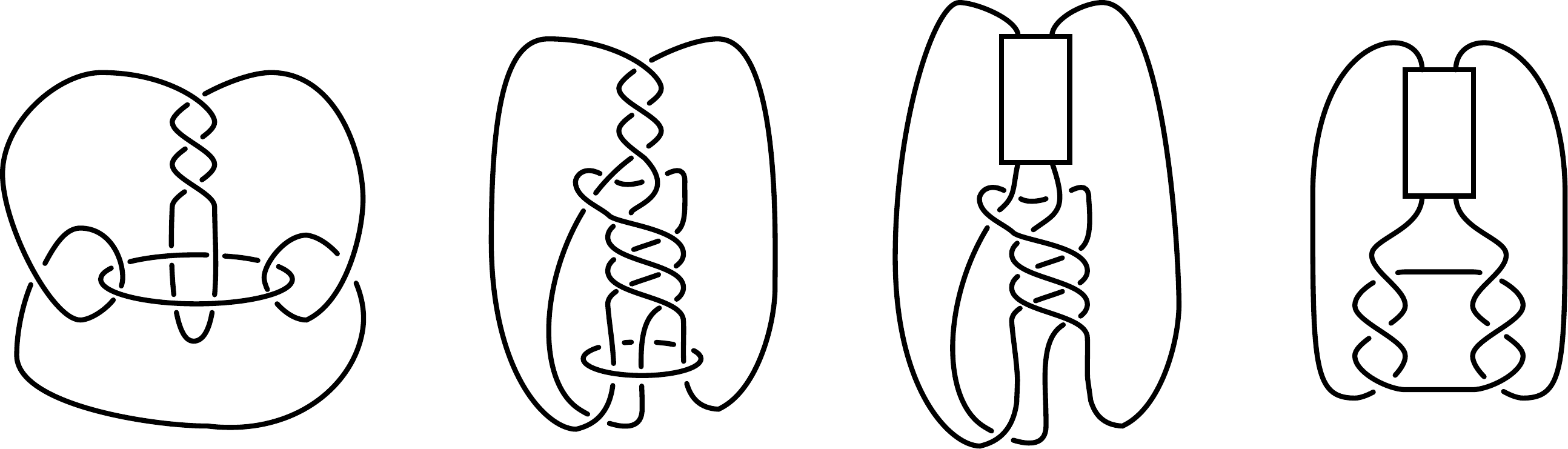}
\put(-80,45){\huge$\widetilde{}$}
\put(-41,80){\scriptsize$m$-1}
\put(-145,90){$m$}
\put(-100,-15){(c)}
\put(-250,-15){(b)}
\put(-370,-15){(a)}
\caption{(a) The link $(K_m)_f\cup\text{Fix}(f$) in the quotient, (b) the link $(L_m)_f$ resulting from 1-surgery on the unknotted component (note that the core of the surgery torus is a component of the resulting link), and (c) the tangle formed by removing the $m$-twist area of the knotted component $J_m$ of the resulting link.  Here, $m=-3$.}\label{figure:P(3,-3,2m+1)KfLfr=2}
\end{figure}
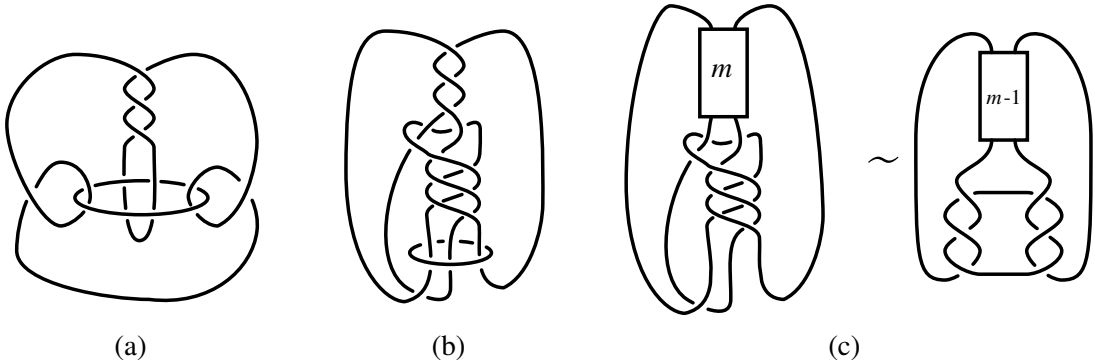

First, we note that $(L_m)_f$ is the union of the unknot with the knot $J_m=K[1/3,-1/3,1/(m-1)]$ (see the right half of Figure \ref{figure:P(3,-3,2m+1)KfLfr=2}).  Since $J_m$ is never torus knot, by Criterion \ref{crit:torus}, $(L_m)_f$ is never a Seifert link.  If $m=0$ or $m=2$, then $J_0=K[-2/9]$ and $J_2=K[2/9]$, respectively; otherwise, $J_m$ is not a 2-bridge knot, so $(L_m)_f$ is not a Montesinos link, by Criterion \ref{crit:torus}.

If $m=0$, then we are considering 2-surgery on $P(3,-3,1)$, which is $K[2/9](2)$.  By the classification of Brittenham and Wu \cite{britt-wu:2bridge}, this space is Seifert fibered.  If $m=2$, then we are considering the space $P(3,-3,5)(2)$, and if $(L_m)_f$ is a Montesinos link, then it has the form $K[x,y,z]$, where $z=2/9$ or $4/9$ and $x$ and $y$ have even denominator.

Now, since $(L_2)_f$ has a diagram with 12 crossings, and since the $z$--tangle would contribute 6 crossings, if $(L_2)_f$ were to be a two component length three Montesinos link, then the $x$-- and $y$--tangles must contribute at most 6 crossings.  Without loss of generality, we can assume $x=\pm 1/2$ and $y=\pm 1/2$ or $\pm 1/3$ or $\pm 3/4$.  However, an easy check shows that the determinant of such Montesinos links cannot be 2.

\end{proof}

Finally, we consider the case of $3$-surgery on hyperbolic pretzel knots $K_m$ of the type $P(3,-3,2m+1)$.  These knots have period 2, so we can analyze the surgery space $K_m(3)$ as the double branched cover of $(K_m)_f(3/2)$, where $(K_m)_f$ is the factor knot $K_m/f$ for the self diffeomorphism $f:S^3\to S^3$ of order two that preserves $K_m$.  In this case, $(K_m)_f$ is the unknot, so $(K_m)_f = -L(3,2)$.  Let $L_m$ denote the image of Fix$(f)/f$ in the surgery space $(K_m)_f(3/2)$, i.e., $L_m$ is the branching set for the double covering.  (Note that in our convention $p$-surgery on the unknot is the lens space $-L(p,1)$.)

\begin{proposition}
 $P(3,-3,2m+1)(3)$ is never a small Seifert fibered space for $m\not=0,-1$.
\end{proposition}

\begin{proof}

\begin{figure}
\centering
\includegraphics[scale = .65]{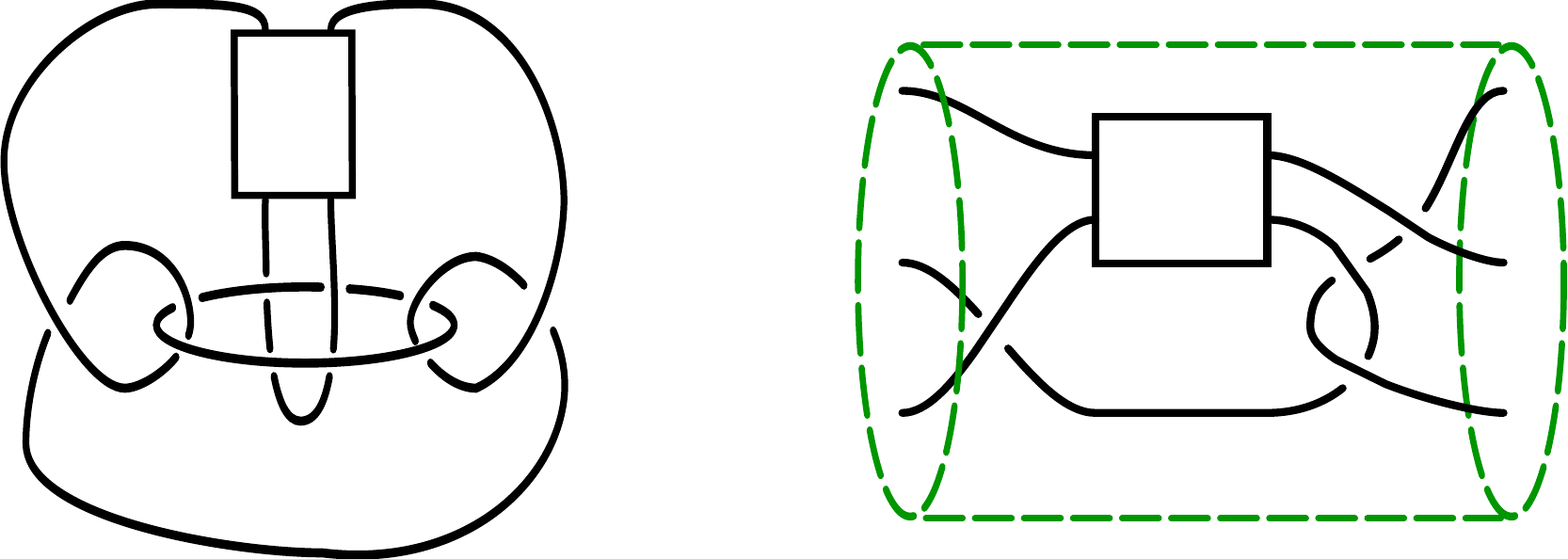}
\put(-259,85){\large $m$}
\put(-88,72){\large $-m$}
\caption{Above, on the left we have the link $(K_m)_f\cup\text{Fix}(f)/f$ in the quotient, and, on the right, we have the result of $(3/2)$-surgery on the unknotted component: the link $L_m$ contained in a solid torus (simply view the knot as lying in the solid torus that comprises the exterior of unknot).}
\label{figure:P(3,-3,2m+1)r=3}
\end{figure}

By the analysis of \cite{miyazaki-motegi_Seifert}, the link $L_m$ is actually a 2-bridge link, contained in the solid torus that, together with the surgery solid torus, comprises half of the genus-1 Heegaard splitting of the lens space $-L(3,2)$ (see Figure \ref{figure:P(3,-3,2m+1)r=3}).  Thus, if we pass to a 3-fold cover, the lift, $\widetilde{L_m}$, of $L_m$ will be a length three Montesinos link, contained in one half of the standard genus one Heegaard splitting of $S^3$.  We now describe how to see this.

The two solid tori that comprise the splitting of $-L(3,2)$ are attached via a map which sends the meridian of one to a ($-3/2)$--curve on the boundary of the other.  Passing to the 3-fold cover changes the image of the attaching map to a $(-1/2)$--curve, which gives $S^3$.  This lift simply triplicates the knotted part of $L_m$.  However, if we want to think about this lift as a knot in the standard 3--sphere, we must apply a self-diffeomorphism of $S^3$ to get the standard Heegaard splitting of $S^3$ (i.e., where the meridian of one torus is glued along a $(1/0)$--curve).  This final step introduces two full negative twists of the strands of $L_m$.  The result is a link $\widetilde L_m$ in $S^3$.  See Figure \ref{figure:P(3,-3,2m+1)TripleCover}.

\begin{figure}
\centering
\includegraphics[scale = .53]{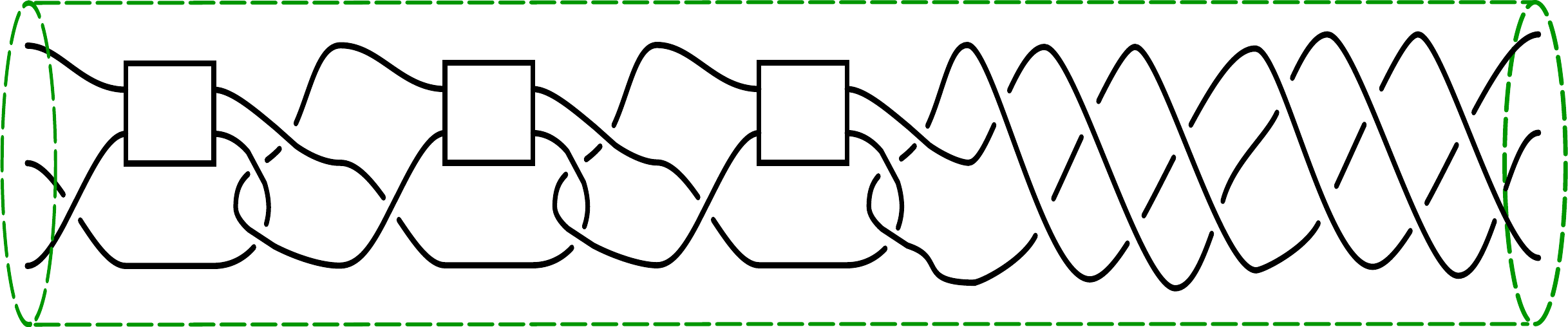}
\put(-262,50){-$m$}
\put(-338,50){-$m$}
\put(-188,50){-$m$}
\caption{The link $\widetilde{L_m}$ in $S^3$, which is the triple cover of $L_m$ in $-L(3,2)$.}
\label{figure:P(3,-3,2m+1)TripleCover}
\end{figure}

Since $P(3,-3,2m+1)(3)$ has quotient $-L(3,2)$, it has the form $S^2(3,3,c)$, where the exceptional fiber of multiplicity $c$ corresponds to the branching locus, $(L_m)_f$.  It follows that the lift $\widetilde L_m$ should be a Montesinos link of type $K[c,c,c]$.  From this, it follows that $c^2$ divides the determinant of $\widetilde L_m$.  We can calculate the determinant of this lift to be 49, independent of $m$, so it follows that $c=7$.  Thus, $P(3,-3,2m+1)(3)$ is $S^2(3,3,7)$, and $\widetilde L_m = K[a/7, b/7, b/7]$.

By considering the determinant (of the corresponding Montesinos knot), we see that $P(3,-3,2m+1)(3)$ must have the form $S^2(-1/3, -1/3, 5/7)$ and that $\widetilde L_m = K[-2;5/7,5/7,5/7]$ (being the triple cover of the two bridge knot $K[5/7]$ in $S^1\times D^2$, see \cite{miyazaki-motegi_Seifert}).  Let $V(q)$ be the Jones polynomial of $K[-2;5/7, 5/7, 5/7]$.  A straightforward calculation gives an expression for the Jones polynomial of $\widetilde L_m$: $$V_{\widetilde L_m}(q) = \begin{cases} q^{-3m-3}(V(q)-1)+1 & \text{ if $m$ is odd} \\ q^{-3m}(21-V(q^{-1}))+1  & \text{ if $m$ is even}\end{cases}$$

It follows that $\widetilde L_m$ is not $K[-2;5/7, 5/7, 5/7]$ unless $m=-1$.  In this case the knots \emph{are} the same, which reflects the fact that $P(3,-3, -1)(3) =S^2(3,3,7)$; however, this case is not of interest to us.  For any other value of $m$, we have shown that $P(3,-3, 2m+1)(3)$ cannot be a small Seifert fibered space.
\end{proof}

\subsection{Completing the proof of Proposition \ref{prop:(3,-3,odd)}}

It remains to show that the knots $L_{m} = \mathcal T_m(1)$ are neither Montesinos knots, nor Seifert knots, for $m\in[-8,8]\backslash\{-1,0,1,2\}$ (notice, when $m=1,2$ we \emph{do} get small Seifert fibered surgeries, and when $m=-1,0$, we have $P(3,-3,\pm1)$, which are 2-bridge knots).

In fact, for these knots we have that $s(L_{m})\not=2g(L_m)$, so they cannot be torus knots by Criterion \ref{crit:positive}.   Furthermore, we can apply Method 1 to show that $L_m$ is never a Montesinos knot.

\subsection{Case (4)}

Finally, consider the case when $K_m$ is a hyperbolic pretzel knot of the form $P(3,3,2m,-1)$ with $m>1$.  We note that such knots are often considered to be non-pretzel Montesinos knots.  If $m=1$, then $K_1=P(-2,3,3)$, and is not hyperbolic.   We see that $K_m$ has a cyclic of period 2 (see Figure \ref{figure:P(3,3,2m,-1)KnotAndFactorKnot}) with factor knot $K_f = T_{2,3}$ as well as a strong inversion.  Since $K_f(r/2)$ must be a lens space surgery on the trefoil, it follows that $\Delta(r/2,6) = 1$.  In this case we consider the possibility of non-integral exceptional surgeries.  We note that the link $L_{m,r}=\mathcal T_m(-r)$ has $4m+10-r$ half-twists.  See Figure \ref{figure:P(3,3,2n,-1)Lf}.  This can be seen by carefully keeping track of the framing curve throughout the Montesinos trick.

\begin{figure}
\centering
\includegraphics[scale = .45]{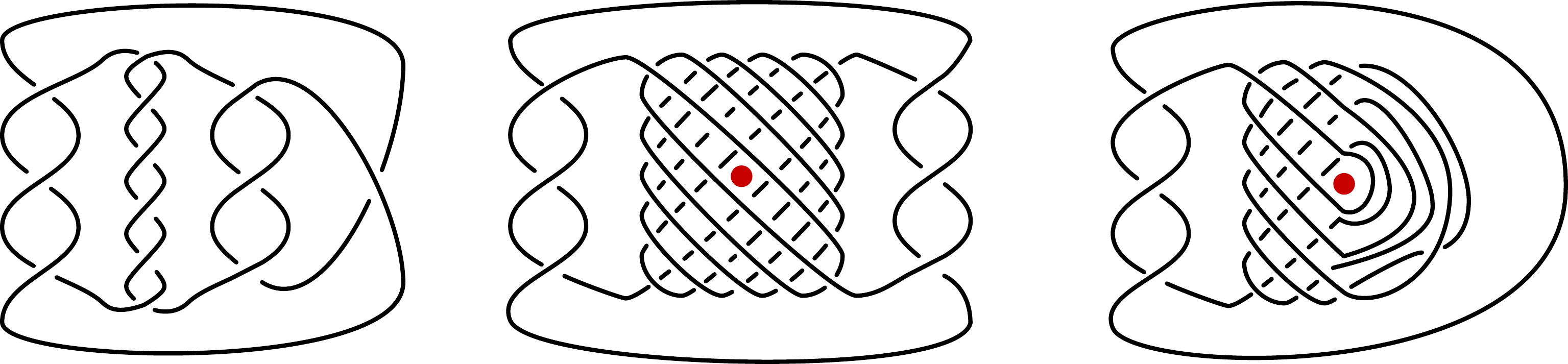}
\put(-360,-15){(a)}
\put(-220,-15){(b)}
\put(-70,-15){(c)}
\caption{The Montesinos knot $P(3,3,2m,-1)$ in (a) standard form and (b) pillowcase form, and (c) the factor knot resulting from rotation about the axis perpendicular to the page.}
\label{figure:P(3,3,2m,-1)KnotAndFactorKnot}
\end{figure}

\begin{proposition}
The Montesinos knots $P(3,3,2m,-1)$ with $m>1$ admit no small Seifert fibered surgeries.
\end{proposition}

\begin{proof}

\begin{figure}
\centering
\includegraphics[scale = .5]{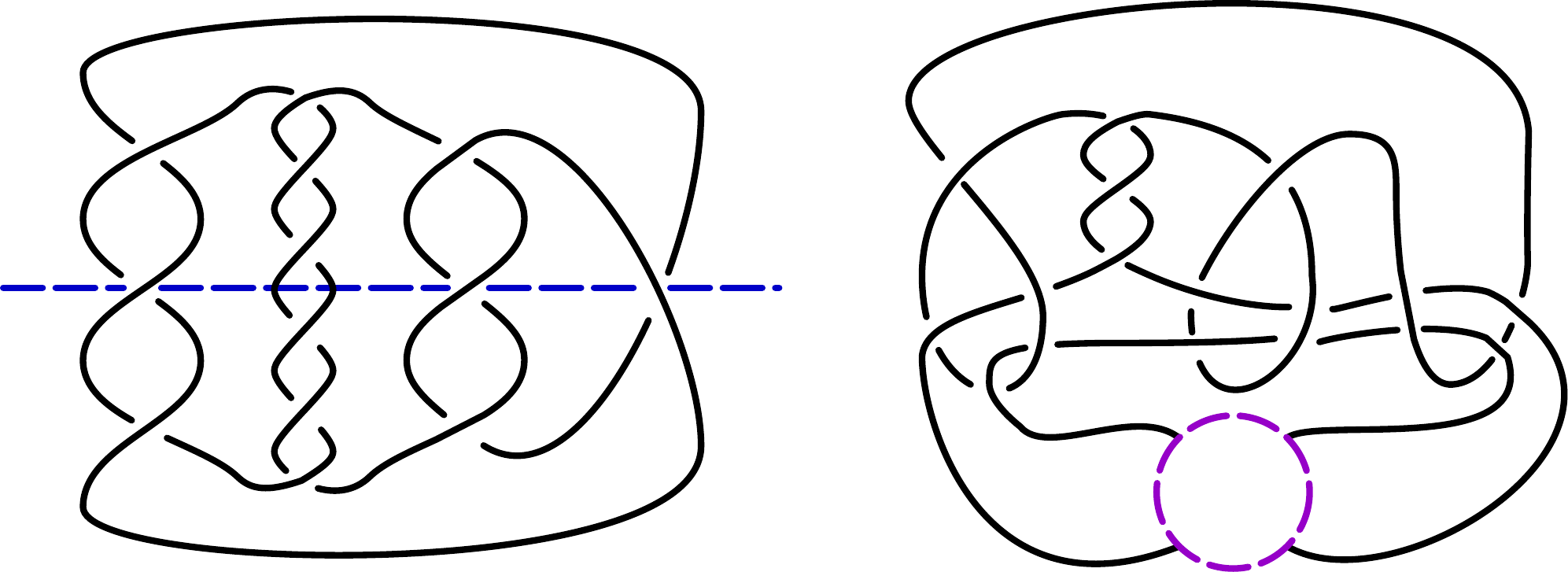}
\caption{The Montesinos knot $P(3,3,2m,-1)$, shown with the axis of its strong inversion, and the resulting tangle $\mathcal T_m$.  Here $m=3$.}
\label{figure:P(3,3,2n,-1)StrongInversion}
\end{figure}

We begin by noting that we can show that $m\leq 8$ just as we did when dealing with $K_m^+$, earlier in this section (recall, Figure \ref{figure:P(3,-3,2m)MrFillings}), so we will omit the details.  Consider the quotient links $L_{m,r}$ obtained via the Montesinos trick.  When $L_{m,r}$ is a link, it consists of an  unknotted component, together with a component $J_m$, which is the knot $K[-1/3,-1/3,1/m]$ (see Figure \ref{figure:P(3,3,2n,-1)Lf}).  

Since $J_m$ is never a torus knot or a 2-bridge knot (for $m>1$), $L_{m,r}$ is never a Seifert link or a Montesinos link, by Criteria \ref{crit:torus} and \ref{crit:2bridge}, respectively.

When $L_{m,r}$ is a knot, we see that $|\text{Kh}(L_{m,r})|\geq 4$ and $s(L_{m,r})<\text{br}(\Delta_{L_{m,r}}(t))$, so, by Criteria \ref{crit:width} and \ref{crit:positive}, $L_{m,r}$ is never a Montesinos knot or a torus knot.

\begin{figure}
\centering
\includegraphics[scale = .6]{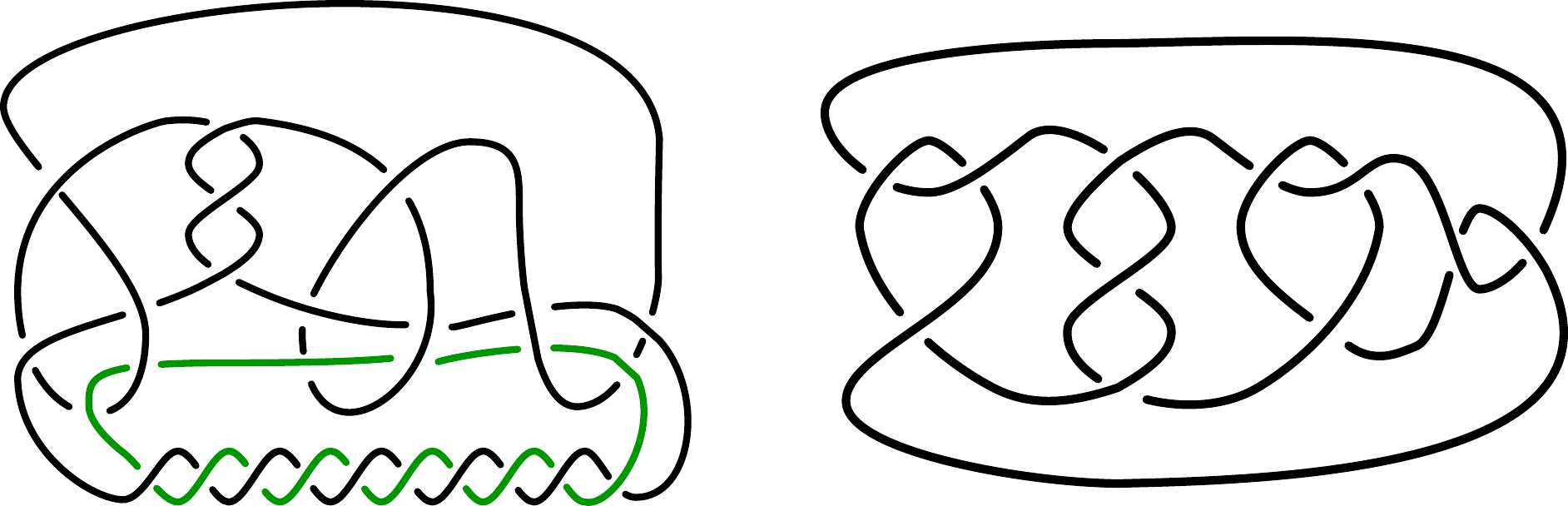}
\caption{The link $L_{m,r}$ and the component $J_m$.  Here $m=3$ and $r=12$ (hence, $4m+10-r=10$ twists).}
\label{figure:P(3,3,2n,-1)Lf}
\end{figure}

\end{proof}

\section{The case of $(3,4,5)$}\label{section:345}

We next turn our attention to the pretzel knots $P(3,\pm4,\pm5)$ and $P(3,4,5,-1)$.  We will follow the same program in which we make use of the strong inversion and analyze the quotient link along with its double branched cover.  Because these are not infinite families of knots, we do not need to argue to restrict any parameters as we have above.  In the case of the length three pretzel knots, since these knots bound punctured Klein bottles at slope $\alpha_0$, we only need to consider fillings $\alpha_r = \alpha_0-r$ at distance at most 8 from $\alpha_0$ and from $1/0$.  Our only task here is to show that the quotient links are not Seifert links nor Montesinos links.  In the case of $P(3,-4,5)$, we must consider non-integral surgeries.  The length four pretzel knot $P(3,4,5,-1)$ may also admit non-integral exceptional surgeries, and, in this case, there is no exceptional surgery by which we can bound the possible surgery slopes.  On the other hand, if $P(3,4,5,-1)(r/s)$ is exceptional, then $|s|\leq 4$, since it is known that this pretzel knot is not tunnel number one \cite{msy:tunnel}, and Baker, Gordon, and Luecke have recently shown that knots of tunnel number greater than one cannot have non-integral small Seifert fibered surgery slopes whose denominator is 5 or larger \cite{BGL}.  The pictures corresponding to the analysis of $P(3,4,5,-1)(r/s)$ are nearly identical to the diagrams in Figures \ref{figure:P(3,3,2n,-1)StrongInversion} and \ref{figure:P(3,3,2n,-1)Lf} (which corresponded to the analysis for $P(3,3,2m,-1)$ in the previous section; just let $m=2$, and change a 1/3 tangle to a 1/5 tangle) and the reader is encouraged to keep these in mind throughout this section.

\begin{reptheorem}{thm:main3}
The pretzel knots $P(3,\pm4, \pm5)$ and  $P(3,4,5,-1)$ admit no small Seifert fibered surgeries.
\end{reptheorem}

\begin{proof}

It is shown below, in Lemma \ref{lemma:khisom}, that the quotient links $L_{r/s}$ corresponding to the surgery spaces $P(3,4,5,-1)(r/s)$ have Khovanov homology of width at least 4 for $|s|\leq 4$ when $L_{r/s}$ is a knot. Thus, Criterion \ref{crit:width} suffices to prove that the knots $L_{r/s}$ are not Montesinos knots. Lemma \ref{lem:alex} shows that $L_{r/s}$ is never a torus knot.   If $L_{r/s}$ is a link, then it is the union of the unknot with the Montesinos knot $K[-2;1/2,2/3,2/5]$, which is never a 2-bridge knot or a torus knot, so $L_{r/s}$ is never a Montesinos link or a Seifert link, by Criteria \ref{crit:2bridge} and \ref{crit:torus}.  

Now, we consider the length three pretzel knots.  Let $L_{\pm,\pm,r}$ be the quotient link resulting from the Montesinos trick, applied to $P(3,\pm4,\pm5)$.  When $L_{\pm,\pm,r}$ is a link, $L_{\pm,\pm,r}=U\cup J$, where $J=K[2/3,\pm1/2,\pm2/5]$, which is never the unknot, a two-bridge knot, or a torus knot. Thus, by Criteria \ref{crit:2bridge} and \ref{crit:torus}, $L_{\pm,\pm,r}$ is never a Montesinos link with two components or a Seifert link with two components.

When $L_{\pm,\pm,r}$ is a knot, $2g(L_{\pm,\pm,r})\not=s(L_{\pm,\pm,r})$, so $L_{\pm,\pm,r}$ is not a torus knot, by Criterion \ref{crit:positive}, and we can use Method 1 to show that $L_{\pm,\pm,r}$ is never a Montesinos knot.

\end{proof}

\begin{lemma}\label{lem:alex}
$L_{r/s}$ is never a torus knot.
\end{lemma}

\begin{proof}
Recall that $|s|\leq 4$, and write $r/s = a/b+n$.  A general reference for the facts in this proof is \cite{Lickorish}.  If $|b|=2$ or $|b|=4$, then $L_{r/s}$ has unknotting number one or two.  Since the unknotting number of a $(p,q)$--torus knot is $(p-1)(q-1)/2$, we could only have the trefoil or $T(5,2)$.  However, both of these knots are alternating, so they cannot have wide Khovanov homology, as $L_{r/s}$ will be shown to have below.

If $a/b=0$ or $b=3$, we can apply the oriented Skein relation to the $n$--twist region of $L_{a/b+n}$ to calculate a recursive formula for the Alexander polynomials $\Delta_{L_{r/s}}(t)$.  In general, we write $$\Delta_K(t) = a_0+\sum_{i=1}^m a_i(t^{-i}+t^i),$$ and, applying the Skein relation to these knots, we calculate that $$\Delta_{L_{a/b+n}}(t) = k_1(-t^{-1/2})^n+k_2(t^{1/2})^n,$$ where $k_1$ and $k_2$ are fixed polynomials of small degree, depending on $a/b$ and the sign of $n$.  In any event, we see that $a_{m-1}$ for $L_{a/b+n}$ will be constant as $|n|$ increases for a fixed $a/b$.  In fact, we can calculate that $|a_{m-1}|$ takes values $3, 2, 5, 5, 7,$ and 4, respectively, for the following cases:  $a/b=0$ and $n<0$,$a/b=0$ and $n>0$, $a/b=1/3$ and $n<0$, $a/b=1/3$ and $n>0$, $a/b=-1/3$ and $n<0$, and $a/b=-1/3$ and $n>0$.

If $K$ is a torus knot, then $|a_i|\leq 1$ for all $i$.  It follows that $L_{r/s}$ is never a torus knot.
\end{proof}

\subsection{A Khovanov homology diversion}

In order to prove that no surgery on $K=P(3,4,5,-1)$ is a small Seifert fibered space, we will argue that the quotient link, $L_{r/s}$ corresponding to $K(r/s)$ has  Khovanov cohomology that is too wide when $L_{r/s}$ is a knot and $|s|\leq4$, i.e., $|\Kh(L_{r/s})|\geq4$.

We will need one important fact about Khovanov cohomology (see, for example, \cite{turner}, for an overview).  Let $D$ be a diagram for a knot, and let $D_0$ and $D_1$ be the diagrams identical to $D$, except that a single crossing has been resolved as pictured below.
$$
\begin{array}{c}
\includegraphics[scale=1]{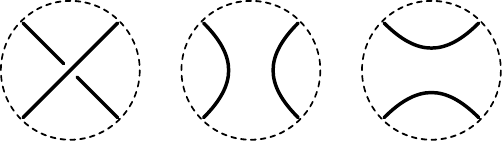} \\
\ D \hspace{.6in} D_0 \hspace{.6in} D_1
\end{array}
$$
Define the value $c$ to be
$$c=\text{ (number of negative crossings in $D_0$) -- (number of negative crossings in $D$) }.$$
Then there is a long exact sequence relating the Khovanov cohomology groups:
$$\longrightarrow \Kh^i_{j+1}(D_1)\longrightarrow \Kh^i_{j}(D)\longrightarrow \Kh^{i-c}_{j-3c-1}(D_0)\longrightarrow \Kh^{i+1}_{j+1}(D_1)\longrightarrow$$
In our examples, one of $D_0$ or $D_1$ will represent a simple knot type (unknot, Hopf link, trefoil, or (2,4)--torus link), and so the corresponding $\Kh$ will have a small range of support.  Outside of this range, there will be isomorphisms between the graded components of $\Kh(D)$ and those of $\Kh(D_1)$ or $\Kh(D_0)$.  We will make use of these isomorphism below.

\begin{lemma}\label{lemma:khisom}
Let $r/s\in\Q$ with $r/s = a/b+t$ for $a/b\in\{0,1/2,\pm1/3,\pm1/4\}$ and $t\in\Z$.  Then $|\Kh(L_{r/s})|\geq4$.
\end{lemma}

\begin{table}
\begin{tabular}{|c|c|c|c|c|}
\hline
Surgery type & $\Kh$ for fixed $r/s$ & General $\Kh$ & Diagonals ($j-2i$) & Width \\
\hline
\hline
Integral 		  & $\Kh^{0}_{1}(L_{-9})\cong\Q$   & $\Kh^{0}_{-t-8}(L_{t})\cong\Q$   &    $\{-8-t,-2-t\}$   & 4  \\ 
($t<0$)		& $\Kh^{7}_{21}(L_{-9})\cong\Q$   & $\Kh^{7}_{-t+12}(L_{t})\cong\Q$   &                            &      \\ 
\hline
Integral 		& $\Kh^0_{-19}(L_{11})\cong\Q$         & $\Kh^0_{-t-8}(L_{t})\cong\Q$           & $\{-8-t, -t\}$  & 5  \\
($t>0$)		& $\Kh^{-10}_{-31}(L_{11})\cong\Q$   & $\Kh^{-10}_{-t-20}(L_{t})\cong\Q$   &                       &   \\ 
                             \hline
Half-integral 		 & $\Kh^{-4}_{-13}(L_{1/2-5})\cong\Q$  & $\Kh^{t+1}_{2t-3}(L_{1/2+t})\cong\Q$ & $\{-5, 1\}$  & 4  \\
($t$ odd, $t<0$)	& $\Kh^{-11}_{-21}(L_{1/2-5})\cong\Q$   & $\Kh^{t-6}_{2t-11}(L_{1/2+t})\cong\Q$   &                       &   \\ 
                                        \hline
Half-integral	 	& $\Kh^{-7}_{-21}(L_{1/2-6})\cong\Q$   & $\Kh^{t-1}_{2t-9}(L_{1/2+t})\cong\Q$   &  $\{-7,-1\}$ & 4  \\ 
($t$ even, $t<0$)	& $\Kh^{-14}_{-29}(L_{1/2-6})\cong\Q$   & $\Kh^{t-8}_{2t-17}(L_{1/2+t})\cong\Q$   &                       &   \\ 
\hline
Half-integral 		 & $\Kh^{7}_{15}(L_{1/2+5})\cong\Q^2$  & $\Kh^{t+2}_{2t+1}(L_{1/2+t})\cong\Q^2$ & $\{-5, 1\}$  & 4  \\
($t$ odd, $t>0$)	& $\Kh^{14}_{23}(L_{1/2+5})\cong\Q$   & $\Kh^{t+7}_{2t+9}(L_{1/2+t})\cong\Q$   &                       &   \\ 
                                        \hline
Half-integral	 	& $\Kh^{4}_{7}(L_{1/2+6})\cong\Q$   & $\Kh^{t-2}_{2t-5}(L_{1/2+t})\cong\Q$   &  $\{-7,-1\}$ & 4  \\ 
($t$ even, $t>0$)	& $\Kh^{11}_{15}(L_{1/2+6})\cong\Q$   & $\Kh^{t+5}_{2t+3}(L_{1/2+t})\cong\Q$   &                       &   \\ 
\hline
Third-integral    & $\Kh^{-4}_{-9}(L_{1/3-8})\cong\Q$         & $\Kh^{-4}_{-t-17}(L_{1/3+t})\cong\Q$           & $\{-t-9,- t-3\}$  & 4  \\
($t<0$)		& $\Kh^{-11}_{-17}(L_{1/3-8})\cong\Q$   & $\Kh^{-11}_{-t-25}(L_{1/3+t})\cong\Q$   &                       &   \\ 
			 & $\Kh^{1}_{1}(L_{-1/3-6})\cong\Q$         & $\Kh^{1}_{-t-5}(L_{-1/3+t})\cong\Q$           & $\{-t-7,- t-1\}$  & 4  \\
			& $\Kh^{-9}_{-13}(L_{-1/3-6})\cong\Q$   & $\Kh^{-9}_{-t-19}(L_{-1/3+t})\cong\Q$   &                       &   \\ 
                             \hline
Third-integral	& $\Kh^{1}_{-15}(L_{1/3+8})\cong\Q^2$   & $\Kh^{1}_{-t-7}(L_{1/3+t})\cong\Q^2$   & $\{-t-9,-t-1 \}$	& 5   \\ 
($t>0$)		& $\Kh^{-8}_{-25}(L_{1/3+8})\cong\Q$   & $\Kh^{-8}_{-t-17}(L_{1/3+t})\cong\Q$   &                       &   \\ 
			 & $\Kh^{1}_{-13}(L_{-1/3-8})\cong\Q^2$   & $\Kh^{1}_{-t-7}(L_{-1/3+t})\cong\Q^2$           & $\{-t-9,- t+1\}$  & 5  \\
			& $\Kh^{-6}_{-19}(L_{-1/3-8})\cong\Q$   & $\Kh^{-6}_{-t-11}(L_{-1/3+t})\cong\Q$   &                       &   \\ 
\hline
Fourth-Integral  & $\Kh^{-8}_{-19}(L_{1/4-9})\cong\Q$         & $\Kh^{t+1}_{2t-1}(L_{1/4+t})\cong\Q$           & $\{-3, 3\}$  & 4  \\
($t$ odd, $t<0$) & $\Kh^{-9}_{-15}(L_{1/4-9})\cong\Q^7$   & $\Kh^{t}_{2t+3}(L_{1/4+t})\cong\Q^7$   &                       &   \\ 
                            
                             & $\Kh^{-9}_{-27}(L_{-1/4-9})\cong\Q^2$   & $\Kh^{t}_{2t-9}(L_{-1/4-t})\cong\Q^2$   & $\{-9,-3\}$	& 4  \\ 
                             & $\Kh^{-10}_{-23}(L_{-1/4-9})\cong\Q^2$   & $\Kh^{t-1}_{2t-5}(L_{-1/4-t})\cong\Q^2$   &                       &   \\ 
\hline
Fourth-Integral  & $\Kh^{-11}_{-29}(L_{1/4-8})\cong\Q$         & $\Kh^{t-3}_{2t-13}(L_{1/4+t})\cong\Q$           & $\{-7, -1\}$  & 4  \\
($t$ even, $t<0$) & $\Kh^{-12}_{-25}(L_{1/4-8})\cong\Q^7$   & $\Kh^{t-4}_{2t-9}(L_{1/4+t})\cong\Q^7$   &                       &   \\ 
                            
                             & $\Kh^{-7}_{-19}(L_{-1/4-8})\cong\Q$   & $\Kh^{t+1}_{2t-3}(L_{-1/4-t})\cong\Q$   &$\{-5, 1\}$  & 4  \\ 
                             & $\Kh^{-8}_{-15}(L_{-1/4-8})\cong\Q^2$   & $\Kh^{t}_{2t+1}(L_{-1/4-t})\cong\Q^2$   &                       &   \\ 
\hline
Fourth-Integral  & $\Kh^{8}_{19}(L_{1/4+7})\cong\Q^2$         & $\Kh^{t}_{2t+5}(L_{1/4+t})\cong\Q^2$           & $\{-1, 5\}$  & 4  \\
($t$ odd, $t>0$) & $\Kh^{9}_{15}(L_{1/4+7})\cong\Q$   & $\Kh^{t+1}_{2t+1}(L_{1/4+t})\cong\Q$   &                       &   \\ 
                            
                             & $\Kh^{6}_{9}(L_{-1/4+9})\cong\Q^2$   & $\Kh^{t-3}_{2t-9}(L_{-1/4+t})\cong\Q^2$   &$\{-3, -9\}$  & 4  \\ 
                             & $\Kh^{9}_{9}(L_{-1/4+9})\cong\Q^2$   & $\Kh^{t}_{2t-9}(L_{-1/4+t})\cong\Q^2$   &                       &   \\ 
\hline
Fourth-Integral  & $\Kh^{6}_{11}(L_{1/4+8})\cong\Q^2$         & $\Kh^{t-2}_{2t-5}(L_{1/4+t})\cong\Q^2$           & $\{-7, -1\}$  & 4  \\
($t$ even, $t>0$) & $\Kh^{7}_{7}(L_{1/4+8})\cong\Q^2$   & $\Kh^{t-1}_{2t-9}(L_{1/4+t})\cong\Q^2$   &                       &   \\ 
                            
                             & $\Kh^{7}_{15}(L_{-1/4+8})\cong\Q$   & $\Kh^{t-1}_{2t-1}(L_{-1/4+t})\cong\Q$   &$\{-5, 1\}$  & 4  \\ 
                             & $\Kh^{8}_{11}(L_{-1/4+8})\cong\Q$   & $\Kh^{t}_{2t-5}(L_{-1/4+t})\cong\Q$   &                       &   \\ 
\hline
\end{tabular}
\label{tableKh}
\caption{For each type of $a/b$ (first column), $\Kh(L_{a/b+t_0})$ can be seen to have width at least 4 (columns 2, 4, and 5).  Furthermore, due to the isomorphisms exhibited in Lemma \ref{lemma:khisom}, these graded components persist (up to consistent grading shifts) for all $|t|>|t_0|$ (column 3), which proves that $\Kh(L_{a/b+t})$ always has width at least 4 (columns 3,4, and 5).}
\end{table}

\begin{proof}
This proof will be split into cases based on the value of $a/b$.  Values of $t$ will be chosen so that $L_{a/b+t}$ is a knot, since the case of a link has already been covered above.  Throughout, keep in mind that the diagram $D$ of $L_{r/s}$ is a slight variation on the left side of Figure \ref{figure:P(3,3,2n,-1)Lf}, as mentioned before.

First, assume that $a/b=0$.  If $t>0$, then choose one of the negative crossings in the $t$-twist area of a diagram $D$ for $L_{t}$.  If $t<0$, form $D$ by creating a pair of opposite crossings next to the $t$-twist area, so that it contains a negative crossing and $t+1$ positive crossings.  In either case, the 0-resolution of the negative crossing, $D_0$, is the unknot and the 1-resolution, $D_1$, is $L_{t-1}$.  In either case, $c=-t-2$.  Repeat the process once again, using $D_1$ as $D'$.  Once again, $D_0'$ is the unknot, but now $D_1'=L_{t-2}$ and $c'=-t-1$.  Combining all of this, we have
$$\Kh^i_{j}(L_t)\cong\Kh^i_{j-2}(L_{t+2}) \hspace{.25in} \text{ if } i\not=-t-3,-t-2,-t-1.$$
Now, if we refer to Table \ref{tableKh}, the second column provides examples of graded components of $\Kh(L_{11})$ and $\Kh(L_{-9})$ that demonstrate that these knots have wide Khovanov cohomology.  But as $|t|$ increases, these graded components are preserved isomorphically (with a grading shift) in $\Kh(L_{t})$.  It follows that for large values of $|t|$, $\Kh(L_t)$ is also wide.  For small values of $|t|$ (say, $|t|\leq 11$), it is easily verified by computer that $\Kh(L_t)$ is wide.

When $a/b=1/2$, an identical argument (producing $c$-values of 1 and $-3$) gives us that 
$$\Kh^i_{j}(L_{1/2+t+2})\cong\Kh^{i-2}_{j-4}(L_{1/2+t}) \hspace{.25in} \text{ if } i\not\in[-3,3].$$
One difference here is that $D_1$ and $D_1'$ are the Hopf link, instead of the unknot, so $\Kh$ vanishes outside $i\in[-2,2]$.  With this in mind, the values in Table \ref{tableKh} give the desired width estimates.

When $a/b=\pm1/3$, the analysis is identical to that of the case when $a/b=0$, except that $D_0$ and $D_0'$ are a trefoil, both of whose $\Kh$ is supported in $i\in[-3,3]$, so we have 
$$\Kh^i_{j}(L_{1/3+t})\cong\Kh^i_{j-2}(L_{1/3+t+2}) \hspace{.25in} \text{ if } i\not\in[-t-6,-t+2],$$
and
$$\Kh^i_{j}(L_{-1/3+t})\cong\Kh^i_{j-2}(L_{-1/3+t+2}) \hspace{.25in} \text{ if } i\not\in[-t-3,-t+5].$$

When $a/b=\pm1/4$, the analysis is similar to that of the case when $a/b=\pm1/3$, except that $D_0$ and $D_0'$ are both either the $T(2,4)$ or $T(2,-4)$ torus link, both of whose $\Kh$ is supported in $i\in[-4,4]$, so we have 
$$\Kh^i_{j}(L_{\pm1/4+t})\cong\Kh^{i+2}_{j+4}(L_{\pm1/4+t+2}) \hspace{.25in} \text{ if } i\not\in[-5,5].$$

All these cases are concluded by regarding the information in Table \ref{tableKh}, as has been done above.

\end{proof}

\section{Non-pretzel Montesinos knots}\label{section:non-pretzel}

In this section, we discuss small Seifert fibered surgery on non-pretzel Montesinos knots.  By Wu  \cite{wu_plbs,wu_imm}, we need only consider a handful of cases:  $K[1/3,-2/3,2/5]$, $K[-1/2,1/3,2/(2a+1)]$ for $a\in\{3,4,5,6,\}$, and $K[-1/2,2/5,1/(2q+1)]$ for $q\geq 1$.  We prove the following theorem.

\begin{reptheorem}{thm:main4}
Suppose that $K$ is a non-pretzel Montesinos knot and $K(\alpha)$ is a small Seifert fibered space.  Then either $K=K[-1/2,2/5, 1/(2q+1)]$ for $q\geq 5$, or $K$ is on the following list and has the described surgeries.
\begin{itemize}
\item $K[1/3, -2/3,2/5](-5) = S^2(1/4,2/5,-3/5)$
\item $K[-1/2, 1/3, 2/7](-1) = S^2(1/3, 1/4, -4/7)$
\item $K[-1/2, 1/3, 2/7](0) = S^2(1/2, 3/10,-4/5)$
\item $K[-1/2, 1/3, 2/7](1) = S^2(1/2, 1/3, -16/19)$
\item $K[-1/2, 1/3, 2/9](2) = S^2(1/2, -1/3, -3/20)$
\item $K[-1/2, 1/3, 2/9](3) = S^2(1/2, -1/5, -3/11)$
\item $K[-1/2, 1/3, 2/9](4) = S^2(-1/4, 2/3, -3/8)$
\item $K[-1/2, 1/3, 2/11](-2) = S^2(-2/3, 2/5, 2/7)$
\item $K[-1/2, 1/3, 2/11](-1) = S^2(-1/2, -2/7, 2/9)$
\item $K[-1/2, 1/3, 2/5](3) = S^2(1/2, -1/3, -2/15)$
\item $K[-1/2, 1/3, 2/5](4) = S^2(1/2, -1/6, -2/7)$
\item $K[-1/2, 1/3, 2/5](5) = S^2(-1/3, -1/5, 3/5)$
\item $K[-1/2, 1/5, 2/5](7) = S^2(1/2, -1/5, -2/9)$
\item $K[-1/2, 1/5, 2/5](8) = S^2(-1/4, 3/4, -2/5)$
\item $K[-1/2, 1/7, 2/5](11) = S^2(-1/3, 3/4, -2/7)$
\end{itemize}
\end{reptheorem}

\begin{repquestion}{question1}
Do the Montesinos knots $K[-1/2, 2/5, 1/(2q+1)]$ with $q\geq 5$ admit small Seifert fibered surgeries?
\end{repquestion}

With the exception of the case noted in the question above, we now prove the list give in Theorem \ref{thm:main4} is complete.  In proving the theorem, we will consider the three types of Montesinos knots involved separately.  Note that throughout, we assume $r\in\Z$.

\subsection{The case of $K[1/3,-2/3,2/5]$}

First, consider the case when $K=K[1/3, -2/3, 2/5]$.  By Wu \cite{wu_toroidal}, $K(-4)$ and $K(-6)$ are toroidal, so it suffices to consider $K(r)$ for $-12\leq r\leq 2$, by Theorem \ref{thm:lack-meyer}.  Define $L_r$, as we have done before (see Figure \ref{figure:NonPretzel}, left side).  To show that $L_r$ is not a Montesinos knot or link, we apply Method 1.  In the case of even $r$, we note that $L_r = U\cup T(2,3)$, so if $L_r$ is a length three Montesinos link with two components, then one tangle is either a (1/3)--tangle or a (2/3)--tangle.  When such a check is performed, precisely one match is found: $K[1/3,-2/3,2/5](-5)$ is a Seifert fibered space, as shown in Theorem \ref{thm:main4}.

To see that $L_r$ is never a Seifert link, we simply note that for odd $r$, $L_r$ cannot be a torus knot, by Criterion \ref{crit:positive}, since $|s(L_r)|<\text{br}(\Delta_{L_r}(q))$.  If $r$ is even, we note that $L_r = U\cup T(2,3)$.  \emph{A priori}, $L_r$ may be a trefoil union one of its core curves; however, this link has crossing number 7, and the crossing number of $L_r$ is at least $\text{br}(V_{L_r})=10$.  Thus, $L_r$ is never a two component Seifert link.

\subsection{The case of $K[-1/2,1/3,2/(2a+1)]$}

In the case where $K_a=K[-1/2, 1/3, 2/(2a+1)]$ for $a\in\{3,4,5,6\}$, we note that by Wu \cite{wu_toroidal} we have the following toroidal fillings: $K_3(-2)$, $K_4(5)$, $K_5(2)$, $K_5(-3),$ and $K_6(2)$, so we consider surgery slopes $r$ with distance at most 8 from the corresponding toroidal filling.

We proceed as above, considering links and knots $L_{a,r}$ (see Figure \ref{figure:NonPretzel}, right side), to show that $L_{a,r}$ is never a Montesinos knot or link, noting in this case that $L_{a,r}=U\cup J$, where $J$ is $T(2,5)$ if $a=3$, $T(2,3)$ if $a=4$, and the unknot if $a=5$ or 6.  By applying Method 1, we find that the Montesinos links are precisely those stated in Theorem \ref{thm:main4}.

If $r$ is even, and $L_{a,r}$ is a Seifert link, we must have that $L_{3,r}$ is the union of $T(2,5)$, together with a core curve, $L_{4,r}$ is the union of $T(2,3)$, together with a core curve, and $L_{5,r}$ and $L_{6,r}$ have the form $T(2, 2n)$ for some $n$.  However, $L_{5,r}$ and $L_{6,r}$ cannot have this form, since they are not alternating, a fact deduced by noticing that $|\text{Kh}(L_{a,r})|\geq 3$.  As above, $L_{4,r}$ cannot have the said form because it must have at least 9 crossings.  Finally, we see that $L_{3,r}$ has linking number $\pm1$, while $T(2,5)$ has linking number $\pm5$ or $\pm2$ with its cores.  Thus, $L_{a,r}$ cannot be a two component Seifert link.

\begin{figure}
\centering
\includegraphics[scale = .75]{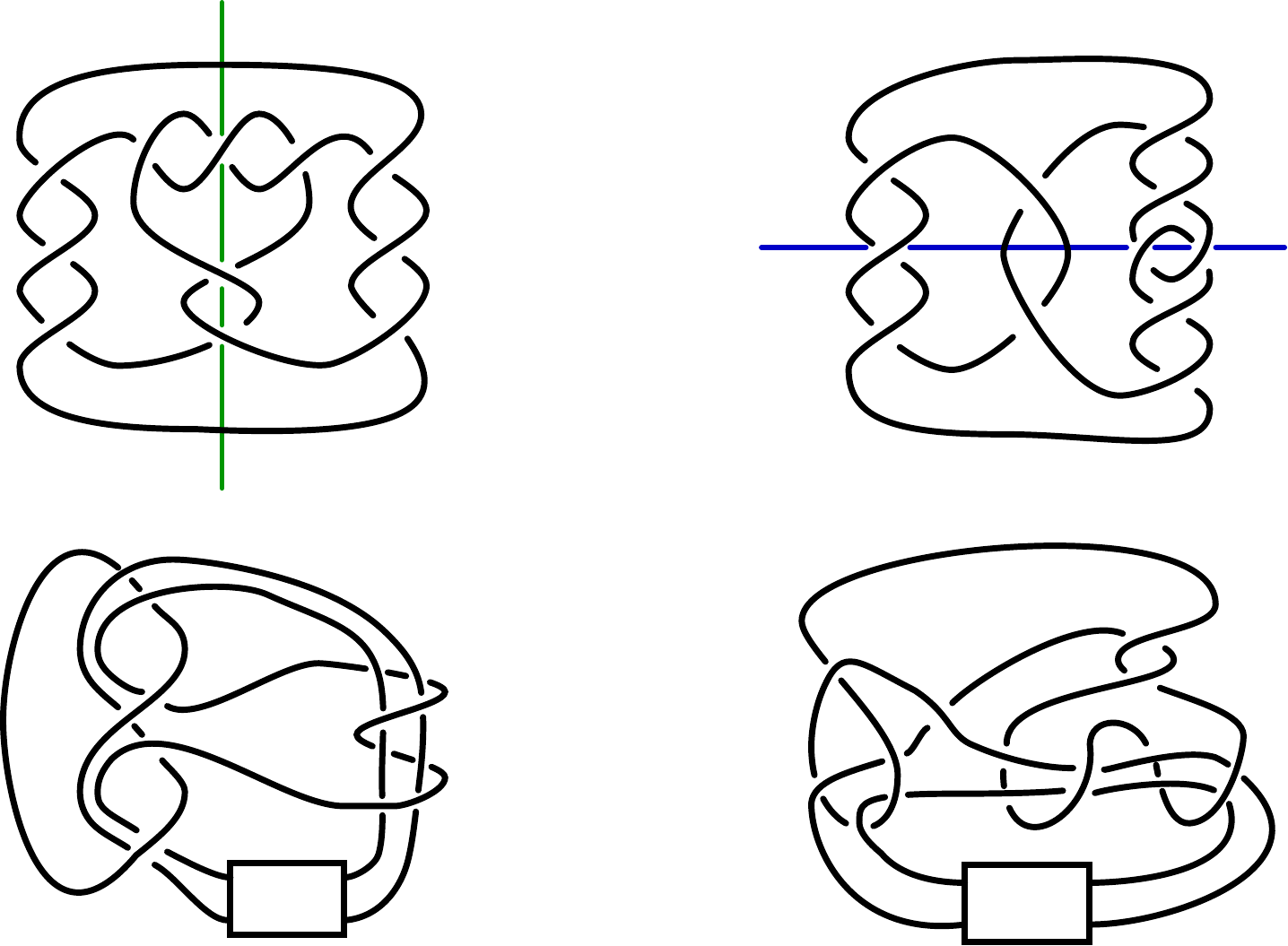}
\put(-252,6.5){-8-$r$}
\put(-71,6){6-$r$}
\caption{On the left, we have the Montesinos knot $K[1/3,1/3-3/5]$, along with the axis of one of its strong inversions, and the quotient link $L_r$.  On the right, we have the Montesinos knot $K[-1/2,1/3,2/9]$, along with the axis of its strong inversion, and the quotient link $L_{a,r}$ with $a=4$.  In the case of $K[-1/2,2/5,1/(2a+1)]$, we have a picture similar to that on the right.}
\label{figure:NonPretzel}
\end{figure}

If $r$ is odd, we can obstruct most of the $L_{a,r}$ from being torus knots by using Criterion \ref{crit:positive}.  However, $L_{3,1}$ and $L_{4,1}$ have equal Rasmussen invariant and breadth of Alexander polynomial.  The former, in fact, corresponds to a Seifert fibered space, so consider $L_{4,1}$.  Since this knot has determinant 1, if it is to be a torus knot of the form $T(a,b)$, both $a$ and $b$ are odd.   Furthermore, since it has Rasmussen invariant equal to 10, we have $10 = (a-1)(b-1)$, which is impossible if $a$ and $b$ are both odd.  Thus, $L_{a,r}$ is never a torus knot.

\subsection{The case of $K[-1/2,2/5,1/(2q+1)]$}

Finally, we are left with the case when $K_q=K[-1/2, 2/5, 1/(2q+1)]$ and $q\geq 1$.  By Wu \cite{wu_toroidal}, we have the following toroidal fillings:  $K_1(6), K_2(9), K_3(12), K_4(15)$ so for $q\leq 4$, we can bound the surgery slope $r$, and complete the classification.  For larger $q$, we cannot.  Furthermore, we obtain no bound on $q$, as we have done above.  Below, we argue that when $q\leq4$, our classification is complete.

If we again consider $L_{q,r}$, we see that for $q\leq 4$ and odd $r$, many $L_{q,r}$ have $s(L_{q,r})<\text{br}(\Delta_{L_{q,r}}(t))$ or $\det(L_{q,r})>s(L_{q,r})+1$, so they cannot be torus knots, by Criteria \ref{crit:positive} and \ref{crit:det}, respectively.  However, $\det(L_{1,r})=s(L_{1,r})+1$ for $r\in\{5,7,9,11,13\}$.  So, it is possible that $L_{1,r} = T(2,r)$.  However, $|\text{Kh}(L_{1,r})|=3$ for such $r$, so they cannot be alternating knots.

When $r$ is even, we note that $L_{q,r}=U\cup J_q$, where $J_q$ is the 2-bridge knot $K[2/(2q-9)]$.  This knot is only a torus knot (or the unknot) if $q=3,4,5,$ or 6, so $L_{q,r}$ can only be a Seifert link for these values, by Criterion \ref{crit:torus}.  However, $L_{4,r}$ is never alternating, so it cannot be $T(2,2n)$, and,  $L_{3,r}$ has at least 9 crossings, so it cannot be the union of a trefoil and a core curve.

We show $L_{q,r}$ cannot be a Montesinos knot or link by applying Method 1.  When $r$ is even, we note that one tangle would be a $(2/(2q+5))$--tangle or a $((q+3)/(2q+5))$--tangle.

%
%

\bibliographystyle{amsalpha}
\bibliography{SFSBibv2.bib}

\newcommand{\etalchar}[1]{$^{#1}$}
\providecommand{\bysame}{\leavevmode\hbox to3em{\hrulefill}\thinspace}
\providecommand{\MR}{\relax\ifhmode\unskip\space\fi MR }
\providecommand{\MRhref}[2]{%
  \href{http://www.ams.org/mathscinet-getitem?mr=#1}{#2}
}
\providecommand{\href}[2]{#2}
\begin{thebibliography}{MMM06}

\bibitem[AK10]{abe}
Tetsuya Abe and Kengo Kishimoto, \emph{The dealternating number and the
  alternation number of a closed 3-braid}, Journal of Knot Theory and its
  Ramifications \textbf{19} (2010), no.~9, 1157--1181.

\bibitem[AP04]{asaeda-przytycki}
Marta Asaeda and J\'{o}zef Przytycki, \emph{Khovanov homology: torsion and
  thickness}, Advances in topological quantum field theory, no. 179 of NATO
  Scientific Series II: Mathematics, Physics and Chemistry, Kluwer Academic
  Publishers, Dordrecht, 2004.

\bibitem[BGL]{BGL}
Kenneth Baker, Cameron~McA. Gordon, and John Luecke, \emph{Small {S}eifert
  fiber spaces from {D}ehn surgery}, In preparation.

\bibitem[BGZ11]{BGZ:genusone}
Steven Boyer, Cameron~McA. Gordon, and Xingru Zhang, \emph{Dehn fillings of
  knot manifolds containing essential once-punctured tori}, arXiv:1109.5151v3
  [math.GT] (2011).

\bibitem[BGZ12]{BGZ:toroidalSeifert}
\bysame, \emph{Characteristic submanifold theory and toroidal {D}ehn filling},
  Advances in Mathematics \textbf{230} (2012), no.~4-6, 1673--1737.

\bibitem[BM70]{burde-murasugi_links}
Gerhard Burde and Kunio Murasugi, \emph{Links and {S}eifert fiber spaces}, Duke
  Mathematics Journal \textbf{37} (1970), 89--93.

\bibitem[Bri98]{britt_exceptional}
Mark Brittenham, \emph{Exceptional {S}eifert-fibered spaces and {D}ehn surgery
  on 2-bridge knots}, Topology \textbf{37} (1998), no.~3, 665--672.

\bibitem[BW01]{britt-wu:2bridge}
Mark Brittenham and Ying-Qing Wu, \emph{The classification of exceptional
  {D}ehn surgeries on 2-bridge knots}, Communications in Analysis and Geometry
  \textbf{9} (2001), no.~1, 97--113.

\bibitem[CGLS87]{cgls}
Marc Culler, Cameron~McA. Gordon, John Luecke, and Peter Shalen, \emph{Dehn
  surgery on knots}, Annals of Mathematics (2) \textbf{125} (1987), no.~2,
  237--300.

\bibitem[Con70]{conway}
John~H. Conway, \emph{An enumeration of knots and links, and some of their
  algebraic properties}, Computational Problems in Abstract Algebra (Pergamon,
  Oxford) (J.~Leech, ed.), vol. Proc. Conf., Oxford, 1967, 1970, pp.~329--358.

\bibitem[Cro04]{cromwell}
Peter~R. Cromwell, \emph{Knots and links}, Cambridge University Press,
  Cambridge, 2004.

\bibitem[Deh10]{dehn}
Max Dehn, \emph{{\"U}ber die {T}opologie des dreidimensionalen {R}aumes},
  Mathematische Annalen \textbf{69} (1910), no.~1, 137--168.

\bibitem[Del95]{delman_esslam}
Charles Delman, \emph{Essential laminations and {D}ehn surgery on 2-bridge
  knots}, Topology and its Applications \textbf{63} (1995), no.~3, 201--221.

\bibitem[EM97]{eudave-munoz}
Mario Eudave-Mu{\~n}oz, \emph{Non-hyperbolic manifolds obtained by {D}ehn
  surgery on hyperbolic knots}, Geometric Topology, Studies in Advanced
  Mathematics (William~H. Kazez, ed.), vol. 2.1, AMS and International Press,
  1997, pp.~35--61.

\bibitem[EN85]{eisenbud-neumann_links}
David Eisenbud and Walter Neumann, \emph{Three-dimensional link theory and
  invariants of plane curve singularities}, Annals of Mathematics Studies, 110,
  Princeton University Press, Princeton, NJ, 1985.

\bibitem[FIK{\etalchar{+}}09]{futer-etal_finite}
David Futer, Masaharu Ishikawa, Yuichi Kabaya, Thomas Mattman, and Koya
  Shimokawa, \emph{Finite surgeries on three-tangle pretzel knots}, Algebraic
  \& Geometric Topology \textbf{9} (2009), no.~2, 743--771.

\bibitem[Gab89]{gabai}
David Gabai, \emph{Surgery on knots in solid tori}, Topology \textbf{28}
  (1989), no.~1, 1--6.

\bibitem[GL89]{gordon-luecke}
Cameron~McA. Gordon and John Luecke, \emph{Knots are determined by their
  complements}, Jounal of the American Mathematical Society \textbf{2} (1989),
  no.~2, 371--415.

\bibitem[Gor99]{gordon:small}
Cameron~McA. Gordon, \emph{Small surfaces and {D}ehn filling}, Proceedings of
  the Kirbyfest (Berkeley, CA 1998) (Coventry), vol.~2, Geometry and Topology
  Monographs, Geometry and Topology Publishing, 1999, pp.~177--199.

\bibitem[Gor09]{gordon_survey}
\bysame, \emph{Dehn surgery and 3-manifolds}, Low dimensional topology
  \textbf{IAS/Park City Mathematics Series, 15} (2009), 21--71.

\bibitem[Hei74]{Heil}
Wolfgang Heil, \emph{Elementary surgery on {S}eifert fiber spaces}, Yokohama
  Math Journal \textbf{22} (1974), 135--139.

\bibitem[IJ09]{ichi-jong_finite}
Kazuhiro Ichihara and In~Dae Jong, \emph{Cyclic and finite surgeries on
  {M}ontesinos knots}, Algebraic \& Geometric Topology \textbf{9} (2009),
  no.~2, 731--742.

\bibitem[IJ11]{ichi-jong_pqq}
\bysame, \emph{Seifert fibered surgery and {R}asmussen invariant},
  arXiv:1102.1117v1 [math.GT], 2011.

\bibitem[IJK11]{ichi-jong-kabaya_2pp}
Kazuhiro Ichihara, In~Dae Jong, and Yuichi Kabaya, \emph{Exceptional surgeries
  on $(-2,p,p)$-pretzel knots}, arXiv:1102.1118v1 [math.GT], 2011.

\bibitem[Kan10]{kang}
Sungmo Kang, \emph{Examples of reducible and finite {D}ehn fillings}, Journal
  of Knot Theory and its Ramifications \textbf{19} (2010), no.~5, 677--694.

\bibitem[Lee11]{lee:lens}
Sangyop Lee, \emph{Lens spaces and toroidal {D}ehn fillings}, Mathematische
  Zeitschrift \textbf{267} (2011), no.~3-4, 781--802.

\bibitem[Lic97]{Lickorish}
W.~B.~Raymond Lickorish, \emph{An {I}ntroduction to {K}not {T}heory}, Graduate
  Texts in Mathematics, 175, Springer-Verlag, New York, 1997.

\bibitem[LM08]{lack-meyer_max}
Marc Lackenby and Robert Meyerhoff, \emph{The maximal number of exceptional
  {D}ehn surgeries}, arXiv:0808.1176v1 [math.GT] (2008).

\bibitem[LT88]{lick-thistle}
W.~B.~Raymond Lickorish and Morwen~B. Thistlethwaite, \emph{Some links with
  nontrivial polynomials and their crossing-numbers}, Commentarii Mathematici
  Helvetici \textbf{63} (1988), no.~4, 527--539.

\bibitem[MB84]{morgan-bass_smith}
John~W. Morgan and Hyman Bass (eds.), \emph{The {S}mith {C}onjecture}, Pure and
  Applied Mathematics, 112, Academic Press, Orlando, FL, 1984.

\bibitem[MM02]{miyazaki-motegi_Seifert}
Katura Miyazaki and Kimihiko Motegi, \emph{Seifert fibering surgery on periodic
  knots}, Topology and its Applications \textbf{121} (2002), no.~1-2, 275--285.

\bibitem[MMM06]{MMM}
Thomas~W. Mattman, Katura Miyazaki, and Kimihiko Motegi, \emph{Seifert fibered
  surgeries which do not arise from primitive/{S}eifert-fibered constructions},
  Transactions of the America Mathematical Society \textbf{358} (2006), no.~9,
  4045--4055.

\bibitem[Mon75]{montesinos_surgery}
Jos{\'e}~M. Montesinos, \emph{Surgery on links and double branched covers of
  ${S}^3$}, Knots, groups, and 3-manifolds (Papers dedicated to the memory of
  R. H. Fox, vol.~84, Princeton University Press, 1975, pp.~227--259.

\bibitem[Mot03]{motegi_symm}
Kimihiko Motegi, \emph{Dehn surgeries, group actions, and {S}eifert fiber
  spaces}, Communications in Analysis and Geometry \textbf{11} (2003), no.~2,
  343--389.

\bibitem[MSY96]{msy:tunnel}
Kanji Morimoto, Makoto Sakuma, and Yoshiyuki Yokota, \emph{Identifying tunnel
  number one knots}, Journal of the Mathematical Society of Japan \textbf{48}
  (1996), no.~4, 667--688.

\bibitem[Oer84]{oertel}
Ulrich Oertel, \emph{Closed incompressible surfaces in complements of star
  links}, Pacific Journal of Mathematics \textbf{111} (1984), no.~1, 209--230.

\bibitem[Per03a]{perelman2}
Grisha Perelman, \emph{The entropy formula for {R}icci flow and its geometric
  applications}, arXiv:math.DG/0211159 (2003).

\bibitem[Per03b]{perelman3}
\bysame, \emph{Finite extinction time for the solutions to the {R}icci flow on
  certain three-manifolds}, arXiv:math.DG/0307245 (2003).

\bibitem[Per03c]{perelman1}
\bysame, \emph{Ricci flow with surgery on three-manifolds},
  arXiv:math.DG/0303109 (2003).

\bibitem[Ras10]{rasmussen}
Jacob Rasmussen, \emph{Khovanov homology and the slice genus}, Inventiones
  Mathematicae \textbf{182} (2010), no.~2, 419--447.

\bibitem[Tur06]{turner}
Paul Turner, \emph{Five lectures on {K}hovanov homology}, arXiv:math/0606464
  [math.GT], 2006.

\bibitem[Wol99]{wolfram}
Stephan Wolfram, \emph{The {M}athematica\textsuperscript{\textregistered}
  {B}ook}, fourth ed., Wolfram Media, Inc., Champaign, IL; Cambridge University
  Press, Cambridge, 1999.

\bibitem[Wu96]{wu_arbor}
Ying-Qing Wu, \emph{Dehn surgery on arborescent knots}, Journal of Differential
  Geometry \textbf{43} (1996), no.~1, 171--197.

\bibitem[Wu98a]{wu_survey}
\bysame, \emph{Dehn surgery on arborescent knots and links-a survey}, Chaos
  Solitons Fractals \textbf{9} (1998), no.~4-5, 671--679.

\bibitem[Wu98b]{wu_sutured}
\bysame, \emph{Sutured manifold hierarchies, essential laminations, and {D}ehn
  surgery}, Journal of Differential Geometry \textbf{48} (1998), no.~3,
  407--437.

\bibitem[Wu10]{wu_plbs}
\bysame, \emph{Persistently laminar branched surfaces}, arXiv:1008.2680v1
  [math.GT], 2010.

\bibitem[Wu11a]{wu_toroidal}
\bysame, \emph{The classification of toroidal {D}ehn surgeries on {M}ontesinos
  knots}, Communications in Analysis and Geometry \textbf{19} (2011), no.~2,
  305--345.

\bibitem[Wu11b]{wu_large}
\bysame, \emph{Exceptional {D}ehn surgery on large arborescent knots}, Pacific
  Journal of Mathematics \textbf{252} (2011), no.~1, 219--243.

\bibitem[Wu11c]{wu_imm}
\bysame, \emph{Immersed surfaces and {S}eifert fibered surgery on {M}ontesinos
  knots}, arXiv:0910.4882v2 [math.GT], 2011.

\bibitem[Wu12a]{wu:wrapped}
\bysame, \emph{Dehn surgery on knots of wrapping nuber 2}, arXiv:1105.4287v1,
  2012.

\bibitem[Wu12b]{wu:SFSMontKnots}
\bysame, \emph{Seifert fibered surgery on {M}ontesinos knots},
  arXiv:12007.0154v1, 2012.

\end{thebibliography}

\end{document}